\documentclass{amsart}
\usepackage{graphicx}
\usepackage{amsmath,amssymb}
\usepackage{amsthm}
\usepackage{color}
\usepackage{eufrak}
\usepackage[title]{appendix}
\usepackage{tikz}

\newtheorem{thm}{Theorem}[section]
\newtheorem{cor}[thm]{Corollary}
\newtheorem{lem}[thm]{Lemma}
\newtheorem{prop}[thm]{Proposition}
\newtheorem{alg}[thm]{Algorithm}
\theoremstyle{definition}
\newtheorem{defn}[thm]{Definition}
\theoremstyle{remark}
\newtheorem{rem}[thm]{Remark}
\theoremstyle{example}

\theoremstyle{conjecture}

\numberwithin{equation}{section}

\newcommand{\del}{\delta}

\newcommand{\B}{{\mathbb B}}

\newcommand{\HH}{{\mathbb H}}
\newcommand{\R}{{\mathbb R}}
\newcommand{\SSS}{{\mathbb S}}
\newcommand{\C}{{\mathbb C}}

\newcommand{\N}{{\mathbb N}}

\newcommand{\Z}{{\mathbb Z}}

\newcommand{\im}{{\rm Im}\,}

\newcommand{\calB}{{\mathcal B}}
\newcommand{\calC}{{\mathcal C}}
\newcommand{\calD}{{\mathcal D}}

\newcommand{\calG}{{\mathcal G}}
\newcommand{\calI}{{\mathcal I}}
\newcommand{\calJ}{{\mathcal J}}

\newcommand{\calL}{{\mathcal L}}

\newcommand{\calQ}{{\mathcal Q}}
\newcommand{\calS}{{\mathcal S}}
\newcommand{\calT}{{\mathcal T}}
\newcommand{\calU}{{\mathcal U}}

\newcommand{\frakC}{{\mathfrak C}}
\newcommand{\frakc}{{\mathfrak c}}
\newcommand{\frakS}{{\mathfrak S}}
\newcommand{\sfC}{{\mathsf C}}

\newcommand{\vardbtilde}[1]{\tilde{\raisebox{0pt}[0.85\height]{$\tilde{#1}$}}}

\makeatletter
\newcommand{\doublewidetilde}[1]{{%
  \mathpalette\double@widetilde{#1}%
}}
\newcommand{\double@widetilde}[2]{%
  \sbox\z@{$\m@th#1\widetilde{#2}$}%
  \ht\z@=.9\ht\z@
  \widetilde{\box\z@}%
}
\makeatother

\def\one{\mbox{1\hspace{-4.25pt}\fontsize{12}{14.4}\selectfont\textrm{1}}}

\begin{document}

\title[Sparse domination]{Sparse domination of Singular Radon transform}%

\author{Bingyang Hu}

\address{Bingyang Hu: Department of Mathematics, University of Wisconsin, Madison, WI 53706-1388, USA.}%
\email{bhu32@wisc.edu}

\date{\today}%

\thanks{$^{1}$ Supported in part NSF grant DMS 1600458 and NSF grant 1500162.}


\maketitle


\begin{abstract}
The purpose of this paper is to study the sparse bound of the operator of the form $f \mapsto \psi(x) \int f(\gamma_t(x))K(t)dt$, where $\gamma_t(x)$ is a $C^\infty$ function defined on a neighborhood of the origin in $(x, t) \in \R^n \times \R^k$, satisfying  $\gamma_0(x) \equiv x$, $\psi$ is a $C^\infty$ cut-off function supported on a small neighborhood of $0 \in \R^n$ and $K$ is a Calder\'on-Zygmund kernel suppported on a small neighborhood of $0 \in \R^k$.  Christ, Nagel, Stein and Wainger gave conditions on $\gamma$ under which $T: L^p \mapsto L^p (1<p<\infty)$ is bounded. Under the these same conditions, we prove sparse bounds for $T$, which strengthens their result.  As a corollary, we derive weighted norm estimates for such operators. 
\end{abstract}

\tableofcontents

\section{Introduction}

\bigskip

The purpose of this paper is to prove a sparse domination theorem for the singular Radon transforms. The study of the $L^p$ theory of singular Radon transforms culminated in the work of Christ, Nagel, Stein and Wainger \cite{CNSW}. In this paper, we strength their results to a sparse domination estimate, which roughly says that the behavior of the singular Radon transforms can be captured by a ``sparse collection of cubes". As a corollary of our sparse domination theorems, we obtain weighted estimates for such operators where the weights belong to nonisotropic Muckenhoupt $A_p$ classes (see, Corollary \ref{weightedineq}).  

In general, sparse bounds can be interpreted as a finer quantification of the boundedness of singular integral operators such as Calder\'on-Zygmund operators. These arguments have delivered the most powerful known proof \cite{AL} of the $A_2$ conjecture and later, they are further studied by several authors to establish various interesting results in Calder\'on-Zygmund theory (see, e.g, \cite{CDO, LM, AL2} and references therein). 

In this paper, we study the operators of the form\footnote{We refer the reader Section 1.1 and Section 5 for a more detailed definition of such operators.} 
$$
Tf(x)=\psi(x) \int f(\gamma_t(x)) K(t)dt
$$
where $\gamma$ is a $C^\infty$ mapping from a neighborhood of the origin in $(x, t) \in \R^n \times \R^k$ to $\R^n$, with $\gamma_0(x) \equiv x$ and satisfying the curvature condition of Christ, Nagel, Stein and Wainger (see, Definition \ref{Hormanderdefn}), $\psi(x)$ is a $C^\infty_0(\R^n)$ cut-off function, supported near $0 \in \R^n$ and $K(t)$ is a standard Calder\'on-Zygmund kernel. Such operators are known as singular Radon transforms. 

Sparse domination for singular Radon transforms and related operators is a recent subject. This began with the work of Lacey \cite{LM2} on spherical maximal functions and then was furthered by Oberlin\cite{oberlin} by considering convolution operators induced by general compactly supported measures with Fourier decay. The first attempt where the underlying dilation structure is nonisotroptic was due to Cladek and Ou \cite{CO}, where they studied the sparse bound of the Hilbert transform along the monomial curve. More precisely, given $\gamma: \R \rightarrow \R^n$ a monomial curve, that is
$$
\gamma(t):= \left( |t|^{\alpha_1}, \dots, |t|^{\alpha_n} \right), \quad t \in \R
$$
for real numbers $0<\alpha_1<\dots<\alpha_n<\infty$, the \emph{Hilbert transform along the monomial curve} is defined as 
\begin{equation} \label{190226eq01}
H_\gamma f(x):= p.v. \int_{\R} f(x-\gamma(t)) \frac{dt}{t}, \quad x \in \R^n. 
\end{equation}
Note that the dilation structure in this case is given by the so-called \emph{$\gamma$-cube}, which is a hyperrectangle in $\R^n$ with side parallel to the coordinate axes whose side lengths $(\ell_1, \dots, \ell_n)$ satisfy the relation $\ell_1^{\frac{1}{\alpha_1}}=\dots \ell_n^{\frac{1}{\alpha_n}}$. A second example we shall keep in mind is the sparse bound of the \emph{lacunary spherical maximal function} on the Heisenberg group $\HH^n$. This example can be viewed as an nonisotropic analog of Lacey \cite{LM2} and was studied by Bagchi, Hait, Roncal and Thangavelu \cite{BHRT}. They studied the operators of the form
\begin{equation} \label{190303eq01}
M_{\textrm{lac}} f(z, t)=\sup_{z \in \Z} \left|A_{\del^j} f(z, t)\right|, \quad \del>0, \  (z, t) \in \HH^n=\C^n \times \R, 
\end{equation}
where for $r>0$ and a given function $f$ on $\HH^n$, 
$$
A_r f(z, t):= \int_{|w|=r} f \left(z-w, t-\frac{1}{2} \im z \cdot \overline{w} \right) d\mu_r(w).
$$
Here $\mu_r$ is the normalised surface measure on the sphere $S_r:=\left\{ (z, 0): |z|=r \right\}$ in $\HH^n$. The dilation structure underlying these operators is given by the one on the Heisenberg group, namely, for each $\del>0$ and $z:=[z', t] \in \mathbb H^n=\C^n \times \R$, $\del \circ z=\del \circ [z', t]:=[\del z', \del^2 t]$ (see, e.g., \cite[Page 541]{Stein}). 

Our article furthers this line of research by considering the singular Radon transforms, where, in general, there are no dilation structures and group structures and appears to be the first attempt which adapts  delicate techniques from sub-Riemannian geometry (as introduced by Nagel, Stein and Wainger \cite{NSW}) to study the sparse bound of Radon-type transforms. 

The outline of this paper is as follows. In Section 2, we recall the general theory of dyadic systems in the space of homogeous type, which is due to Hyt\"onen and Kairema, and we also prove a version of Whitney decomposition under their setting. In Section 3, we study the underling spaces of homogeneous type associated to the singular Radon transform. Section 4 is devoted a lemma of modulus of continuity by using a quantitative scaling map technique, due to Street \cite{BS}. In Section 5, we present a quantitative version of the set up of our theorem, and then we state our main theorem. Section 6--7 are devoted to the proof our main theorem (see, Theorem \ref{MainTheorem}), and finally, as an application of our main result, we also get a version of weighted norm estimates for singular Radon transforms (see, Corollary \ref{weightedineq}). \\
\textbf{Acknowledgements} I would like to thank Brian Street. He suggested me this nice project and explained to me lots of the surrounding theories, in particular the theory of the scaling maps from his book \cite{NAO}, which plays a significant role in overcoming the main difficulities in this project, arosed by the generality of the singular Radon transfrom. I would also like to thank Tess Anderson, who explained the dyadic system on the spaces of homogeneous type to me in detail. I thank both of them for their 
invaluable help.

\subsection{Statement of the main result}

\bigskip

We suppose we are given a  $\gamma$ be a $C^\infty$ mapping 
$$
(x, t) \mapsto \gamma(x, t):=\gamma_t(x), \quad (x, t) \in \R^n \times \R^k
$$
from a neighborhood of the origin in $\R^n \times \R^k$ to $\R^n$, satisfying $\gamma_0(x) \equiv x$. Let further,  $\psi(x)$ be a $C_0^\infty(\R^n)$ cut-off function, supported near $0 \in \R^n$ and $K(t)$ be a standard Calder\'on-Zygmund kernel on $\R^k$ supported for $t$ near $0$. This means 
\begin{enumerate}
\item [(a)] $K(t)$ is a distribution with compact support near the origin in $\R^k$ and it coincides with a $C^\infty$ function away from $0 \in \R^k$;
\item [(b)] (Differential inequalities) For each multi-index $\alpha \in \N^k$, there is a constant $C_\alpha>0$ so that for $t \neq 0$, we have
$$
\left| \partial^\alpha K(t) \right| \le C_\alpha |t|^{-|\alpha|-k}; 
$$ 
\item [(c)] (Cancellation condition) Given any $C^\infty$ function $\phi$ supported in the unit ball of $\R^k$ with $C^1$-norm bounded by $1$, and any $R>0$, 
$$
\sup_{R>0, \|\phi\|_{C^1} \le 1} \int_{\R^k} K(t) \phi(Rt)dt<\infty. 
$$
\end{enumerate}
For any Schwarz funciton $f$, we define the \emph{singular Radon transform} of $f$ as
\begin{equation} \label{eq01}
Tf(x):= \psi(x) \int f(\gamma_t(x))K(t)dt. 
\end{equation}
It is well-known that if $\gamma$ is \emph{curved to finite order} at $0$ (see, Definition \ref{Hormanderdefn}), then for any $1<p<\infty$, $T$ can be extended to all $L^p$ functions with a bounded $L^p$ norm (see, e.g., \cite{CNSW}). 

The $L^p$ boundedness of singular Radon transform is a non-trivial result and the original approach by Christ, Nagel, Stein and Wainger relies on a careful study of the nilpotent Lie groups and a lifting procedure due to Rothschid and Stein \cite{RS}. This allows one to study the $L^p$ boundedness problem on a high dimensional Euclidean spaces which is very ``close" to a  stratified Lie group. Later, Stein and Street were able to attack this problem directly without applying the lifting argument , and this improvement allows them to study the $L^p$ boundedness of the multi-parameter case (see, e.g., \cite{SS1, SS, BS}). We proceed by using Stein and Street's approach and we do not apply the Rothschid and Stein's lifting procedure. 

To state our main result, we need the notion of positive sparse form, which consists of two main ingredients. The first ingredient is the region where the associated single scale operators (with respect to $T$) satisfy certain $L^p$ improving estimates. More precisely, we assume the Calder\'on-Zygmund kernel $K$ is supported in $B^k(a) \subsetneq \R^k$,  for some small $a>0$ (depend on $\gamma$) to be choosen later. Then for any $\chi \in C^\infty_0( B^k(a))$, we define the single scale operator\footnote{In Section 5, we will work with a slightly more general single scale operator (see, \eqref{190310eq01}).}
$$
\calT_\chi (f)(x):= \psi(x) \int f(\gamma_t(x)) \chi(t)dt, 
$$
where $f$ is measurable. It turns out that the collection of single scale operators satisfies some certain $L^p$ improving property, that is, for any $r \ge 1$, there exists $s>r$, such that
\begin{equation} \label{190310eq02}
\sup_{\chi \in \calB} \|\calT_\chi f\|_{L^s} \le C_{r, s, \calB} \|f\|_{L^r},
\end{equation}
where $\calB \subseteq C^\infty_0(B^k(a))$ is a bounded set\footnote{In Section 5, we will construct a bounded set $\calB$ from the kernel $K$. However, in general, the constant $C_{r, s, \calB}$ is independent of the choice of $K$.  We refer the reader Theorem \ref{apdthm003} for a more precise statement of this result.} and the constant $C_{r, s, \calB}$ only depends on $r, s$ and $\calB$. To this end, we define
$$
\Sigma:= \textrm{the interior of} \ \left\{ \left(\frac{1}{r}, \frac{1}{s} \right), (r,s) \ \textrm{satisfies \eqref{190310eq02}} \right\}.
$$
The second ingredient is the underlying dyadic structure. Indeed, given a curve $\gamma$ which is curved to finite order at $0$, there is a small neighborhood\footnote{In Section 5, we will quantify how small this neighborhood is.}  near $0 \in \R^n$, which can be viewed as a space of homogeneous type (see, Definition \ref{defnSHT}), and therefore, we can give a dyadic grid structure $\calG$ (see, Section 2 and 5), that is, informally, a collection of  dyadic cubes,  associated to such a neighborhood.  

\begin{defn} \label{sparsedefn01}
Given a dyadic grid $\calG$, a \emph{$\sigma$-sparse family} $\calS \subset \calG$ associated with $\sigma \in (0, 1)$ is a collection of dyadic cubes for which there exists a collection of sets $\{E(Q): Q \in \calS\}$ such that the sets $E(Q)$ are pairwise disjoint, $E(Q) \subset Q$ and $\sigma |Q| \le |E(Q)|$. Here, $| \cdot |$ denotes the Lebesgue measure\footnote{For general space of homogeneous type, the associated measure can be any non-negative doubling Borel measure. In this paper, it suffices for us to consider the case when such a measure is Lebesuge. We will come back to this point in Section 3.} on $\R^n$. 
\end{defn}

\begin{rem}
Let $S$ be a $\sigma$-sparse family. For every $Q \in \calS$, one may choose
$$
E(Q)=Q \big\backslash \bigcup_{P \in \calS, P \subsetneq Q} P.
$$
Then Definition \ref{sparsedefn01} is equivalent to the following:  a collection of dyadic cubes $\calS \subset \calG$ is said to be a $\sigma$-sparse family, if for every $Q \in \calS$, 
$$
\left| \bigcup_{P \in \calS, P \subsetneq Q} P \right| \le (1-\sigma) |Q|.
$$

\end{rem}

\begin{defn} \label{sparseform}
Let $\calG$ be a dyadic grid and $\calS \subset \calG$ be a sparse family. Then  the \emph{$( r, s)$-sparse form} $\Lambda_{\calS, r, s}$ is defined as
$$
\Lambda_{\calS, r, s} (f, g):=\sum_{S \in \calS}|S| \langle f \rangle_{S, r} \langle g \rangle_{S, s}
$$
for $1 \le r, s<\infty$, where $\langle f \rangle_{Q, r}:=\left(\frac{1}{|Q|} \int_Q |f|^r\right)^{\frac{1}{r}}$ and in particular,  $r=1$, we write $\langle f \rangle_Q$, which is short for $\langle f \rangle_{Q, 1}$. We sometimes abbreivate this as $\Lambda_{r, s}$ when the sparse collection $\calS$ is clear from the context. 
\end{defn}

Our main theorem is

\begin{thm}
Supposed the above assumptions hold (which are made precise in Section 5). Then for any $0<\sigma<1$ and a pair $(r,s)$ satisfying $\left(\frac{1}{r}, \frac{1}{s} \right) \in \Sigma$,  there is a constant $C_{\sigma, r, s, T}$ such that for any compactly supported bounded functions $f_1$ and $f_2$ on $\R^n$, there exists a $\sigma$-sparse family $\calS \subset \calG$, such that
$$
|\langle Tf_1, f_2 \rangle| \le C_{\sigma, r, s, T} \Lambda_{\calS, r, s'} (f_1, f_2), 
$$
where $s'$ is the conjugate of $s$. 
\end{thm}

The proof of this theorem consists of two parts. The first part is devoted to constructing an appropriate space of homogeneous type adapted to $T$. The key ingredient to this construction is a quantitative scaling technique\footnote{The scaling map we are using in the current paper can be viewed as a particular case of the quantitative theorem of Frobenius, which was carefully studied in \cite{BS}.} due to Street in \cite{BS}. This powerful tool allows us to make all estimates uniformly at all ``scales" and avoids the original lifting argument in \cite{CNSW}, which would result in a weaker version of $L^p$ improving bound for the single scale operators. In general, the geometry underlying a Radon type transform is a sub-Riemannian geometry. This geometry gives rise to ``dyadic scales" needed to deveolop sparse domination. 

The second part is to find an approatiate approach to apply the sparse domination principle to the dyadic systems constructed in the previous step. Main difficulities under the current situation are that this space of homogeneous type lacks both a dilation structure and a group strucure, in particular, we are not able to identify\footnote{Here by identifying two different ``balls", we are referring to the following easy facts: in $\R^n$, one has $B(a, r)=a+B(0, r)$ for any $a \in \R^n$ and $r>0$; similarly, in $\HH^n$, $B(a, r)=a \cdot B(0, r)$ for any $a \in \HH^n$, where $\cdot$ denotes the group operator in the Heisenburg group $\HH^n$. } two ``balls" with same radius but different centers. The main idea to overcome these difficulties is to observe that by sparse domination principle,  these two ``balls" are indeed not too ``far away" from each other, and it allows us to make a uniform estimate to their first common dyadic ancestor. This is also new, and it gives an example on how sub-Riemannian geometry and sparse domination theory interact with each other.  

\begin{rem}
Throughout this paper, for $a, b \in \R$, $a \lesssim b$ ($a \gtrsim b$, respectively) means there exists a positive number $C$, which is independent of $a$ and $b$, such that $a \leq Cb$ ($ a \geq Cb$, respectively). In particular, if both $a \lesssim b$ and $a \gtrsim b$ hold, then we say $a \simeq b$. Moreover, for $A, B \subseteq \R^n$, $A \Subset B$ means the closure of $A$ is a compact subset of $B$.
\end{rem}

\bigskip

\section{Space of homogeneous type}

\bigskip

In this section, we will recall some general theory of space of homegeneous type, whose geometry plays an important role in defining the sparse operator. We also prove a quantitative version of dyadic Whitney decomposition associated to the Hyt\"onen-Kairema decomposition. 

\begin{defn} \label{defnSHT}
A \emph{space of homogeneous type} is an ordered triple $(X, \rho, \mu)$, where $X$ is a set, $\rho$ is a quasimetric, that is
\begin{enumerate}
\item $\rho(x, y)=0$ if and only if $x=y$;
\item $\rho(x, y)=\rho(y, x)$ for all $x, y \in X$;
\item $\rho(x, y) \le \kappa(\rho(x, z)+\rho(y, z))$, for all $x, y, z \in X$. 
\end{enumerate}
for some constant $\kappa>0$, and the non-negative Borel measure $\mu$ is doubling, that is,
$$
0<\mu(B(0, 2r)) \le D \mu(B(0, r))<\infty, \quad \textrm{for some} \ D>0, 
$$ 
where $B(x, r):=\left\{y \in X, \rho(x, y)<r\right\}$, for $x \in X$ and $r>0$. 
\end{defn} 

We use the following notation in the sequel. For any sets $X_1, X_2 \subset X$, we write $\textrm{diam}(X_1)=\sup\limits_{x, y \in X_1}\rho(x, y)$  and $\textrm{dist}(X_1, X_2)=\inf\limits_{x \in X_1, y \in X_2} \rho(x, y)$. In particular, we abbreviate $\textrm{dist}(X_1, x)=\textrm{dist}(X_1, \{x\})$ for a set $X_1 \subset X$ and $x \in X$.

\begin{defn}
 $(X, \rho, \mu)$ is said to be \emph{$A$-uniformly perfect} if there exists a constant $A>0$, such that for each $x \in X$ and $0<r<\textrm{diam}(X)$, there is a point $y \in X$ which satisfies $A^{-1}r \le \rho(x, y) \le  r$.
\end{defn}

The following result is crucial, which we refer as the \emph{Hyt\"onen-Kairema decomposition}. 

\begin{thm}[{{\cite[Theorem 2.1]{MR3302574}}, \cite{HK}}] \label{dyadicSHT}
There exists a family of sets (\emph{we refer it as a dyadic grid in space of homogeneous type}) $\calD=\bigcup\limits_{k \in \Z} \calD_k$, called a dyadic decomposition of $X$, constants $\frakC>1, 0<\del<\frac{1}{100}, 0<\epsilon<1$, and a corresponding family of points $\{x_c(Q)\}_{Q \in \calD}$, such that
\begin{enumerate}
\item $X=\bigcup\limits_{Q \in \calD_k} Q$, for all $k \in \Z$;
\item For any $Q_1, Q_2 \in \calD$, if $Q_1 \cap Q_2 \neq \emptyset$, then $Q_1 \subseteq Q_2$ or $Q_2 \subseteq Q_1$;
\item For every $Q \in \calD_k$, there exists at least one child cube $Q_c \in \calD_{k+1}$ such that $Q_c \subseteq Q$;
\item For every $Q \in \calD_k$, there exists exactly one parent cube $\hat{Q} \in \calD_{k-1}$ such that $Q \subseteq \hat{Q}$;
\item If $Q_2$ is a child of $Q_1$ , then $\mu(Q_2) \ge \epsilon \mu(Q_1)$;
\item $B(x_c(Q), \del^k)  \subset Q \subset B(x_c(Q), \frakC\del^k)$.
\end{enumerate}
\end{thm}
We will refer to the last property as the \emph{sandwich property}. The sets $Q \in \calD$ are referred to as \emph{dyadic cubes} with center $x_c(Q)$ and sidelength $\ell(Q)=\del^k$, but we must emphasize that these are not cubes in any standard sence even if the underlying space is $\R^n$.

An interesting feature of this definition is the dilation of such a dyadic cube, which is defined as for any $\lambda>1$, 
$$
\lambda Q:=B(x_c(Q), \lambda \frakC \del^k).
$$
Note that this definition does not extend to the case when $\lambda \le 1$. We make a remark that due to this fact, the technique used under setting of space of homogeneous type is quite different from the approaches in the usual Euclidean spaces by Lacey, Cladek, Ou, etc, which depends heavily on the dilation with scale less than $1$. 

Here is a version of Whitney decomposition of space of homogeneous type based on the Hyt\"onen-Kairema decomposition.

\begin{lem} \label{Whitney}
Suppose that $(X, \rho, \mu)$ is an $A$-uniform perfect space of homogeneous type, $Y$ is a closed subset of $X$, and $\Omega=X \backslash Y$. Then $\Omega$ has a Whitney decomposition $M_{\Omega} \subseteq \calD$ with respect to $\calD$, satisfying the following conditions:
\begin{enumerate}
\item $\Omega=\bigcup\limits_{Q \in M_{\Omega}} Q$;
\item For any $\frakc'>2\kappa^2\frakC$ and for each $Q \in M_{\Omega}$, we have
$$
\left(\frac{\frakc'}{2\kappa^2 \frakC}-1\right) \textrm{diam}(Q) \le \textrm{dist}(Q, Y) \le \frac{A\frakc'}{\del} \textrm{diam}(Q).
$$
\item $Q \cap Q'=\emptyset, Q, Q' \in M_{\Omega}$;
\item For any $Q \in M_\Omega$, there exists a $x \in \Omega$, such that
$$
B(x, \del^k) \subseteq Q \subseteq B(x, \frakC \del^k)
$$
for some $k$.
\end{enumerate}
Here, the constant $\frakC$ and $\del$ are defined in Theorem \ref{dyadicSHT}. 
\end{lem}

\begin{proof}
We prove the result by following the idea in \cite{SJ}. Since $X$ is an space of homogeneous type, we apply Theorem \ref{dyadicSHT} and we obtain a dyadic grid $\calD$ with respect to constants $\frakC, \del, \epsilon$ and a corresponding family of points $\{x_c(Q)\}_{Q \in \calD}$. 

We now consider level sets, defined by 
$$
\Omega_k=\left\{x \in \Omega: \frakc'\del^k<\textrm{dist}(x, Y) \le \frakc'\del^{k-1} \right\}.
$$
Obviously, $\Omega=\bigcup\limits_{k=-\infty}^\infty \Omega_k$. 

We now make the initial choice of $Q$'s, and denote the resulting colleciton by $M_0 \subseteq \calD$. Our choice is made as follows:
$$
M_0=\bigcup_k \left\{Q \in \calD_k \bigg | Q \cap \Omega_k \neq \emptyset \right\}. 
$$

\bigskip

\textit{Claim 1: }  For any $\frakc'>2\kappa^2\frakC$, we have for any $Q \in M_0$, 
$$
\left(\frac{\frakc'}{2\kappa^2 \frakC}-1\right) \textrm{diam}(Q) \le \textrm{dist}(Q, Y) \le \frac{A\frakc'}{\del} \textrm{diam}(Q).
$$

\bigskip

First, by Theorem \ref{dyadicSHT} and the assumption that $Q \in \calD_k$ for some $k \in \Z$, there exists a point $x=x_c(Q) \in \Omega$, such that
$$
B(x_c(Q), \del^k) \subseteq Q \subseteq B(x_c(Q), \frakC\del^k),
$$
which, combining with the $A$-unifrm perfectness, implies that、
\begin{equation} \label{Whitney01}
\frac{\del^k}{A} \le \textrm{diam}(Q) \le 2\kappa \frakC\del^k.
\end{equation}
On one hand, since $Q \in M_0$, there exists $x \in Q \cap \Omega_k$, and hence
$$
\textrm{dist}(Q, Y) \le \textrm{dist}(x, Y) \le \frakc' \del^{k-1} \le \frac{A\frakc'}{\del} \textrm{diam}(Q).
$$
On the other hand, for any $q \in Q$ and $y \in Y$, we have
$$
\rho(x, y) \le \kappa( \rho(x, q)+\rho(q, y)),
$$
which implies
$$
\rho(q, y) \ge \frac{\rho(x, y)}{\kappa}-\rho(x, q) \ge \frac{\textrm{dist}(x, Y)}{\kappa}-\textrm{diam}(Q) \ge \frac{\frakc'\del^k}{\kappa}-2\kappa \frakC\del^k. 
$$
Taking the infimum with respect $q$ and $y$ in the above inequality, we have
$$
\textrm{dist}(Q, Y) \ge \frac{\frakc'\del^k}{\kappa}-2\kappa \frakC\del^k \ge \left(\frac{\frakc'}{2\kappa^2 \frakC}-1\right) \textrm{diam}(Q).
$$

\bigskip

\textit{Claim 2: $\Omega=\bigcup\limits_{Q \in M_0} Q$.} 

\bigskip

Indeed, for any $x \in \Omega$, there exists some $k \in \Z$, such that $x \in \Omega_k$. Moreover, since $X=\bigcup\limits_{Q \in \calD_k} Q$, it follows that $x \in Q \in \calD_k$. Thus, $\Omega \subseteq \bigcup\limits_{Q \in M_0} Q$. On the other hand, for any $Q \in M_0$, by Claim 1, the choice of $\frakc'$ ensures that $\textrm{dist}(Q, Y)>0$, which implies that $Q \in \Omega$. 

It is also clear that for each cube in the collection of $M_0$, (4) is satisfied. However, the last requirement (3) may not be satisfied. To finish the proof, we need to furnish our choice of the dyadic cubes from $Q_0$, by eliminating $Q$'s which were really unnecessary. We require the following observation. Suppose $Q \in \calD_k$ and $Q' \in \calD_{k'}$ with $Q \cap Q' \neq \emptyset$, then one of two must be contained in the other, since both of these cubes are taken from $\calD$. Start now with any $Q \in M_0$, and consider the unique maximal parent contains it. We let $M_\Omega$ denote the collection of maximal $Q$'s in $M_0$, which clearly satisfies all the conditions listed above. 
\end{proof}

The following propetries are important in the theory of dyadic calculas. The first one is the Lebesgue differentiation theorem. 

\begin{prop}[{\cite[Corollary 2.5]{MR3302574}}] \label{LDT}
Given a space of homogeneous type $(X, \rho, \mu)$ and a dyadic grid $\calD$ that satisfies the hypotheses of Theorem \ref{dyadicSHT}, then for $\mu$-almost every $x \in X$, if $\{Q_k\}_{k=1}^\infty$ is the sequence of dyadic cubes in $\calD$ such that $\bigcap\limits_{k=1}^\infty Q_k=\{x\}$, then for $f$ measurable, 
$$
\lim_{k \to \infty} \frac{1}{\mu(Q_k)} \int_{Q_k} |f(y)-f(x)|d\mu(y)=0.
$$
\end{prop}

We will use this result to conclude that later in the Calder\'on-Zygmund decomposition, the ``good'' function is bounded almost everywhere. 

The second important property is the Calder\'on-Zygmund decomposition. Recall the dyadic Hardy-Littlewood maximal operator associated to the given dyadic grid $\calD$ is defined as
$$
M^{\calD} f(x)=\sup_{x \in Q, Q \in \calD} \frac{1}{\mu(Q)} \int_{Q} |f(x)|d\mu(x), \quad x \in X. 
$$
We need the following result. 

\begin{prop}[{\cite[Theorem 2.8]{MR3302574}}] \label{CZdeco}
Given a space of homogeneous type $(X, \rho, \mu)$ such that $\mu(X)<\infty$, and a dyadic grid $\calD$ on it, suppose $f$ is a function such that $\int_X |f(x)| d\mu(x)<\infty$. Then for any $\lambda>\frac{1}{\mu(X)} \int_X |f(x)| d\mu(x)<\infty$, there exists a family $\{Q_j\} \subset \calD$ and functions $b$ and $g$ such that
\begin{enumerate}
\item $f=b+g$;
\item $g=f \one_{\{x: M^{\calD} f(x) \le \lambda \}}+\sum\limits_j \langle f \rangle _{Q_j}$;
\item for $\mu$-a.e. $x \in X$, $|g(x)| \le C_X \lambda$;
\item $b=\sum\limits_j b_j$, where $b_j=(f-\langle f\rangle_{Q_j}) \one_{Q_j}$;
\item $\textrm{supp}(b_j) \subset Q_j$ and $\int_{Q_j} b_j(x)d\mu(x)=0$.
\end{enumerate}
\end{prop}

The third important property is the three Lattice Theorem in the setting of space of homogeneous type, which was proved by Hyt\"onen and Kairema.

\begin{prop}[{\cite[Theorem 4.1]{HK}}] \label{TLT}
Suppose $96 \kappa^6 \del \le 1$. Then there exists a finite collection of dyadic grids $\calD^t, t=1, 2, \dots, K_0$, such that for any ball $B=B(x, r) \subset X$, there exists a dyadic cube $Q \in \calD^t$, such that
$$
B \subseteq Q \quad \textrm{and} \quad \ell(Q) \le \widetilde{\frakC}r.
$$
Here,  $\widetilde{\frakC}$ is an absolute constant which only depends on $\kappa$ and $\del$. Moreover, the constants $\frakC_t, \del_t$ and $\varepsilon_t$ constructed in Theorem \ref{dyadicSHT} can be taken to be the same, that is, $\frakC_1=\frakC_2=\dots=\frakC_{K_0}$, $\del_1=\del_2=\dots=\del_{K_0}$ and $\varepsilon_1=\varepsilon_2=\dots=\varepsilon_{K_0}$. 
\end{prop}

\bigskip

\section{Sub-Riemannian geometry}

\bigskip

In this section, we first study the underlying space of homogeneous type inherited in the singular Radon transform \eqref{eq01} under an appropriate ``curvature condition", and therefore, by the results in Section 2, we can construct dyadic systems on it. In the second part, shall state and prove a result describing the uniformness of a class of dyadic decompositions induced by this inherited geometry. 

This section can be viewed as the first part of our construction for the space of homogeneous type used in the sparse domination principle.

\medskip

\subsection{The space of homogeneous type induced by Sub-Riemannian geometry}

\medskip

Recall that $\gamma$ is a $C^\infty$ mapping
$$
(x, t) \mapsto \gamma(x, t)=\gamma_t(x)
$$
from a neighborhood of the point $(0, 0) \in \R^n \times \R^k$ to $\R^n$, satisfying $\gamma(x, 0)=x$. Let us consider the $C^\infty$ vector field
\begin{equation} \label{correctCZ1}
W(t, x):=\frac{\partial}{\partial \varepsilon} \bigg |_{\varepsilon=1} \gamma_{\varepsilon t} \circ \gamma_t^{-1}(x),
\end{equation}
with its formal Taylor expansion
\begin{equation} \label{correctCZ}
W(t) \sim \sum_\alpha t^\alpha X_\alpha.
\end{equation}
Here $\{X_\alpha: 0 \neq \alpha \in \N^k\}$ is a unique collection of $C^\infty$ vector fields, all defined in some common neighborhood $U$ of $0 \in \R^n$; while \eqref{correctCZ} means that for each $N \in \N$, there exists  some constant $C_N>0$, such that for all $x \in U$, 
\begin{equation} \label{190320eq01}
\left|\exp \left(W(t, x) \right) x- \exp \left( \sum_{0<|\alpha|<N} t^\alpha X_\alpha/ \alpha! \right)x \right| \le C_N |t|^N,
\end{equation}
when $t$ is sufficiently small.

\begin{defn} \label{Hormanderdefn01}
We say a collection of vector fields $\frakS$ satisfies \emph{H\"ormander's condition} at $0$ if the Lie algebra generated by the vector fields in $\frakS$ spans the tangent space to $\R^n$ at $0$.  Moreover, if the length of the commutators in a spanning set is at most $m$, we say that $\frakS$ is of \emph{type $m$}. 
\end{defn}

\begin{defn} \label{Hormanderdefn}
We say that $\gamma$ is \emph{curved to finite order} at $0$ if the collection of vector fields $X_\alpha$ defined in \eqref{correctCZ} satisfies the H\"ormander's condition at $0$.
\end{defn}

\begin{rem}
An alternative way to define a curve $\gamma$ which is curved to finite order at $0$ is to consider those vector fields given by the Taylor expansion of $\gamma$ itself. More precisely, we may write
$$
\gamma(t, x) \sim \exp\left( \sum t^\alpha \widehat{X_\alpha}/ \alpha! \right) x
$$
in the sense of \eqref{190320eq01}. \cite[Theorem 9.1]{BS} asserts that the collection $\left\{X_\alpha: 0 \neq \alpha \in  \N^k \right\}$ satisfies the H\"ormander's condition at $0$ if and only if the collection $\left\{\widehat{X_\alpha}: 0 \neq \alpha \in  \N^k \right\}$ does. 
\end{rem}

This assumption turns out to be very important in defining the correct geometry, as well as proving the $L^p (1<p<\infty)$ boundedness of the operator \eqref{eq01}. Here, we list some consequences of H\"ormander's condition we need in the sequel. 

\medskip

\textit{1. Curved to finite order is equivalent to the following curvature condition $(\calC_J)$. }

\medskip

Christ, Nagel, Stein and Wainger \cite{CNSW} gave several conditions which are equivalent to $\gamma$ being curved to finite order at $0$. Among all these conditions, we are interested in the one which was denoted as $(\calC_J)$ by them. To define this condition, we first consider the iterates of the mapping $t \mapsto \gamma(x, t)$, namely, for any $1 \le j \le n$, define $\Gamma^1(x, t)=\gamma(x, t)$ and
$$
\Gamma^j(x, t_1, \cdots, t_j)=\gamma( \Gamma^{j-1}(x, t_1, \dots, t_{j-1}), t_j).
$$
Among these iterates, we single out the $n$-th iterate, 
$$
\Gamma(x, \tau)=\Gamma^n(x, \tau)
$$
for $\tau \in \R^{kn}$. The domain of the map $\tau \mapsto \Gamma(x, \tau)$ is a small neighborhood of $0 \in \R^{nk}$; its range is contained in a small neighborhood of $x \in \R^n$. 

Write $\tau=(\tau_1, \tau_2, \dots, \tau_{kn})$, where the coordinates belong to $\R$ and are ordered in any fixed manner. To each $n$-tuple $\xi=(\xi_1, \dots \xi_n)$ of elements of $\{1, 2, \dots, kn\}$ is associated the Jacobian determinant
$$
J_{\xi}(x, \tau)=\det \left( \frac{\partial \Gamma(x, \tau)}{\partial( \tau_{\xi_1}, \dots,  \tau_{\xi_n})} \right)
$$
of the $n \times n$ submatrix of the differential of $\Gamma$ with respect to $\tau$. 

\begin{defn} \label{defnCJ}
We say $\gamma$ satisfies \emph{curvature condition $(\calC_J)$} at $0$, if there exists an $n$-tuple $\xi$ and a multi-index $\beta$ such that 
\begin{equation} \label{eq03}
\partial^\beta_\tau J_\xi(0, \tau) \big |_{\tau=0} \neq 0,
\end{equation}
where $\partial^\beta_\tau$ represents an arbitrary partial derivative with respect to the full variable $\tau \in \R^{kn}$, not merely $(\tau_{\xi_1}, \dots, \tau_{\xi_n})$. 
\end{defn}

\cite[Theorem 8.8]{CNSW} states that $\gamma$ is curved to finite order at $0$ if and only if $\gamma$ satisfies the curvature condition $(\calC_J)$ (see, also \cite[Theorem 9.1]{BS}). Therefore, we can take an $n$-tuple $\xi_\gamma$ and a multi-index $\beta_\gamma$ associated to $\gamma$, such that
\begin{equation} \label{190305eq01}
c_\gamma:=\left| \partial^{\beta_\gamma}_\tau J_{\xi_\gamma} (0, \tau) \big |_{\tau=0} \right|>0. 
\end{equation} 

\medskip

\textit{2. Space of homogeneous type and Carnot-Carath\'edory metric.}

\medskip

Nagel, Stein, and Wainger \cite{NSW} showed that Carnot-Carath\'eodory metrics induce a space of homogeneous type. We recall some of their results here. Let $0 \in \Omega' \Subset\Omega \subset \R^n$, where both $\Omega$ and $\Omega'$ are connected open sets, and suppose $Y_1, \dots, Y_q$ are $C^\infty$ real vector fields defined on a neighborhood of $\overline{\Omega}$. We suppose that each vector field $Y_j$ has associated a formal degree $d_j \ge 1$, where $d_j$ is an integer. Henceforth, we shall write $(Y, d)$, which is short for the list $(Y_1, d_1), \dots, (Y_q, d_q)$. We now make the following hypotheses:

\begin{enumerate}
\item[(a).] For each $j$ and $k$ we can write 
$$
[Y_j, Y_k]=\sum_{d_l \le d_j+d_k} c_{j, k}^l(x) Y_l
$$
where $c^l_{j, k}\in C^\infty(\Omega)$. 

\item[(b).] For each $x \in \overline{\Omega'}$, the vectors $Y_1(x), \dots Y_q(x)$ span $\R^n$. 
\end{enumerate}

Noting that in the above setting, we do not require any linearly indepedence assumption. 

\begin{defn} \label{defn01}
Let $\del>0$ and $C(\del)$ denote the class of absolutely continuous mappings $\varphi: [0, 1] \mapsto \Omega'$ which satisfy the differential equation\footnote{This means $ \varphi(t)=\varphi(0)+\int\limits_0^t \left(\sum\limits_{j=1}^q a_j(s) \del^{d_j} Y_j(\varphi(s)) \right) ds$.}
$$
\varphi'(t)=\sum_{j=1}^q a_j(t) \del^{d_j} Y_j(\varphi(t))
$$
with
$$
|a_j(t)|<1, \quad a \in L^\infty([0, 1]). 
$$
Then for $x, y \in \Omega'$, define the \emph{Carnot-Carath\'eodory metric} as 
$$
\rho(x, y)=\inf\left\{\del>0 | \  \exists \varphi \in C(\del) \ \textrm{with} \ \varphi(0)=x, \varphi(1)=y\right\}.
$$
\end{defn}

By \cite[Proposition 1.1]{NSW}, $\rho$ is a metric (in particular, this suggests that we can take $\kappa=1$, where $\kappa$ is the constant associated to the quasimetric $\rho$ in Definition \ref{defnSHT}), and it is continuous in the sense that $\rho: \Omega' \times \Omega' \mapsto [0, \infty)$ is continuous. For $x \in \Omega'$ and $\del>0$, we can define the \emph{Carnot-Carath\'eodory ball} centered at $x$ with radius $\del$ on $\Omega'$ by 
$$
B_{(Y, d)}(x, \del):=\left\{y \in \Omega' | \rho(x, y)<\del\right\}.
$$

The main results in \cite{NSW} give the following properties of $B_{(Y, d)}(x, \del)$. There exists a $\del_0>0$, such that for each $x \in \Omega'$ and each $\del$ with $0<\del \le \del_0$, the set $B_{(Y, d)}(x, \del) \subset \Omega$ satisfies :
\begin{enumerate}
\item [(1).] $B_{(Y, d)}(x, \del)$ is open, and if $0<\del \le \del_0, B_{(Y, d)}(x, \del)=\bigcup\limits_{s<\del} B_{(Y, d)}(x, s)$;
\item [(2).] $\bigcap\limits_{s>0} B_{(Y, d)}(x, s)=\{x\}$;
\item [(3).] For any $x_1, x_2 \in \widetilde{\Omega}$ and $\del_1 \le \del_2$, if 
$$
B_{(Y, d)}(x_1, \del_1) \cap B_{(Y, d)}(x_2, \del_2) \neq \emptyset,
$$
then $B_{(Y, d)}(x_1, \del_1) \subset B_{(Y, d)}(x_2, 3\del_2)$;
\item [(4).] For every compact set $\widetilde{\Omega} \Subset \Omega'$ and $\del \le \del_0$, there are constant $C_{\widetilde{\Omega},1}$ and $C_{\widetilde{\Omega},2}$ so that for all $x \in \widetilde{\Omega}$,
\begin{equation} \label{1231eq01}
0<C_{\widetilde{\Omega},1} \le \frac{\textrm{Vol}(B_{(Y, d)}(x, \del))}{\Lambda(x, \del)} \le C_{\widetilde{\Omega},2}<\infty, 
\end{equation}
where 
$$
\Lambda(x, \del)=\sum_I |\lambda_I(x)| \del^{d(I)}.
$$
Here $I=(i_1, \dots, i_n), 1 \le i_j \le q$ is a $n$-tuple of integers. For such an $I$, set
$$
\lambda_I(x)=\det(Y_{i_1}, \dots, Y_{i_n})(x)
$$
and 
$$
d(I)=d_{i_1}+\dots+d_{i_n}
$$
and $\textrm{Vol}(A)$ denotes the induced Lebesgue volume on the leaf generated by the $Y_j$’s, passing through the point $x_0$;  
\item [(5).] As a consequence of (4), we have for every compact set $\widetilde{\Omega} \Subset \Omega'$, there is a constant $C_{\widetilde{\Omega}}$ so that if $x \in \widetilde{\Omega}$ and if $\del<\frac{\del_0}{2}$,
\begin{equation} \label{doublingineq}
\textrm{Vol}(B_{(Y, d)}(x, 2\del)) \le C_{\widetilde{\Omega}} \textrm{Vol}(B_{(Y, d)}(x, \del)).
\end{equation}
\end{enumerate}

Finally, it will be convenient to assume that the ball $B_{(Y, d)}(x, \del)$ lies ``inside" of $\Omega$ in the following sense. 

\begin{defn}
Given $x \in \Omega$ and $\Omega' \Subset \Omega$ as above. We say the list of vector fields $Y$ satisfies $\sfC(x, \Omega')$ if for every $a=(a_1, \dots, a_q) \in \left( L^\infty([0, 1]) \right)^q$, with
$$
\|a\|_{L^\infty([0, 1])}=\left\| \left( \sum_{j=1}^q |a_j|^2 \right)^{\frac{1}{2}} \right\|_{L^\infty([0, 1])}<1, 
$$
there exists a solution $\varphi: [0, 1] \to \Omega'$ of the ODE
$$
\varphi'(t)=\sum_{j=1}^q a_j(t)Y_j(\varphi(t)), \quad \varphi(0)=x.
$$
\end{defn}
Note, by Gronwall's inequality, when this solution exists, it is unique. Similarly, we say $(Y, d)$ satisfies $\sfC(x, \sigma, \Omega')$ if $\sigma Y$ satisfies $\sfC(x_0, \Omega')$, where $\sigma Y$ is the list $(\sigma^{d_1}Y_1, \dots, \sigma^{d_q}Y_q)$ for $\sigma>0$. 

\medskip

Recall that  we wish to construct a space of homogeneous type from a given $\gamma_t(x)$ which is curved to finite order at $0$. To do this, we first note that by \eqref{190305eq01}, we can take $V \Subset U$ be a sufficiently small, open and path-connected neighborhood of $0 \in \R^n$, so that
$$
\left| \partial^{\beta_\gamma}_\tau J_{\xi_\gamma} (x, \tau) \big |_{\tau=0} \right | \ge \frac{c_\gamma}{2},
$$
uniformly in $x \in \overline{V}$. Thus, by \cite[Theorem 9.1]{BS}, we can take some $m_0$, such that the collection $\{X_\alpha\}$ defined in \eqref{correctCZ} is of type $m_0$ on $\overline{V}$

Take $\Omega=U$ and $\Omega'=V$ in the above general setting. Next, to define the Carnot-Carath\'eodory metric on $V$, we need to construct a finite collection of vector fields, which satisfy the assumptions (a) and (b), and the idea is to choose such a finite collection from $\{X_\alpha\}$ and their commutators. 

It turns out that one can do this better: not only we can pick a subset from $\{X_\alpha\}$ and  their commutators, which satisfies the assumption $(a)$ and $(b)$, but also it will satisfy some other nice properties. Such a construction was introduced in \cite{BS}, where Street used this idea to study the $L^2$ boundedness of the multi-parameter singular Radon transform. 

We recall some basic definitions first. 

\begin{defn}
Let $\calL=(\calL_1, \dots, \calL_q)$ be a list of, possibly non-commuting, operators, we use \emph{ordered multi-index} notation to define $\calL^\alpha$, where $\alpha$ is a list of numbers $1, \dots, q$. $|\alpha|$ will denote the length of the list. For instance, if $\alpha=(1, 2, 2, 3, 1)$, then $|\alpha|=5$ and $\calL^\alpha=\calL_1\calL_2\calL_2\calL_3\calL_1$. Moreover, if $\calL_1, \dots, \calL_q$ are vector fields, then $\calL^\alpha$ is an $|\alpha|$ order partial differential operator.
\end{defn}

\begin{defn}
Consider the formal Taylor series \eqref{correctCZ} of $W$. We assign each $X_\alpha$ a formal degree $|\alpha|$ and then define
$$
\frakS(W):=\left\{ (X_\alpha, |\alpha|): |\alpha|>0 \right\}
$$
and let $\calL(\frakS(W))$ be the smallest set such that:
\begin{enumerate}
\item [$\bullet$] $\frakS(W) \subset \calL(\frakS(W))$;

\item [$\bullet$] If $(X_1, d_1), (X_2, d_2) \in  \calL(\frakS(W))$, then $([X_1, X_2], d_1+d_2) \in  \calL(\frakS(W))$. 
\end{enumerate}
\end{defn}
 
We now introduce the way to choose such a collection of vector fields, which is introduced in \cite[Section 6]{BS}. 

\begin{alg} \label{alg001}

Let $\gamma$ be a mapping from $\R^n \times \R^k \to \R^n$, which is curved to finite order at $0 \in \R^n$. 

\medskip

\textit{Step I:} Recall that the collection $\frakS(W)$, which is given by \eqref{correctCZ}, is of type $m_0$ on $\overline{V}$.  Therefore, we can take a finite list of vector fields 
$$
\{(X_1, d_1), \dots, (X_r, d_r)\} \subseteq \frakS(W), 
$$ 
which is also of type\footnote{The choice of such a list is not unique, as long as it is of type $m_0$.} $m_0$ . Note that this list only depends on $\gamma$;

\medskip

\textit{Step II:} Enumerate the list of all commutators of $\{(X_1, d_1), \dots, (X_r, d_r)\}$ up to order $m_0$ along with their formal degrees, and we label them as
$$
(X_1, d_1), \dots, (X_L, d_L);
$$
so that $X_1, \dots, X_L$ span the tangent space at each point in $V$. Note that 
$$
\{(X_1, d_1), \dots, (X_L, d_L) \} \subseteq \calL(\frakS(W)); 
$$ 

\medskip

\textit{Step III:} Let $(X_1, d_1), \dots, (X_q, d_q)$ be an enumeration of all the vector fields that belongs $\calL(\frakS(W))$, and such that their formal degrees are less or equal to $\max\limits_{1 \le l \le L} d_l$. Note that there are only a finite number of such vector fields.         
\end{alg}

By \cite[Proposition 17.3]{BS},  the list $(X, d):=\{(X_i, d_i)\}_{1 \le i \le q}$ constructed in the above algorithm satisfies the following properties: 

\medskip

(1). There exists a $0<\del_0<1$ and some compact set $\widetilde{V} \Subset V$ such that
\begin{equation} \label{190404eq01}
 (X, d) \ \textrm{satisfies} \ \sfC(x, \del_0, V), \ \forall x \in \widetilde{V}.
\end{equation}

\medskip

(2). $(X, d)$ satisfies the assumptions (a) and (b) above Definition \ref{defn01}, and hence we can construct a Carnot-Carath\'eodory metric $\rho$ associated to $\gamma$;

\medskip

To be self-contained, we include the proof. It is easy to see that assumption (b) is satisfied. To verify assumption (a), that is, for $1 \le i, j \le q$,  
\begin{equation} \label{1116eq16}
[X_j, X_k]=\sum_{d_l \le d_j+d_k} c_{j, k}^l X_l, \quad c_{j, k}^l \in C^\infty(U), 
\end{equation}
 we note that, if $d_j+d_k \le \max\limits_{1 \le l \le L} d_l$, then $([X_j, X_k], d_j+d_k)$ is already in the list $(X, d)$ by definition. On the other hand, if $d_j+d_k>\max\limits_{1 \le l \le L} d_l$, we use the fact that $[X_j, X_k]=\sum\limits_{l=1}^L c_{j, k}^l X_l, c_{j, k}^l \in C^\infty$, since $X_1, \dots, X_L$ span the tangent space at each point; 

\bigskip

(3). $(X, d)$ controls $\gamma(x, t)$, in the following sense:\\

 Let $0 \in V''  \Subset V' \Subset \widetilde{V} \Subset  V$, with $V''$ compact and $V'$ open,  relatively compact, connected, and $a'$ sufficiently small, such that the map
\begin{equation} \label{eq1001}
\gamma(x, t): V' \times B^k(a') \rightarrow V
\end{equation}
is well defined, where $B^k(a')$ is the Euclidean ball in $\R^k$ centered at the origin with radius $a'$. Moreover, we may also assume that for every $t \in B^k(a')$, $\gamma_t$ is a diffeomorphism onto its image by the inverse function theorem, so that it make sense to write $\gamma_t^{-1}$ in the sequel. Here we write $\gamma_t(x)=\gamma(x, t)$ for $(x, t) \in V' \times B^k(a')$. 

\begin{defn} \label{controlVC}
We say $(X, d)$ \emph{controls} $\gamma$ if there exists $0<a_1 \le a'$ and $\tau>0$, which is sufficiently small,  such that for every $x_1 \in V'', \sigma \in [0, 1]$, there exists functions $c_l^{x_1, \sigma}$ on $B^k(a_1) \times B(x_1, \tau \sigma)$ satisfying 

$\bullet$ $W(\del t, x)=\sum\limits_{l=1}^q c_l^{x_1, \sigma}(t, x) \sigma^{d_l} X_l(x)$ on $B^k(a_1) \times  B(x_1, \tau \sigma)$, where $W(t, x)$ is defined in \eqref{correctCZ1};

$\bullet$ 
$$
\sup_{x_1 \in V'', \sigma \in [0, 1]} \sum_{|\alpha|+|\beta| \le N} \left\| (\sigma X)^\alpha \partial_t^\beta c_l^{x_1, \sigma} \right\|_{C^0\left(B^k(a_1) \times B(x_1, \tau \sigma)\right)}<\infty,
$$
for every $N \in \N$, where $\sigma X$ is the list $(\sigma^{d_1}X_1, \dots, \sigma^{d_q} X_q)$, $\alpha$ is an ordered multi-index and $\beta$ is a $k$-tuple multi-index. 
 
\end{defn}

\begin{rem}
Recall that in \eqref{eq01}, we make the following assumptions
\begin{enumerate}
\item [$\bullet$] The kernel $K$ to be supported near $0 \in \R^k$; 
\item [$\bullet$] The cut-off function $\psi$ is supported near $0 \in \R^n$. 
\end{enumerate}

In the sequel, we are interested in the case that\footnote{Indeed, in our precise assumptions on $T$ in Section 5, $\textrm{supp} K$ and $\textrm{supp} \psi$ will be ``smaller" and we will come back to this point later.}
\begin{enumerate}
\item [$\bullet$] $\textrm{supp} K \subseteq B^k(a')$;  
\item [$\bullet$] $\textrm{supp} \psi \subseteq V''$,
\end{enumerate}
where $a'$ and $V''$ are defined as above. 
\end{rem}

\medskip

By Algorithm \ref{alg001}, given a $C^\infty$ map $\gamma$ which is curved to finite order at $0 \in \R^n$,  we see that $(V, \rho, |\cdot|)$ locally forms a space of homogeneous type. Applying Theorem \ref{dyadicSHT}, we get a dyadic system on $V$, 
\begin{equation} \label{190311eq01}
\calD=\bigcup_{k \in \Z} \calD_k,
\end{equation} 
with constants $\frakC>1$, $0<\del<\frac{1}{100}$ and $\epsilon<1$ satisfying all the conclusions in Theorem \ref{dyadicSHT}.

Finally, in order to apply the Whitney decomposition Theorem \ref{Whitney}, we shall show that $(V, \rho, |\cdot|)$ is $3$-uniformly perfect.

\begin{prop} \label{3uniformperfect}
The space of homogeneous type $(V, \rho, |\cdot|)$ is $3$-uniformly perfect.
\end{prop}

To prove the above result, we need the following lemma. 

\begin{lem}[{\cite[Proposition 1.1]{NSW}}] \label{metriceq}
Let $V$ and $\rho$ defined as above. Let further, $\widetilde{d}=\max\limits_{1 \le j \le q} d_j$, and $\widetilde{V} \Subset V$ be any compact set, then there exist constants $C_1$ and $C_2$, so that if $x, y \in \widetilde{V}$, 
$$
C_1|x-y| \le \rho(x, y) \le C_2 |x-y|^{\frac{1}{\widetilde{d}}}.
$$
\end{lem}

\begin{proof} [Proof of Proposition \ref{3uniformperfect}]
Let $r_0:=\textrm{diam}(V)<\infty$. By definition, it suffices to show that for each $x \in V$, and $0<r<r_0$, there is a point $y \in V$, which satisfies 
$$
\frac{r}{3} \le \rho(x, y) \le r. 
$$
Take an $r \in (0, r_0)$. By triangle inequality, there exists some $x_0 \in V$, such that
$$
\rho(x, x_0):=r_1 \in \left(\frac{r_0}{3}, r_0\right] . 
$$
Next, we take some $\varepsilon>0$, sufficiently small, such that there exists an absolute continuous mapping $\varphi: [0, 1] \mapsto V$ satisfying
$$
\varphi \in C(r_1+\varepsilon) \ \textrm{with} \ \varphi(0)=x, \varphi_1(1)=x_0. 
$$ 
Then, clearly $\widetilde{V}:=\im(\varphi)$ is a compact subset of $V$, and hence we can take $W \Subset V$, such that $\widetilde{V} \subset W$ with $W$ open and path-connected.

Fix such a $W$. Then by Lemma \ref{metriceq}, we know that the Euclidean topology on $\overline{W}$ is the same as the topology induced by $\rho$, and hence so is the topology on $W$. Consider the level sets
$$
W_0:=\left\{y \in W: \rho(x, y)\ge  r_1 \right\}.
$$
and
$$
W_i:=\left\{y \in W: \frac{r_1}{i+1} \le \rho(x, y)< \frac{r_1}{i} \right\}, \quad i \ge 1. 
$$

We first note that $W_0$ is not empty, as $x_i \in W_0$ for $i=0$ or $1$ by the triangle inequality. It is clear that $W=\bigcup\limits_{i=0}^\infty W_i$, $W_i \cap W_j=\emptyset$ for $i \neq j$ and $\{x\}=\bigcap\limits_{i=0}^\infty W_i$. Moreover, since $W$ is path connected, we can also see that $W_i \neq \emptyset$ for $i \ge 0$. Finally, since $r$ belongs to one of the following intervals $\left[ r_1, r_0\right]$ and $\left[ \frac{r_1}{i+1}, \frac{r_1}{i} \right), i \ge 1$, the desired result follows easily from the construction. 
\end{proof}


\subsection{A uniform theorem for the dyadic decompositions}

\medskip

In the second part of this section, we prove a uniform theorem for the dyadic decompositions induced by Carnot-Carath\'edory metric. This result plays an important role in Section 5 later when we ``rescale"  the dyadic system into an appropriate one, where sparse domination principle applies.  

Recall that by Algorithm \ref{alg001}, we can construct a collection of vector fields $(X, d)$ so that we can further construct a space of homogeneous type $(V, \rho, |\cdot|)$ as in \eqref{190311eq01}. We start with the following observation. Let $w>0$ and consider the list of vector fields
$$
(w^dX, d).
$$
Then by the argument in Algorithm \ref{alg001}, it is easy to see the following facts holds:
\begin{enumerate}
\item [(1).] For any $x \in \widetilde{V}$, $(w^d X, d)$ satisfies $\calC(x, w, V)$; 

\item [(2).] $(w^dX, d)$ satisfies the assumptions $(a)$ and $(b)$ above Definition \ref{defn01}, and therefore, we can construct another Carnot-Carath\'edory metric $\rho_{w^d X} $ on $V$, and it is easy to see that 
$$
\rho_{w^dX} (x, y)=w^{-1} \rho(x, y);
$$

\medskip

\item [(3).] $(w^dX, d)$ controls $\gamma$, and the implict constants in the definition of control depends only on $w$ and how $\gamma$ is controlled by $(X, d)$. 
\end{enumerate}

Therefore, we can construct a new space of homogeneous type $(V, \rho_{w^dX}, |\cdot|)$ as before. Next we apply Theorem \ref{dyadicSHT} to the SHT $(V, \rho_{w^dX}, |\cdot|)$ and conclude that there exists a dyadic decomposition 
$$
\calD_{w^dX}=\bigcup_{k \in \Z} \calD_{w^dX, k}
$$
with constants $\frakC_{w^dX}>1, 0<\del_{w^dX}<\frac{1}{100}$ and $\epsilon_{w^dX}<1$ and a collection of centers 
$$
\{x_{c, w^dX}(Q)\}_{Q \in \calD_{w^dX}}
$$  
satisfying the conclusions in Theorem \ref{dyadicSHT}.

An interesting question is to ask what is the relationship between $(V, \rho, |\cdot|)$ and $(V, \rho_{w^dX}, |\cdot|)$, in particular, how their dyadic decompositions are related. 

\begin{thm} \label{uniformdecomp}
Under the above setting, the dyadic decomposition $\calD_{w^d X}$ can be made to satisfy the following conditions: 
\begin{enumerate}
\item [1.] $\del_{w^d X}=\del$;
\item [2.] $\epsilon_{w^d X}=\epsilon$;
\item [3.] $\frakC_{w^d X} \le \frac{\frakC}{\del}$;
\item [4.] For each $k \in \Z$, $\calD_{w^dX, k}=\calD_{k+N_w}$, where $N_\omega$ is an absolute constant, which only depends on $w$ and $\del$, that is, the cubes in the $k$-th generation of $\calD_{w^d X}$ is exactly the same as those in the $(k+N_w)$-th generation of $\calD$. In particular, we have $\{x_{c, w^d X}(Q)\}_{Q \in \calD_{w^d X, k}}=\{x_c(Q)\}_{Q \in \calD_{k+N_w}}$ for each $k \in \Z$. 
\end{enumerate}
\end{thm}

\begin{proof}
We consider two different cases. 
\medskip

\textit{Case I: $w=\del^{M_0}$ for some $M_0 \in \Z$.}

\medskip

In this case, we see that the list $(w^dX, d)=(\del^{M_0 d}X, d)$, and hence
\begin{eqnarray*}
\{x \in V: | \rho_{\del^{M_0 d}X}(x, x_0)<r\}%
&=& B_{(\del^{M_0 d}X, d)}(x_0, r)=B_{(X, d)}(x_0,\del^{M_0}r)\\
&=&\{x \in V: | \rho(x, x_0)<\del^{M_0} r\}
\end{eqnarray*}
for $x_0 \in V$ and $r>0$. Therefore, by letting $N_w=M_0$, $\del_{w^d X}=\del$, $\epsilon_{w^d X}=\epsilon$ and $\frakC_{w^d X}=\frakC$, it is easy to see that the desired result follows.

\medskip

\textit{Case II: $\del^{M_1+1}<w<\del^{M_1}$ for some $M_1 \in \Z$.}

\medskip

Note that in this case, we have
\begin{equation} \label{uniformeq0011}
B_{(X, d)}(x_0, \del^{M_1+1}r) \subseteq B_{(w^d X, d)}(x_0, r) \subseteq  B_{(X, d)}(x_0, \del^{M_1}r) 
\end{equation}
for $x_0 \in V$ and $r>0$. Now we let $N_w=M_1$, that is, $\calD_{w^dX, k}=\calD_{k+M_1}$ for each $k \in Z$, $\del_{w^dX}=\del$ and $\epsilon_{w^dX}=\epsilon$. Finally, we claim that $\frakC_{w^dX}=\frac{\frakC}{\del}$.  To see this, we take a cube 
$$
Q \in \calD_{w^dX, k}=\calD_{k+M_1},
$$
for some $k \in \Z$, by our assumption, we have
$$
B_{(X, d)}(x_c(Q), \del^{k+M_1}) \subset Q \subset B_{(X, d)}(x_c(Q), \frakC\del^{k+M_1}).
$$
This, together with \eqref{uniformeq0011}, implies
$$
B_{(w^d X, d)}(x_c(Q), \del^k) \subset Q \subset B_{(w^d X, d)}\left(x_c(Q), \frac{\frakC}{\del} \cdot \del^k\right),
$$
where in the first inclusion, we apply the second inclusion in \eqref{uniformeq0011} with $r=\del^k$ and in the second one, we apply the first inclusion in \eqref{uniformeq0011} with $r=\frakC\del^{k-1}$. 
\end{proof}

\begin{rem}
We make some remarks for the above result. 
\begin{enumerate}
\item[(1).] We may think the above result as a uniform description for the dyadic decompositions of the collection of spaces of homogeneous type $\{(V, \rho_{w^d X}, |\cdot|)\}_{w>0}$, in the sense that once we know one dyadic decompsition of $(V, \rho_{w^d X}, |\cdot|)$ for a particular choice of $w$, then we can construct a dyadic decomposition for all $(V, \rho_{w^d X}, |\cdot|)$, $w>0$. Moreover, the parameters of these SHTs (that is, the constants in Theorem \ref{dyadicSHT}) are controlled uniformly, indepedent of the choice of $w$;

\medskip

\item [(2).] We will see from Section 5 that the second part of the construction of a proper space of homogeneous type is indeed on dealing with how to choose a proper $w$ (once such a $w$ is choosen, it will be fixed in the rest of the paper). At this moment, one may think $w$ is choosen to be a sufficiently small number. 

\medskip

\item [(3).] The same proof of Proposition \ref{3uniformperfect} yields the following easy fact: for any $w>0$, the space of homogeneous type $(V, \rho_{w^d X}, |\cdot|)$ is $3$-uniformly perfect. 

\end{enumerate}
\end{rem}

\bigskip

\section{The lemma of modulus of continuity}

\bigskip

In this section, we study a lemma of modulus of continuity, which plays an important role in the coming estimates. Before we move on, let us consider our model case \eqref{190226eq01}, that is, the Hilbert trasform along the monomial curve. Recall that 
$$
H_\gamma f(x)=p.v. \int_\R f(x-\gamma(t))\frac{dt}{t}, \quad x \in \R^n,
$$
where $\gamma(t)=(|t|^{\alpha_1},  \dots, |t|^{\alpha_n}), \quad t \in \R$ for real numbers $0<\alpha_1<\dots<\alpha_n<\infty$. Define the \emph{single scale operator} $A_\gamma$ by
$$
A_\gamma f(x):=\int_{\frac{1}{2} \le |t|<1} f(x-\gamma(t))\frac{dt}{t}. 
$$ 
An easy application of Plancherel and van der Corput's lemma yields the following result: there exists some $\eta>0$, such that
\begin{equation} \label{hilbertCV}
\|A_\gamma-\tau_y A_\gamma\|_{L^2 \mapsto L^2} \lesssim |y|^\eta, 
\end{equation}
where $\tau_y$ is the translation operator by $y \in \R^n$, namely, $\tau_yf(x):=f(x-y)$ (see, e.g., \cite[Lemma 2.1]{CO}). 

Our aim in this section is to generalize this result under the general setting of Sub-Riemannian geometry.


\subsection{The quantitative scaling maps}

\medskip

The key tool to generalize the inequality \eqref{hilbertCV} is a quantitative scaling map, which is originated in the work of Nagel, Stein and Wainger \cite{NSW}, then studied systematically by Street in \cite{BS2} and later, plays an important role in the study of the $L^p$ boundness of singular Radon transform by Stein and Street in the multi-parameter setting (see, e.g., \cite{SS1, SS, BS}). Moreover, it turns out that this result contains all the geometric properties one need to study the sparse bound of singular Radon transform. Therefore, in the first half of this section, we will recall this important result, together with its corollaries. We start with some notations. 

Let $\Omega \subset \R^n$ and $Y=\sum\limits_{i=1}^n a_i(x) \partial_{x_i}$ be a $C^\infty$ vector fields defined on $\Omega$, then we write
$$
\|Y\|_{C^j(\Omega)}:=\sum_{i=1}^n \|a_i\|_{C^j(\Omega)}, \quad j \in \N. 
$$
Next, given two integers $1 \le m \le n$, we let
$$
\calJ(m, n):=\left\{ (i_1, \dots, i_m): 1 \le i_1<i_2<\dots<i_m\le n \right\}. 
$$ 
Furthermore, suppose $A$ is an $n \times q$ matrix and suppose $1 \le \tilde{n} \le \min\{n, q\}$. For $I \in \calJ( \tilde{n}, n )$, $J \in \calJ( \tilde{n}, q)$, we let $A_{I, J}$ denote the $\tilde{n} \times \tilde{n}$ matrix given by taking the rows form $A$ which are listed in $I$ and the columns from $A$ which are listed in $J$. We define
$$
\det_{\tilde{n} \times \tilde{n}} A:=(\det A_{I, J})_{I \in \calJ(\tilde{n}, n), J \in \calJ(\tilde{n}, q)}, 
$$
so that, in particular, $\det\limits_{\tilde{n} \times \tilde{n}} A$ is a \emph{vector} (here, the order of components does not matter). $\det\limits_{\tilde{n} \times \tilde{n}} A$ comes up when one changes variables. Indeed, suppose $\Phi$ is a $C^1$ diffeomorphism from an open subset $U \subset \R^{\tilde{n}}$ mapping to an $\tilde{n}$ dimensional submanifold of $\R^n$, where this submanifold is given the induced Lebesgue measure $dx$. Then, we have\footnote{\eqref{changevab01} is known as the Cauchy-Binet formula.}
\begin{equation} \label{changevab01}
\int_{\Phi(U)} f(x)dx=\int_U f(\Phi(t)) \left| \det\limits_{\tilde{n} \times \tilde{n}} d\Phi(t) \right| dt.
\end{equation}

Let $Z_1, \dots, Z_q$ be a list of $C^\infty$ vector fields defined on an open set $\Omega \subset \R^n$, with associated formal degrees $\tilde{d}_1, \dots, \tilde{d}_q \in [1, \infty)$. Let further $\Omega' \Subset \Omega$, where $\Omega'$ is open, relatively compact in $\Omega$.  

Fix $x_0 \in \Omega$. We suppose $n=\dim \  \textrm{span} \left\{ Z_1(x_0), \dots, Z_q(x_0)\right\}$. For $J=(j_1, \dots, j_n) \in \calJ(n, q)$, let $Z_J$ denote the list of vector fields $Z_{j_1}, \dots, Z_{j_n}$ (meanwhile, we denote $(Z, \tilde{d})_J$ be the list of vector fields $Z_{j_1}, \dots, Z_{j_n}$, together with the formal degrees $\tilde{d}_{j_1}, \dots, \tilde{d}_{j_n}$). Fix $J_0 \in \calJ(n, q)$ such that $\left| \det Z_{J_0}(x_0) \right|=\left| \det_{n \times n} Z(x_0) \right|_\infty$, where we have identified $Z(x_0)$ with the $n \times q$ matrix whose columns are given by $Z_1(x_0), \dots, Z_q(x_0)$ and similarly for $Z_{J_0}(x_0)$. We assume that
\begin{enumerate}
\item [$\bullet$] $(Z, \tilde{d})$ satisfies $\calC(x_0, \Omega')$. 

\item [$\bullet$] For $1 \le i, j \le q$, 
\begin{equation} \label{190308eq01}
[Z_i, Z_j]=\sum_{l=1}^\infty c_{i, j}^l Z_l, \quad c_{i, j}^l \in C^\infty;
\end{equation}
\item [$\bullet$] For each $m$, 
\begin{equation} \label{adconst01}
\|Z_j\|_{C^m(B_{(Z, \tilde{d})}(x_0, 1))}<\infty;
\end{equation}
\item [$\bullet$] For every $m$ and every $i, j, l$, 
\begin{equation} \label{adconst02}
\sum_{|\alpha| \le m} \left\| Z^\alpha c_{i, j}^l \right\|_{C^0(B_{(Z, \tilde{d})}(x_0, 1))}<\infty; 
\end{equation}
where $\alpha$ is an ordered multi-index. 
\end{enumerate}

\begin{defn} \label{admiconstdefn}
We say that $C$ is an $m$-\emph{admissible constant} if $C$ can be chosen to depend only on upper bounds for \eqref{adconst01} and \eqref{adconst02} (for that particular choice of $m$), $m$, upper bounds for $\tilde{d}_1, \dots, \tilde{d}_q$ and an upper bound for $n$ and $q$. We write $A \lesssim_m B$ if $A \le CB$, where $C$ is an $m$-admissible constant, and we wirte $A \simeq_m B$ if $A \lesssim_m B$ and $B \lesssim_m A$. 

Finally, we say $C=C(\sigma)$ for some $\sigma>0$, is an $m$-\emph{admissible constant} if $C$ can be chosen to depend on all the parameters an $m$ admissible constant may depend on, and $C$ may also depend on $\sigma$. 
\end{defn}

Now we introduce the quantitative version of the scaling maps, which can also be thought as a quantitative version of the theorem of Frobenius. 

\begin{thm}[{\cite[Section 4]{BS2}, \cite[Theorem 11.1]{BS}}] \label{scalingmap}
There exist $2$-admissible constants $\eta_1, \zeta_1>0$, such that if the map $\Phi:B^n(\eta_1) \rightarrow B_{(Z, \tilde{d})}(x_0, 1)$ is defined by $\Phi(u)=e^{u \cdot Z_{J_0}}x_0$, we have
\begin{enumerate}
\item [$\bullet$]$\Phi: B^n(\eta_1) \rightarrow  B_{(Z, \tilde{d})}(x_0, 1)$ is injective;
\item [$\bullet$] $B_{(Z, \tilde{d})}(x_0, \zeta_1) \subset \Phi(B^n(\eta_1))$.
\end{enumerate}
Furthermore, if we let $Y_j$ be the pullback of $Z_j$ under the map $\Phi$, then we have, for $m \ge 0$, 
$$
\|Y_j\|_{C^m(B^n(\eta_1))} \lesssim_{\max\{m, 2\}} 1
$$
and
$$
\|f\|_{C^m(B^n(\eta_1))} \simeq_{\max\{m-1, 2\}} \sum_{|\alpha| \le m} \|Y^\alpha f\|_{C^0(B^n(\eta_1))}. 
$$
Finally, 
$$
\left| \det_{n \times n} Y(u) \right| \simeq_2 1, \quad \textrm{for all} \ u \in B^n(\eta_1).
$$
\end{thm}

Note that by the definition of Carnot-Carath\'eodory metric (see Definition \ref{defn01}), it is clear that
\begin{equation} \label{1207eq01}
\rho_Y (0, u)=\rho_X (x_0, \Phi(u)), \ u \in B^n(\eta_1),
\end{equation}
where we write $\rho_X$ ($\rho_Y$ respectively) to be the Carnot-Carath\'eodory metric induced by the list $(X, d)$. 

\begin{rem}
Recently, Stovall and Street improved Theorem \ref{scalingmap} (see, \cite{MR3881835, BNAO1, BNAO2}). 
\end{rem}

Here are some corollaries we need in the sequel. 

\begin{cor}[{\cite[Theorem 4.1, Corollary 4.2]{BS2}}] \label{scalingmap01}
Let $\eta_1$, $\zeta_1$ and $\Phi$ be as in Theorem \ref{scalingmap}. Then, there exist admissible constants $0<\eta_2<\eta_1$, $0<\zeta_4<\zeta_3<\zeta_1$ such that
\begin{eqnarray*}
B_{(Z, \tilde{d})}(x_0, \zeta_4)%
&\subseteq& B_{(Z, \tilde{d})}(x_0, \zeta_3) \subseteq \Phi( B^n(\eta_2)))\\
&\subseteq& B_{(Z, \tilde{d})_{J_0}}(x_0, \zeta_1) \subseteq B_{(Z, \tilde{d})}(x_0, \zeta_1) \subseteq B_{(Z, \tilde{d})_{J_0}}(x_0, \zeta_2)\\
&\subseteq& \Phi(B^n(\eta_1)) \subseteq B_{(Z, \tilde{d})_{J_0}}(x_0, 1) \subseteq B_{(Z, \tilde{d})}(x_0, 1). 
\end{eqnarray*}
Moreover\footnote{Note that the particular value of $\eta_1$ and $\zeta_1$ in \eqref{changevab233} are not so important. Indeed,  by using \eqref{doublingineq} and scaling, $\eta_1$ and $\xi_1$ can be replaced by any other admissible constants, while the change of the implict constants in \eqref{changevab233} only depends on the change of the admissible constants.}, for all $u \in B^n(\eta_1)$, 
\begin{equation} \label{changevab233}
\left| \det_{n \times n} d\Phi(u) \right| \simeq \left| \det_{n \times n} Z(x_0) \right| \simeq\textrm{Vol} (B_{(Z, \tilde{d})}(x_0, \zeta_1)).
\end{equation}
\end{cor}

\begin{cor}[{\cite[Proposition 11.2]{BS}}]
Suppose $\zeta'_2, \eta'_2>0$ are given and $\Phi$ defined in Theorem \ref{scalingmap}. Then there exists $2$-admissible constants $\eta'=\eta'(\zeta'_2)>0$, $\zeta_2'=\zeta_2'(\eta'_2)>0$ such that
$$
\Phi(B^n(\eta')) \subseteq B_{(Z, \tilde{d})}(x_0, \zeta'_2), \quad B_{(Z, \tilde{d})}(x_0, \zeta') \subseteq \Phi(B^n(\eta'_2)).
$$
\end{cor}

We let $ \eta_1,  \zeta_1$ and $\Phi: B^n(\eta_1) \longrightarrow B_{(Z, \tilde{d})}(x_0, 1)$ be as in Theorem \ref{scalingmap}, and $\Omega'' \Subset \Omega' \Subset \Omega$, where $\Omega'' \subset \R^n$ is open and relatively compact in $\Omega'$. Furthermore, we let $Y_1, \dots, Y_q$ be the pullbacks of $Z_1, \dots, Z_q$ as in Theorem \ref{scalingmap} and $\gamma: B^k(\rho) \times \Omega' \to \Omega$ be a $C^\infty$ function, where $\rho>0$ is some fixed positive number, and $\gamma(0, x) \equiv x$. Here $\rho>0$ is small enough that for $t \in B^k(\rho)$, $\gamma_t^{-1}$ exists. In order to work the the assumption that $(Z, \tilde{d})$ controls $\gamma_t$, we recall the following two conditions on $\gamma_t$. 

\medskip

\begin{enumerate}
\item [1.] $\calQ_1( \rho_1, \tau_1, \{\sigma_1^m\}_{m \in \N}) (\rho_1 \le \rho, \tau_1 \le \xi_1):$ For $x \in \Omega''$, define the vector field
$$
W(t, x)=\frac{d}{d\epsilon} \bigg|_{\epsilon=1} \gamma_{\epsilon t} \circ \gamma_t^{-1}(x).
$$
We suppose
\begin{enumerate}
\item [$\bullet$] $W(t, x)=\sum\limits_{l=1}^q c_l(t, x)Z_l(x)$, on $B_{(Z, \tilde{d})}(x_0, \tau_1)$, 
\item [$\bullet$] $\sum\limits_{|\alpha|+|\beta| \le m} \|Z^\alpha \partial_t^\beta c_l\|_{C^0(B^k(\rho_1) \times B_{(Z, \tilde{d})}(x_0, \tau_1))} \le \sigma_1^m$.
\item [$\bullet$] Note that we may, without the loss of generality, assume that $c_l(0, x)\equiv 0$, as we may replace $c_l(t, x)$ with $c_l(t, x)-c_l(0, x)$ for every $l$ by using the fact that $W(0, x)=0$.
\end{enumerate}

\item [2.]$\calQ_2(\rho_2, \tau_2, \{\sigma_2^m\}_{m \in \N}):$
\begin{enumerate}
\item [$\bullet$] $\gamma(B^k(\rho_2) \times B_{(Z, \tilde{d})}(x_0, \tau_2)) \subseteq B_{(Z, \tilde{d})}(x_0, \zeta_1)$;
\item [$\bullet$] If $\eta'=\eta'(\tau_2)>0$ is a $2$-admissible constant so small that 
$$
\Phi(B^n(\eta')) \subseteq B_{(Z, \tilde{d})}(x_0, \tau_2) \subseteq \Phi(B^n(\eta_1)),
$$
then if we define a new map
$$
\Theta_t(u)=\Phi^{-1} \circ \gamma_t \circ \Phi(u): B^k(\rho_2) \times B^n(\eta') \mapsto B^n(\eta_1),
$$
we have $\|\Theta\|_{C^m(B^k(\rho_2) \times B^n(\eta'))} \le \sigma_2^m$.
\end{enumerate}
\end{enumerate}

\begin{prop}[{\cite[Proposition 12.3]{BS}}] \label{controlQ1Q2}
$\calQ_1 \Leftrightarrow \calQ_2$ in the following sense: 
\begin{enumerate}
\item [$\bullet$] $\calQ(\rho_1, \tau_1, \{\sigma_1^m\}_{m \in \N}) \Longrightarrow$ there exists a $2$-admissible constant
$$
\rho_2(\rho_1, \tau_1, \sigma_1^1, k)
$$
and $m+1$-admissible constants $\sigma_2^m=\sigma_2^m(\sigma_1^{m+1}, k)$ such that
$$
\calQ_2(\rho_2, \tau_1/2, \{\sigma_2^m\}_{m \in \N})
$$
holds.

\item [$\bullet$] $\calQ_2(\rho_2, \tau_2, \{\sigma_2^m\}_{m \in \N}) \Longrightarrow$ there exists a $2$-admissible constant $\tau_1=\tau_1(\tau_2)>0$ and $m$-admissible constants 
$$
\sigma_1^m=\sigma_1^m(\sigma_2^{m+1}, k)
$$ 
such that 
$$
\calQ_1(\rho_2, \tau_1, \{\sigma_1^m\}_{m \in \N})
$$ 
holds. 
\end{enumerate}
\end{prop}

\begin{defn}
We say \emph{$(Z, \tilde{d})$ controls $\gamma$ at the unit scale} if either of the equivalent conditions $\calQ_1$ or $\calQ_2$ holds (for some choice of parameters). If we wish to make the point $x_0$ explicit, we wil say $(Z, \tilde{d})$ controls $\gamma$ at the unit scale near $x_0$. 
\end{defn}

Here are some basic properties for this concept.

\begin{prop}[{\cite[Propostion 12.6, Proposition 12.7]{BS}}] \label{controlcor}
\begin{enumerate}
\item[1.] If $(Z, \tilde{d})$ controls $\gamma_{t_1}^1$ and $\gamma_{t_2}^2$ at the unit scale, then $(Z, \tilde{d})$ controls $\gamma^1_{t_1} \circ \gamma^2_{t_2}$ at the unit scale;
\item [2.] If $(Z, \tilde{d})$ controls $\gamma_t$ at the unit scale, then $(Z, \tilde{d})$ controls $\gamma_t^{-1}$ at the unit scale;
\item [3.] If $(Z, \tilde{d})$ controls $\gamma_t$ at the unit scale and $c:=(c_1, \dots, c_k) \in [0, 1]^k$ is a constant, then $(Z, \tilde{d})$ controls $\gamma^c((t_1, \dots, t_k), x):=\gamma((c_1t_1, \dots, c_kt_k), x)$ at the unit scale. Moreover, the parameters in the definition of control at the unit scale may be chosen indepedent of $c$. 
\end{enumerate}
\end{prop}

Here, we include the proofs for these results, which are contained in \cite{BS}.

\begin{proof}
All of these statements can be easily verified by using $\calQ_2$. Indeed, the $C^m$ norm of $\Phi^{-1} \circ \gamma_{t_1}^1 \circ \gamma_{t_2}^2 \circ \Phi= \left( \Phi^{-1} \circ \gamma_{t_1}^1 \circ \Phi \right)  \circ \left( \Phi^{-1} \circ \gamma_{t_2}^2 \circ \Phi \right)$ can clearly be bounded in terms of the $C^m$ norms of $ \Phi^{-1} \circ \gamma_{t_1}^1 \circ \Phi$ and $ \Phi^{-1} \circ \gamma_{t_2}^2 \circ \Phi$, provided one shrinks the parameters $\rho_2$ and $\tau_2$ appropriately (so that the composition is defined). Smilar proofs work for $\gamma_t^{-1}$ and $\gamma^c_t$. We leave further details to the interested readers.   
\end{proof}


\subsection{The main lemma}

\medskip

We are now ready to describe the assumptions for the lemma of modulus of continuity, which generalizes the inequality \eqref{hilbertCV} under the setting of sub-Riemannian geometry. We begin with the following definition.

\begin{defn} \label{Mgenerate}
Let $\{Z_1, \dots, Z_r\}$ be a subset of $\{Z_1, \dots, Z_q\}$. For $M \in \N$, we say $Z_1, \dots, Z_r$ \emph{$M$-generates} $Z_1, \dots, Z_q$ if each $Z_j (r+1 \le j \le q)$ can be written in the form 
$$
Z_j=\textrm{ad} (Z_{l_1}) \textrm{ad} (Z_{l_2}) \dots \textrm{ad} (Z_{l_m}) Z_{l_{m+1}}, \quad 0 \le m \le M-1, \ l \le l_k \le r.   
$$
\end{defn}

Throughout this section, we take $K_1 \Subset K_0' \Subset K_0 \Subset \Omega'' \Subset \Omega' \Subset \Omega$, where $\Omega \subset \R^n$ is open, $K_1$, $K_0$, $K_0'$ compact, $K_1 \subsetneq K_0' \subsetneq K_0$, and $\Omega'', \Omega$ open, relatively compact in $\Omega$ as above. Let
$$
Z_1, \dots, Z_q
$$
be a list of $C^\infty$ vector fields on $\Omega$ with formal degrees $\tilde{d}_1, \dots, \tilde{d}_q$, satisfying the following assumptions:
\begin{enumerate} 
\item [(a).] $n=\dim  \ \textrm{span} \left\{ Z_1(x), \dots, Z_q(x) \right\}$, for any $x \in \Omega'$;

\medskip

\item [(b).] $(Z, \tilde{d})$ satisfies all the assumptions of Theorem \ref{scalingmap}, uniformly for $x \in K_0$; 

\medskip

\item [(c).] The first $r$ vector fields $Z_1, \dots, Z_r$, $M$-generate $Z_1, \dots, Z_q$, for some $M>0$;

\medskip

\item[(d).] There exists some $\check{C}>1$, such that
$$
B_{(Z, \tilde{d})}(x_0, \check{C}) \subsetneq K_0', \ \forall x_0 \in K_1. 
$$

\end{enumerate}

Moreover, we assume that we are given a $C^\infty$ function $\check{\gamma}:  B^k(a') \times \Omega'' \mapsto \Omega'$ , satisfying the following assumptions:
\begin{enumerate}
\item [(I).] $\check{\gamma}(0, x) \equiv x, x \in \Omega''$;

\medskip

\item [(II).] $\check{\gamma}$ is controlled by $(Z, \tilde{d})$ near $x$ for every $x \in K_0$, uniformly in $x \in K_0$, that is, the parameters in $\calQ_1$ can be choosen uniformly for all $x \in K_0$;

\medskip

\item [(III).]  For each $l$, $1 \le l \le r$, there is a multi-index $\alpha$ (with $|\alpha| \le B$, where $B \in \N$ is some fixed constant which our results are allowed to depended on), such that
\begin{equation} \label{eq3455}
Z_l(x)=\frac{1}{\alpha!} \frac{\partial}{\partial t}^\alpha \bigg|_{t=0} \frac{d}{d\epsilon} \bigg|_{\epsilon=1} \check{\gamma}_{\epsilon t} \circ \check{\gamma}_t^{-1}(x).
\end{equation}
Note that (III) together with (a) and (c) above implies that $\check{\gamma}$ satisfies the curvature conditon $(\calC_J)$ (see, e.g.,  \cite[Theorem 9.1]{BS}). 

\end{enumerate}

We make a remark that in the above assumptions, we shall restrict our attention to $a>0$ small,  so that $\check{\gamma}_t^{-1}$ makes sense (since $\check{\gamma}_0(x) \equiv x$) wherever we use it. 

We make the an observation first, to complete our assumptions on the main lemma. 

\medskip

\textit{Observation I:} Let $x_0 \in  K_1$. Since $\check{\gamma}$ is controlled  by $(Z, \tilde{d})$ at the unit scale, by Proposition \ref{controlQ1Q2} (more precisely, by $\calQ_2$), we can take some $\check{C_1}>\check{C}$ sufficiently large and $0<a'<a$ sufficiently small, such that
$$
\check{\gamma}\left(\overline{B^k(a')} \times \overline{B_{(Z, \tilde{d})}\left(x_0, \check{C}\right)}\right) \subseteq B_{(Z, \tilde{d})}\left(x_0, \check{C_1}\right),
$$
where $a$ and $\check{C_1}$ only depend on $\check{C}$ and how $\check{\gamma}$ is controlled by $(Z, \tilde{d})$ (more precisely, the constants in $\calQ_1$). 

This allows us to state the last assumption on the list of vector fields $(Z, \tilde{d})$: 

\medskip

\begin{enumerate}

\item [(e).] $B_{(Z, \tilde{d})}\left(x,  10\check{C_1} \right) \subseteq K_0, \forall x \in K_1$. 

\end{enumerate}

\begin{rem}
Notice that (e) requires $Z_1, \dots, Z_q$ to be ``small". 
\end{rem}

\medskip

\textit{Observation II:} Apply the scaling map $\Phi$ defined in Theorem \ref{scalingmap} to $x_0 \in K_1$ taken above and the list
$$
\left( \left( 8 \check{C_1} \right)^{\tilde{d}}Z, \tilde{d} \right),
$$
where all the implicit constants derived from this scaling map only depend on $\check{C_1}$ and any $2$-admissible constants with respect to the list $(Z, \widetilde{d})$. In particular, there exists a $0<\zeta_0<1$, only depending on $\check{C_1}$ and any $2$-admissible constants, such that
\begin{equation} \label{0930eq01}
\Phi: B_{(Y, \tilde{d})}(0, \zeta_0) \longrightarrow B_{\left( \left(8\check{C_1} \right)^{\tilde{d}}Z, \tilde{d} \right)}(x_0, \zeta_0)=B_{(Z, \tilde{d})} \left(x_0, 8\check{C_1} \zeta_0 \right) \subset K_0
\end{equation}
is a $C^\infty$ diffeomorphism, where $(Y, \tilde{d})$ is the pullback of the list $\left( \left( 8\check{C_1} \right)^{\tilde{d}}Z, \tilde{d} \right)$ via $\Phi$. 

We are ready to state the lemma of modulus of continuity.

\begin{lem} \label{moduluscont}
Let $\chi \in C^\infty_0(B^k(a''))$, $\psi_1, \psi_2 \in C_0^\infty(\R^n)$ with compact support in the interior of $K_1$  and $f \in L^2(\R^n)$.  Let further, $\tilde{k} \ge 1$ and 
$$
\theta(t, x): B^{\tilde{k}}(1) \times B_{(Z, \tilde{d})}\left(x_0, 5\check{C_1} \zeta_0 \right) \mapsto B_{(Z, \tilde{d})} (x_0, 6 \check{C_1} \zeta_0 ), 
$$ 
with satisfying the following conditions\footnote{The domain of $\theta$ on the variable $t$ is not quite important, as we only requires $t$ to be small in our proof. In our applications later, $\widetilde{k}=q$, and $B^{\widetilde{k}}(1)$ will be replaced by  $B^q(\eta_0)$, where $\eta_0$ is some $2$-admissible constant, which comes from the proof of sparse bound later. There is no harm to pretend the domain of $\theta$ on the variable $t$ is $B^{\widetilde{k}}(1)$ at this moment. }:
\begin{enumerate}
\item [(1).] $\theta(0, x) \equiv x, \forall x \in B_{(Z, \tilde{d})}(x_0, 5\check{C_1} \zeta_0)$; 
\item [(2).] $\theta \in C^\infty (B^{\tilde{k}}(1) \times B_{(Z, \tilde{d})}(x_0, 5\check{C_1} \zeta_0))$;
\item [(3).] $\theta$ is controlled at the unit scale by the list $(Z, \tilde{d})$ near $x \in B_{(Z, \tilde{d})}\left(x_0, 4\check{C_1} \zeta_0 \right)$, uniformly in $x$; 
\item [(4).] For any $b \in B^{\tilde{k}}(1)$, the map $\theta_b (\cdot):=\theta(b, \cdot)$ has an inverse, which maps $\theta_b(B_{(Z, d)}(x_0, 5\check{C_1} \zeta_0))$ back to $B_{(Z, d)}(x_0, 5\check{C_1} \zeta_0)$. 
\end{enumerate}
Consider \footnote{Note that, under these assumptions, 
$$
\|\calL\|_{L^\infty \mapsto L^\infty}, \|\calL\|_{L^1 \mapsto L^1} \lesssim 1,
$$
and furthermore $\calL^*$ is of the same form as $\calL$ with $\check{\gamma}_t$ replaced by $\check{\gamma}_t^{-1}$. }
$$
\calL f(x)=\psi_1(x) \int f(\check{\gamma}_t(x)) \psi_2(\check{\gamma}_t(x)) \varrho(t, x) \chi(t)dt ,
$$
where $\varrho$ is a $C^\infty$ function with
$$
\varrho: B^k(a) \times \Omega' \mapsto \C. 
$$
If  $a''$ and $|b|$ are chosen sufficiently small\footnote{In our applications later, $b$ will be of the form $\del^j \tilde{b}$, where $j \ge 0$ is sufficiently large and $\tilde{b} \in B^n(1)$.}, then there exists $\widetilde{C}, \eta>0$, such that
\begin{equation} \label{1001eq01}
\left\| \calL f (\cdot)-\left( \calL f \right) \left(\theta_{b}(\cdot) \right) \right\|_{L^2(B_{(Z, \tilde{d})}(x_0, \check{C}\zeta_0)))} \le \widetilde{C} |b|^{\eta} \|f\|_{L^2(B_{(Z, \tilde{d})}(x_0, \check{C_1}\zeta_0))}.
\end{equation}
\end{lem}

\begin{rem}
It is important that $a'', \widetilde{C}$, $\eta$ and the upper bound for all suitable choices of ``$|b|$" may be chosen to depend only on certain parameters which are independent of other relevant parameters. More precisely, $a, \widetilde{C}$, $\eta$ and the upper bound for all suitable choices of ``$|b|$"  can be chosen to depend only 
\begin{enumerate}
\item[1.] on $\check{C}$;
\item[2.] on the norms of various functions used to define $\calL$;
\item[3.] on how $\theta$ is controlled\footnote{In our applications later, we will see that although ``$\theta$"  may vary from different scales $\del^j, j \ge 0$ and different centers $x_c(Q), Q \in \calG$, all the parameters in $\calQ_1$ of all these ``$\theta$"s are controlled uniformly by the collection of vector fields $X_1, \dots, X_q$. We will come back to this point in Section 5.}, namely, the constants in $\calQ_1$ with respect to $\theta$;
\item[4.] on the $C^m$ norms of the cut-off functions taken and fixed in the proof;
\item [5.] on parameters $B$ and $M$;
\item [6.] on various dimensions;
\item [7.] on the parameters in $\calQ_1$ with respect to $\gamma$;
\item [8.] on anything that $2$-admissible constants were allowed to depend. 
\end{enumerate}
The reader wishing to, should have no trouble keeping track of the various dependencies in our argument. 
\end{rem}

\begin{proof} [Proof of Lemma \ref{moduluscont}]
First, since both $\check{\gamma}$ and $\theta$ are controlled at the unit scale by $(Z, \tilde{d})$, by Proposition \ref{controlQ1Q2}, so is $ \theta_s \circ \check{\gamma}_t$. Therefore, by $\calQ_2$, we see that for $a$ and $|b|$ sufficiently small, we have
$$
 \check{\gamma} \left( B^k(a) \times B_{(Z, \tilde{d})}(x_0, \check{C}\zeta_0) \right) \subseteq B_{(Z,\tilde{d})} \left(x_0, \check{C_1}\zeta_0 \right)
$$
and 
$$   
(\theta_b \circ \check{\gamma}) \left( B^k(a) \times B_{(Z, \tilde{d})}(x_0, \check{C}\zeta_0) \right) \subseteq B_{(Z,\tilde{d})} \left(x_0, \check{C_1}\zeta_0 \right),
$$
which implies in the inequality \eqref{1001eq01} we wish to prove, only those values of $f$ on $ B_{(Z,\tilde{d})} \left(x_0, \check{C_1}\zeta_0 \right)$ contribute to the left hand side of \eqref{1001eq01}. Therefore, it suffices to consider the case $f$ is supported in $ B_{(Z,\tilde{d})} \left(x_0, \check{C_1}\zeta_0 \right)$, otherwise, we can replace $f$ by $f \one_{ B_{(Z,\tilde{d})} \left(x_0, \check{C_1}\zeta_0 \right)}$.

Now we let $f$ to be an $L^2$ function, supported in $B_{(Z,\tilde{d})} \left(x_0, \check{C_1}\zeta_0 \right)$. Write $\calL f(x)=\psi_1(x) \calI(f)(x)$, where
$$
\calI(f)(x):= \int f(\check{\gamma}_t(x)) \psi_2(\check{\gamma}_t(x)) \varrho(t, x) \chi(t)dt. 
$$
and put $\widetilde{\calI}(f)(x):=\calI(f)(\theta_{b}(x))$. Therefore, 
$$
\calL f(x)-\left( \calL f \right) \left(\theta_{b}(x) \right)=\psi_1(x) \calI(f)(x)-\psi_1(\theta_{b}(x)) \widetilde{\calI}(f)(x).
$$
Since both $\calL$ and $(\psi_1 \circ \theta_{b}) \widetilde{\calI}$ are bounded trivially on $L^2$, by triangle inequality and the fact that $\psi_1 \in C^\infty_0$, it suffices to show
\begin{equation} \label{modeq01}
\left( \int_{B_{(Z, \tilde{d})}(x_0, \check{C}\zeta_0)} \left| \calL(f)(x)-\widetilde{\calL}(f)(x) \right|^2dx \right)^{\frac{1}{2}}\lesssim |b|^\eta \|f\|_{L^2(B_{(Z, \tilde{d})}(x_0, \check{C_1} \zeta_0))}.
\end{equation}
where $\widetilde{\calL}(f):=\psi_1 \widetilde{\calI}(f)$.

By \eqref{0930eq01}, we see that for any $x \in B_{(Z, \tilde{d})}(x_0,\check{C}\zeta_0)$, we can find $u \in B_{(Y, \tilde{d})}(0, \zeta_0)$ such that $x=\Phi(u)$, and hence
\begin{eqnarray*}
&&\calL(f)(x)= \calL(f)(\Phi(u))\\
&&= \psi_1(\Phi(u)) \int f(\check{\gamma}_t \circ \Phi(u)) \psi_2(\check{\gamma}_t \circ \Phi(u))  \varrho (t, \Phi(u)) \chi(t)dt \\
&&= (\psi_1 \circ \Phi)  (u) \int (f \circ \Phi) \circ (\Phi^{-1} \circ \check{\gamma}_t \circ \Phi) (u)  (\psi_2 \circ \Phi) \circ (\Phi^{-1} \circ \check{\gamma}_t \circ \Phi) (u)  \\
&& \quad \quad \quad \quad \quad \quad \quad \quad \quad \quad \quad  \varrho (t, \Phi(u)) \chi(t)dt 
\end{eqnarray*}
and 
\begin{eqnarray*}
&& \widetilde{\calL}(f)(x)=\psi_1(x) \widetilde{\calI}(f)(x)=\psi_1(\Phi(u)) \calI(f) (\theta_{b} \circ \Phi(u))\\
&& =\psi_1(\Phi(u)) \int f(\check{\gamma}_t \circ \theta_{b} \circ \Phi(u)) \psi_2(\check{\gamma}_t \circ \theta_{b} \circ \Phi(u))  \varrho (t, \theta_{ b} \circ \Phi(u)) \chi(t)dt \\
&& = (\psi_1 \circ \Phi)(u) \int (f \circ \Phi) \circ (\Phi^{-1} \circ \check{\gamma}_t \circ \Phi) \circ (\Phi^{-1} \circ \theta_{b} \circ \Phi)(u)\\
&&  \quad \quad \quad \quad \quad \quad \quad \quad   (\psi_2 \circ \Phi) \circ (\Phi^{-1} \circ \check{\gamma}_t \circ \Phi) \circ (\Phi^{-1} \circ \theta_{b} \circ \Phi)(u) \\
&& \quad \quad \quad \quad \quad \quad \quad  \quad \quad  \quad \quad  \quad \quad  \varrho (t, \Phi \circ (\Phi^{-1} \circ \theta_{b} \circ \Phi)(u)) \chi(t)dt
\end{eqnarray*}
Put
$$
\widehat{\psi_1}:=\psi_1 \circ \Phi,  \quad  \widehat{\psi_2}:=\psi_2 \circ \Phi, 
$$
$$
g:=f \circ \Phi,  \quad \widehat{\varrho}(t, u):=\varrho(t, x)=\varrho (t, \Phi(u)), 
$$
$$
\widehat{\gamma}_t:=\Phi^{-1} \circ \check{\gamma}_t \circ \Phi, \quad \widehat{\theta}_{b}:=\Phi^{-1} \circ \theta_{b} \circ \Phi,
$$
and denote
$$
\calJ(g)(u):=\widehat{\psi_1}(u) \int g(\widehat{\gamma}_t (u)) \widehat{\psi_2}(\widehat{\gamma}_t(u)) \widehat{\varrho}(t, u) \chi(t)dt
$$
and
$$
\widetilde{\calJ}(g)(u):=\widehat{\psi_1}(u) \int g(\widehat{\gamma}_t \circ \widehat{\theta}_{b}(u)) \widehat{\psi_2}(\widehat{\gamma}_t \circ \widehat{\theta}_{b}(u)) \widehat{\varrho} (t, \widehat{\theta}_{b}(u)) \chi(t)dt.
$$
It is clear that  $g \in L^2(\R^n)$ with $\textrm{supp} g$ belonging to the interior of $B_{(Y, \tilde{d})} \left(0, \frac{\zeta_0}{8} \right)$ and 
\begin{equation} \label{20180902}
\calJ(g)(u)=\calL(f)(x) \ \textrm{and} \ \widetilde{\calJ}(g)(u)=\widetilde{\calL}(f)(x),
\end{equation}
where $x=\Phi(u)$. Recall that we wish to show there exists some $\eta>0$, 
\begin{equation} \label{modeq021}
\|\calL f-\widetilde{\calL}f\|_{L^2(B_{(Z, \tilde{d})}(x_0, \check{C}\zeta_0))} \lesssim |b|^\eta \|f\|_{L^2(B_{(Z, \tilde{d})}(x_0, \check{C_1}\zeta_0))}
\end{equation}
Changing the variable with $x=\Phi(u)$, together with the fact that
\begin{equation} \label{changevab}
\left| \det d\Phi(u) \right| \simeq \left| \det_{n \times n} Z(x_0) \right| \simeq \textrm{Vol}(B_{(Z, \tilde{d})}(x_0, \check{C_1}\zeta_0)), \quad u \in B_{(Y, \tilde{d})}(0, \zeta_0), 
\end{equation}
(see, Corollary \ref{scalingmap01}), it suffices to show there exists some $\eta>0$, such that
\begin{equation} \label{modeq02}
\|\calJ g-\widetilde{\calJ}g\|_{L^2\left(B_{(Y, \tilde{d})}\left(0, \frac{\check{C}\zeta_0}{8\check{C_1}}\right)\right)} \lesssim |b|^\eta \|g\|_{L^2\left(B_{(Y, \tilde{d})}\left(0, \frac{\zeta_0}{8}\right)\right)}.
\end{equation} 
Indeed, if \eqref{modeq02} holds, then we have 
\begin{eqnarray*}
&&\|\calL f-\widetilde{\calL}f\|_{L^2(B_{(Z, \tilde{d})}(x_0, \check{C}\zeta_0))}^2=\int_{B_{(Z, \tilde{d})}(x_0, \check{C}\zeta_0)} \left|\calL (f)(x)-\widetilde{\calL}(f)(x) \right|^2dx\\
&&=\int_{\Phi^{-1}(B_{(Z, \tilde{d})}(x_0, \check{C}\zeta_0))} \left|\calL (f)(\Phi(u))-\widetilde{\calL}(f)(\Phi(u)) \right|^2 \left| \det d\Phi(u) \right|du \\
&& \quad \quad (\textrm{change variables with} \ x=\Phi(u) \ \textrm{with applying} \  \eqref{changevab01}.) \\
&& \simeq \textrm{Vol}(B_{(Z, \tilde{d})}(x_0, \check{C_1}\zeta_0)) \cdot \int_{\Phi^{-1}(B_{(Z, \tilde{d})}(x_0, \check{C}\zeta_0))} \left|\calJ(g)(u)-\widetilde{\calJ}(g)(u) \right|^2 du\\
&& \quad \quad (\textrm{by} \ \eqref{20180902} \ \textrm{and} \ \eqref{changevab}.) \\ 
&& =  \textrm{Vol}(B_{(Z, \tilde{d})}(x_0, \check{C_1}\zeta_0)) \cdot \int_{B_{(Y, \tilde{d})}\left(0, \frac{\check{C}\zeta_0}{8\check{C_1}}\right)} \left|\calJ(g)(u)-\widetilde{\calJ}(g)(u) \right|^2 du\\
&& \lesssim |b|^{2\eta} \cdot \textrm{Vol}(B_{(Z, \tilde{d})}(x_0, \check{C_1}\zeta_0)) \cdot \int_{B_{(Y, \tilde{d})}\left(0, \frac{\zeta_0}{8}\right)} |g(u)|^2 du \\
&&  \quad \quad (\textrm{by} \ \eqref{modeq02}.)\\
&& \simeq |b|^{2\eta} \int_{B_{(Z, \tilde{d})}(x_0, \check{C_1}\zeta_0)} |f(x)|^2 dx=|b|^{2\eta} \|f\|^2_{L^2(B_{(Z, \tilde{d})}(x_0, \check{C}\zeta_0))}, \\
&& \quad \quad (\textrm{by} \ \eqref{changevab01} \ \textrm{and} \ \eqref{changevab} \ \textrm{again}.)
\end{eqnarray*}
which implies \eqref{modeq021}. Moreover, the following facts are easy to see: 
\begin{enumerate}
\item [1.]  $\widehat{\psi_1}, \widehat{\psi_2} \in C^\infty(B_{(Y, \tilde{d})}(0, \zeta_0))$. More precisely, the $C^m$-norm of $\widehat{\psi_1}$ ($\widehat{\psi_2}$, respectively) depends on the $C^m$-norm of $\psi_1$ ($\psi_2$, respectively) and any $m$-admissible constants; 
\item [2.] $\widehat{\varrho} \in C^\infty(\overline{B^k(a)} \times \overline{B_{(Y, \tilde{d})}(0, \zeta_0)})$.  More precisely, the $C^m$-norm of $\widehat{\rho}$ only depends on the $C^m$-norm of $\rho$ and any $m$-admissible constants; 
\item [3.] $\widehat{\gamma}$ satisfies the curvature condition $(\calC_J)$, uniformly with respect to some parameters. More precisely, using $\calQ_2$, we see that $\widehat{\gamma}$ is a $C^\infty$ mapping from $B^k(a_0) \times B_{(Y, \tilde{d})} \left(0, \frac{3\zeta_0}{4} \right)$ to $B_{(Y, \tilde{d})}(0, \zeta_0)$ for some $a_0>0$ which only depends on $\check{C_1}$ and any $2$-admissible constants, that is, for each $m \ge 0$, 
$$
\|\widehat{\gamma}\|_{C^m\left(B^k(a_0) \times B_{(Y, \tilde{d})} \left(0, \frac{3\zeta_0}{4} \right)\right)} \le C(m),
$$
where $C(m)$ only depends on $\check{C_1}$,  any $m$-admissible constants and also how $Z$ controls $\check{\gamma}$ at the unit scale (more precisely, $\calQ_2$). Moreover, since $\check{\gamma}$ satisfies $(\calC_J)$, an application of Theorem \ref{scalingmap} and Theorem \ref{apdthm002} yields $\widehat{\gamma}$ satisfies the uniform curvature condition
$$
(\calC_J)_{M', c', a_0, \widetilde{\zeta_1}, \{C(m)\}_{m \in \N}}
$$ 
(see Theorem \ref{apdthm002}). Here $\widetilde{\zeta_1}$ is a constant which only depends on $\check{C_1}$ and any $2$-admissible constant such that  
$$
B_{(Y, d)}\left(0, \frac{\zeta_0}{2} \right) \subseteq B^n(\widetilde{\zeta_1}) \subseteq B_{(Y, d)} \left(0, \frac{3\zeta_0}{4} \right).
$$
Moreover, $M'>0, M' \in \N$ and $c'>0$ are two constants which only depends on $a_0, r, M, \widetilde{\zeta_1}, \check{C_1}$, on $\{C(m)\}_{m \in \N}$ and on any $2$-admissible constants (see Theorem \ref{apdthm002}).

\item [4.] $\widehat{\theta}_b: B_{(Y, \tilde{d})}\left(0, \frac{5\zeta_0}{8}\right) \mapsto B_{(Y, \tilde{d})}\left(0, \frac{3\zeta_0}{4}\right)$ is $C^\infty$, that is, for each $m \ge 0$,
$$
\|\widehat{\theta}_b\|_{C^m\left(B_{(Y, \tilde{d})}\left(0, \frac{5\zeta_0}{8}\right)\right)} \le \widetilde{C}(m),
$$
where $\widetilde{C}(m)$ only depends on $\check{C_1}$, any $m$-admissible constants and also how $Z$ controls $\theta$ at the unit scale (more precisely, $\calQ_2$). 
\end{enumerate}

Now we turn to prove \eqref{modeq02} and we shall first extend this inequality to a ``global" estimation (note that both $\calJ$ and $\widetilde{\calJ}$ are only defined for $g \in L^2\left(B_{(Y, \tilde{d})}\left(0, \frac{\zeta_0}{8} \right) \right)$. For simplicity, we denote
$$
U_0:=B_{(Y, \tilde{d})}\left(0, \frac{\check{C}\zeta_0}{8\check{C_1}} \right), 
$$
$$
U_1:=B_{(Y, \tilde{d})} \left(0, \frac{\zeta_0}{8} \right), 
$$
$$
U_5:=B_{(Y, \tilde{d})} \left(0, \frac{\zeta_0}{2} \right),
$$
$$
U_6:=B_{(Y, \tilde{d})} \left(0, \frac{5\zeta_0}{8} \right),
$$
$$
U_7:=B_{(Y, \tilde{d})} \left(0, \frac{3\zeta_0}{4} \right)
$$
and
$$
U_8:=B_{(Y, \tilde{d})}(0, \zeta_0). 
$$
Moreover, we take $U_1 \Subset U_2 \Subset U_3 \Subset U_4 \subseteq U_5$, where $U_4$ is compact,  
$$
U_2:=B^n(\zeta_0'),
$$
$$
U_3:=B^n(\zeta_0''),
$$
 relatively compact in $U_4$ with satisfying $\overline{U_2} \subsetneq U_3$ and $\overline{U_3} \subsetneq U_4$. Here,  $\zeta_0''>\zeta_0'>0$ are two constants only depending on $\check{C_1}$, $\zeta_0$ and any $2$-admissible constants.

 Finally, we take a $0<\zeta_0'''<\zeta_0'$, which only depends on $\check{C_1}$, $\zeta_0$, $\zeta_0'$ and any $2$-admissible constants, such that 
$$
\overline{U_1} \subsetneq B^n(\zeta_0''') \subsetneq U_2.
$$
This allows us to construct two bump functions $h_1$ and $h_2$ satsifying $0 \le h_1, h_2 \le 1$,  $h_1 \equiv 1$ on $B^n(\zeta_0''')$, compactly supported in $U_2$ and $h_2 \equiv 1$ on $\overline{U_2}$, compactly supported in $U_3$. Note that $h_1h_2=h_1$, moreover, the $C^m$-norm of $h_1$ and $h_2$ only depends on $m$, any dimension constants, and the constants $\zeta_0', \zeta_0'', \zeta_0'''$ fixed above (therefore, they also depend on $\check{C_1}$ and any $2$-admissible constants).

\begin{figure}[ht]
\begin{tikzpicture}[scale=4.5]
\draw (0,-1.1) [->] -- (0,1.1) ;
\draw (-1.3,0) [->] -- (1.3,0) ;
\fill (1.1, 1.1) node[left] {$\R^n$}; 
\draw (0.05, -0.02) circle (0.1); 
\fill [opacity=.7, red]  (0, 0) circle [radius=.5pt]; 
\fill (0.04, -0.01) node [below]{\footnotesize{$O$}}; 
\fill (0.12, 0.05) node [right] {\footnotesize{$U_0$}};
\draw (0.09, -0.03) circle (0.2); 
\fill (0.26, 0.06) node [right] {\footnotesize{$U_1$}};
\draw (0.11, -0.01) circle (0.3); 
\fill (0.39, 0.07) node [right] {\footnotesize{$U_2$}};
\draw (0.15, 0.01) circle (0.39); 
\fill (0.52, 0.08) node [right] {\footnotesize{$U_3$}};
\draw (0.19, 0.03) circle (0.48); 
\fill (0.65, 0.09) node [right] {\footnotesize{$U_4$}};
\draw (0.24, 0.05) circle (0.56); 
\fill (0.78, 0.1) node [right] {\footnotesize{$U_5$}};
\draw (0.16, 0.05) circle (0.74); 
\fill (0.9, 0.11) node [right] {\footnotesize{$U_6$}};
\draw (0.18, 0.06) circle (0.84); 
\fill (1.01, 0.12) node [right] {\footnotesize{$U_7$}};
\draw (0.2, 0.07) circle (0.93);
\fill (1.14, 0.13) node [right] {\footnotesize{$U_8$}};
\end{tikzpicture}
\caption{}
\end{figure}

Therefore, using the notation above, it suffices to show there exists some $\eta>0$, such that
\begin{equation} \label{modeq05}
\|\calJ g-\widetilde{\calJ}(g) \|_{L^2(U_0)} \lesssim |b|^{\eta}\|g\|_{L^2(U_1)}. 
\end{equation} 
Note that
\begin{eqnarray*}
&&\|\calJ (g)- \widetilde{\calJ}(g)\|_{L^2(U_0)}^2= \int_{U_0} |\calJ (g)(u)- \widetilde{\calJ}(g)(u)|^2du\\
&&= \int_{U_0} \bigg | \widehat{\psi_1}(u) \int g(\widehat{\gamma}_t (u)) \widehat{\psi_2}(\widehat{\gamma}_t(u)) \widehat{\varrho}(t, u) \chi(t)dt \\
&& \quad \quad \quad \quad   -\widehat{\psi_1}(u) \int g(\widehat{\gamma}_t \circ \widehat{\theta}_{b}(u)) \widehat{\psi_2}(\widehat{\gamma}_t \circ \widehat{\theta}_{b}(u)) \widehat{\varrho} (t, \widehat{\theta}_{b}(u)) \chi(t)dt \bigg|^2 du \\
&&=  \int_{U_0} \bigg | \left(\widehat{\psi_1}h_1 \right)(u) \int \left(gh_1\right)(\widehat{\gamma}_t (u)) \left(\widehat{\psi_2}h_1\right)(\widehat{\gamma}_t(u)) \widehat{\varrho}(t, u) \chi(t)dt \\
&& \quad \quad \quad \quad   -\left(\widehat{\psi_1}h_1 \right) (u) \int \left(gh_1\right)(\widehat{\gamma}_t \circ \widehat{\theta}_{b}(u)) \left(\widehat{\psi_2}h_1 \right)(\widehat{\gamma}_t \circ \widehat{\theta}_{b}(u)) \widehat{\varrho} (t, \widehat{\theta}_{b}(u)) \chi(t)dt \bigg|^2 du \\
&& \le  \int_{\R^n} \bigg | \left(\widehat{\psi_1}h_1 \right)(u) \int \left(gh_1\right)(\widehat{\gamma}_t (u)) \left(\widehat{\psi_2}h_1\right)(\widehat{\gamma}_t(u)) \widehat{\varrho}(t, u) \chi(t)dt \\
&& \quad \quad \quad \quad   -\left(\widehat{\psi_1}h_1 \right) (u) \int \left(gh_1\right)(\widehat{\gamma}_t \circ \widehat{\theta}_{b}(u)) \left(\widehat{\psi_2}h_1 \right)(\widehat{\gamma}_t \circ \widehat{\theta}_{b}(u)) \widehat{\varrho} (t, \widehat{\theta}_{b}(u)) \chi(t)dt \bigg|^2 du.
\end{eqnarray*}
Now given a bounded measurable function $g$ on $\R^n$ (from now on, we denote $g$ to be a general function defined on $\R^n$, which is not required to be supported on $U_1$ anymore), we wish to study the operators
$$
\calU(g)(u):= \left(\widehat{\psi_1}h_1 \right)(u) \int \left(gh_1\right)(\widehat{\gamma}_t (u)) \left(\widehat{\psi_2}h_1\right)(\widehat{\gamma}_t(u)) \widehat{\varrho}(t, u) \chi(t)dt
$$
and
$$
\widetilde{\calU}(g)(u):=\left(\widehat{\psi_1}h_1 \right) (u) \int \left(gh_1\right)(\widehat{\gamma}_t \circ \widehat{\theta}_{b}(u)) \left(\widehat{\psi_2}h_1 \right)(\widehat{\gamma}_t \circ \widehat{\theta}_{b}(u)) \widehat{\varrho} (t, \widehat{\theta}_{b}(u)) \chi(t)dt.
$$
Therefore, to prove \eqref{modeq05}, it suffices to prove there exists some $\eta>0$, such that
\begin{equation} \label{modeq06}
\|\calU-\widetilde{\calU}\|_{L^2 \mapsto L^2} \lesssim |b|^\eta. 
\end{equation}
Indeed, if \eqref{modeq06} is true, then by the calculation above, we have
$$
\|\calJ(g)-\widetilde{\calJ}(g)\|_{L^2(U_0)} \le \|\calU(g)-\widetilde{\calU}(g)\|_{L^2 \mapsto L^2} \lesssim |b|^\eta \|g\|_{L^2}=|b|^\eta\|g\|_{L^2(U_1)},
$$
if $\textrm{supp} g \subset U_1$. 

Moreover, it is easy to see that 
$$
\|\calU\|_{L^1 \mapsto L^1}, \|\widetilde{\calU}\|_{L^1 \mapsto L^1},  \|\calU\|_{L^\infty \mapsto L^\infty}, \|\widetilde{\calU}\|_{L^\infty \mapsto L^\infty} \lesssim 1.
$$

In what follows, $\eta>0$ will be a positive number that may change from line to line.  Now we turn to prove \eqref{modeq06}, and clearly, it suffices to show $\| (\calU-\widetilde{\calU})^*(\calU-\widetilde{\calU})\|_{L^2 \mapsto L^2} \lesssim |b|^\eta$ (where we have replaced ``$2\eta$" by ``$\eta$"), which follows from
\begin{equation} \label{modeq03}
\|\calU^* (\calU-\widetilde{\calU})\|_{L^2 \mapsto L^2} \lesssim |b|^\eta
\end{equation}
and
\begin{equation} \label{modeq04}
\|\widetilde{\calU}^* (\calU-\widetilde{\calU})\|_{L^2 \mapsto L^2} \lesssim |b|^\eta.
\end{equation}

\medskip

Denote $R=\calU-\widetilde{\calU}$. We prove \eqref{modeq03} first. Clearly, it suffices to show that
$$
\|R^* \calU \calU^* R\|_{L^2 \mapsto L^2} \lesssim |b|^\eta.
$$
Since $\|R^*\|_{L^2 \mapsto L^2} \lesssim 1$, it suffices to show
$$
\|\calU \calU^* R\|_{L^2 \mapsto L^2} \lesssim |b|^\eta.
$$
This follows from an application of \cite[Theorem 14.5]{BS} with some necessary modification. We refer the readers Theorem \ref{apdthm01} in Appendix B for this modified version of \cite[Theorem 14.5]{BS} and here, we shall mention those conditions one need to check.

\medskip

First, note that since we are assuming $\chi(t)$ has small support (and we are allowed to choose how small, depending on $\check{\gamma}$), we may, without the loss of generality, assume that
\begin{equation} \label{adjointSj01} 
\left | \det \frac{\partial \widehat{\gamma}}{\partial u}(t, u) \right| \ge \frac{1}{2}, \quad u \in U_7
\end{equation}
on the support of $\chi$. This is clear since $\widehat{\gamma}(0, u) \equiv u$ and $\|\widehat{\gamma}\|_{C^m} \lesssim 1, m \in \N$. 

By a direct calculation, we note that
$$
\calU^*(g)(u)=\left(h_1^2 \overline{\widehat{\psi_2}}\right)(u) \int \left( \overline{\widehat{\psi_1}} h_1 \right) (\widehat{\gamma}_t^{-1}(u)) \left(gh_2 \right) (\widehat{\gamma}_t^{-1}(u)) \widehat{\varrho}_0(t, u) \chi(t)dt, 
$$
where $\widehat{\varrho}_0(t, u)=\overline{\widehat{\varrho}}(t, \widehat{\gamma}_t^{-1}(u)) \left| \det \frac{\partial \widehat{\gamma}}{\partial t} \left(t, \widehat{\gamma}_t^{-1}(u) \right) \right|^{-1}$. Note that $\calU^*$ is of the same form as $\calU$, with $\widehat{\gamma}_t$ replaced by $\widehat{\gamma}^{-1}_t$. 

Therefore, 
\begin{eqnarray*}
\calU\calU^*(g)(u)%
&=& \left(\widehat{\psi_1}h_1 \right)(u) \int \left(\calU^*(g)h_1\right)(\widehat{\gamma}_{t_1} (u)) \left(\widehat{\psi_2}h_1\right)(\widehat{\gamma}_{t_1}(u)) \widehat{\varrho}(t_1, u) \chi(t_1)dt_1 \\
&=& \left(\widehat{\psi_1}h_1\right)(u)  \int \int \left(h_1^4 |\widehat{\psi}_2|^2 \right) (\widehat{\gamma}_{t_1}(u)) \left( \overline{\widehat{\psi_1}}h_1 \right) \left( \widehat{\gamma}_{t_2}^{-1} \circ \widehat{\gamma}_{t_1}(u) \right) \\
&& \quad \quad \quad \quad \quad \quad \quad \quad \cdot \left(gh_2 \right) \left( \widehat{\gamma}_{t_2}^{-1} \circ \widehat{\gamma}_{t_1}(u) \right) \widehat{\varrho}_0(t_2, \widehat{\gamma}_{t_1}(u)) \widehat{\varrho}(t_1, u)  \\
&& \quad \quad \quad \quad \quad \quad \quad \quad  \quad \quad  \quad \quad \quad \quad  \quad  \quad \quad \quad  \quad \cdot \chi(t_1) \chi(t_2)dt_1dt_2 \\
&=&  \left(\widehat{\psi_1}h_1\right)(u) \int \int g \left( \widehat{\gamma}_{t_2}^{-1} \circ \widehat{\gamma}_{t_1}(u) \right) \left( \overline{\widehat{\psi_1}}h_1 \right) \left( \widehat{\gamma}_{t_2}^{-1} \circ \widehat{\gamma}_{t_1}(u) \right) \\
&& \quad \quad \quad \quad \quad \quad \quad \quad \cdot  \widehat{\varrho}_1(t_1, t_2, u) \chi(t_1)\chi(t_2)dt_1dt_2,
\end{eqnarray*}
where in the last equation, we use the fact that $h_1h_2=h_1$, and the function $\widehat{\varrho}_1(t_1, t_2, u)$ is defined as
$$
\widehat{\varrho}_1(t_1, t_2, u):=\left(h_1^4 |\widehat{\psi}_2|^2 \right) (\widehat{\gamma}_{t_1}(u)) \widehat{\varrho}_0(t_2, \widehat{\gamma}_{t_1}(u)) \widehat{\varrho}(t_1, u).
$$
Clearly, for $a''$ sufficiently small, $\widehat{\rho}_1  \in C^\infty( \overline{B^k(a'')} \times \overline{B^k(a'')} \times U_7)$. Meanwhile, we notice that both $\widehat{\psi_1}h_1$ and $\overline{\widehat{\psi_1}}h_1$ are $C^\infty_0$, both of which, have compact supports contained in $U_4$. 

\medskip

(1). \textit{Claim: The operator $\calU \calU^*$ satisfies the conditions of the operator``$S_j$" in Theorem \ref{apdthm01}.}

\medskip

To prove the claim, it suffices to check the following two conditions: 

\medskip

1.  $\widehat{\gamma}_{t_2}^{-1}  \circ \widehat{\gamma}_{t_1}$ is $C^\infty$. This is clear by our assumption. 

\medskip

2. For each $l$, $1 \le l \le r$, there is a multi-index $\alpha$, such that
\begin{equation} \label{modifycond00}
\left(\frac{1}{8\check{C_1}}\right)^{\tilde{d}_l} Y_l(u)=\frac{1}{\alpha !} \frac{\partial}{\partial t}^\alpha \bigg|_{t=(t_1, t_2)=(0, 0)} \frac{d}{d\epsilon} \bigg|_{\epsilon=1} \widehat{\gamma}_{\epsilon t_2}^{-1}  \circ \widehat{\gamma}_{\epsilon t_1} \circ \widehat{\gamma}^{-1}_{t_1} \circ \widehat{\gamma}_{t_2}(u).
\end{equation}
This follows from the pullback of our assumption \eqref{eq3455} via the scaling map $\Phi$. 

\medskip

(2). Let $\omega:=\frac{b}{|b|} \in \SSS^{n-1}$. For $\zeta \in [-1, 1]$ and $g$ a bounded measurable function on $\R^n$, we define the operator
\begin{eqnarray*}
R^{\zeta}(g)(u) %
&:=&\left(\widehat{\psi_1}h_1 \right)(u) \int \left(gh_1\right) \left(\widehat{\gamma}_t \circ \widehat{\theta}_{\zeta \omega}(u)\right) \left(\widehat{\psi_2}h_1 \right) \left(\widehat{\gamma}_t \circ \widehat{\theta}_{\zeta w}(u)\right) \widehat{\varrho} \left(t, \widehat{\theta}_{\zeta w}(u)\right) \chi(t)dt \\
&:=& \left(\widehat{\psi_1}h_1 \right)(u) \int g(\gamma_{t, \zeta}(u)) \left(\widehat{\psi_2}h_1^2\right)(\gamma_{t, \zeta}(u)) \widetilde{\varrho}(t, \zeta, u) \chi(t)dt,
\end{eqnarray*}
where we put
$$
\gamma_{t, \zeta}(u):=\widehat{\gamma}_t \circ \widehat{\theta}_{\zeta \omega}(u),
$$
and
$$
\widetilde{\varrho}(t, \zeta, u):=\widehat{\varrho} \left(t, \widehat{\theta}_{\zeta \omega}(u)\right).
$$
It is clear that $R_0:=R^0=\calU$ and $R_1:=R^{|b|}=\widetilde{\calU}$. Moreover, it is easy to see that both $\widehat{\psi_1}h_1$ and $\widehat{\psi_2}h_1^2$ are $C^\infty_0$, compactly supported in $U_4$, and $\widetilde{\rho} \in C^\infty\left(B^k(a'') \times [-|b|, |b|] \times U_6 \right)$. 

\medskip

\textit{Claim: $\gamma_{t, \zeta}$ is a $C^\infty$ function in $U_5$.}

 This is clear, since both $\widehat{\theta}_{\zeta w}$ and $\widehat{\gamma}_t$ do. Therefore, the desired claim follows from Proposition \ref{controlcor}.

\medskip

The proof for \eqref{modeq03} is then complete, and we turn to the proof of \eqref{modeq04}.  Recall that
$$
\widetilde{\calU}(g)(u):=\left(\widehat{\psi_1}h_1 \right) (u) \int \left(gh_1\right)(\widehat{\gamma}_t \circ \widehat{\theta}_{b}(u)) \left(\widehat{\psi_2}h_1 \right)(\widehat{\gamma}_t \circ \widehat{\theta}_{b}(u)) \widehat{\varrho} (t, \widehat{\theta}_{b}(u)) \chi(t)dt
$$
and we need to show that there exists some $\eta>0$, such that $\|\widetilde{\calU}^* (\calU- \widetilde{\calU})\|_{L^2 \mapsto L^2} \lesssim |b|^\eta$. By the same argument as before, it suffices to show that there exists some $\eta>0$, such that
$$
\|\widetilde{\calU}\widetilde{\calU}^*R\|_{L^2 \mapsto L^2} \lesssim |b|^\eta,
$$
for which we will use Corollary \ref{apdcor01}. 

To see this, we note that the way to deal with $R$ here is the same as part (2) above, and hence we only need to show the operator $\widetilde{\calU}\widetilde{\calU}^*$ is of the form ``$\widetilde{S}_j$" in Corollary \ref{apdcor01}.

Again, we start with calculating $\widetilde{\calU}^*$. As in \eqref{adjointSj01}, we may assume for $a$ and $|b|$ sufficiently small, 
$$
\left| \det \frac{\partial \widehat{\gamma}}{\partial u}(t, u) \right| \ge \frac{1}{2} \ \textrm{and} \ \left| \det \frac{\partial \widehat{\theta}_b}{\partial  u}(u) \right| \ge \frac{1}{2}, \ u \in U_6.
$$
Indeed, the first inequality follows from \eqref{adjointSj01}, since $U_6 \subset U_7$, and the second one follows the same proof as \eqref{adjointSj01}, with replacing role of ``$\widehat{\gamma}$" and ``$t$" there by $\theta$ and $|b|$, respectively. 

Therefore, by a direct calculation, we have
$$
\widetilde{\calU}^*(g)(u)=\left(h_1^2 \overline{\widehat{\psi_2}}\right)(u) \int \left(\overline{\widehat{\psi_1}} h_1 \right) (\widehat{\theta}_b^{-1} \circ \widehat{\gamma}_t^{-1}(u)) \left(gh_2 \right) (\widehat{\theta}_b^{-1} \circ \widehat{\gamma}_t^{-1}(u)) \widehat{\varrho}_2(t, u) \chi(t)dt,
$$
where $\widehat{\varrho}_2(t, u):=\overline{\widehat{\varrho}}(t, \widehat{\gamma}_t^{-1}(u)) \left| \det \frac{\partial \gamma}{\partial u} \left( t, \widehat{\gamma}_t^{-1}(u) \right) \right| \left| \det \frac{\partial \widehat{\theta}_b}{\partial u} \left(\widehat{\theta}_b^{-1} \circ \widehat{\gamma}_t^{-1}(u) \right) \right|$. Hence, $\widehat{\rho}_2 \in C^\infty$. Again, we note that $\widetilde{\calU}^*$ is of the same form as $\widetilde{\calU}$, with $\widehat{\gamma}_t \circ \widehat{\theta}_b$ replaced by $\widehat{\theta}_b^{-1} \circ \widehat{\gamma}_t^{-1}$. 

Therefore, 
\begin{eqnarray*}
\widetilde{\calU}\widetilde{\calU}^*(g)(u)%
&=& \left(\widehat{\psi_1}h_1 \right)(u) \int \left( \widetilde{\calU}^*(g) h_1 \right)(\widehat{\gamma}_{t_1} \circ \widehat{\theta}_{b}(u)) \left(\widehat{\psi_2}h_1 \right)(\widehat{\gamma}_{t_1} \circ \widehat{\theta}_{b}(u)) \widehat{\varrho} (t_1, \widehat{\theta}_{b}(u)) \chi(t_1)dt_1\\
&=& \left(\widehat{\psi_1}h_1 \right)(u) \int  \int \left(h_1^4 |\widehat{\psi}_2|^2 \right) (\widehat{\gamma}_{t_1} \circ \widehat{\theta}_{b}(u)) \left(\overline{\widehat{\psi}_1} h_1\right) (\widehat{\theta}_b^{-1} \circ \widehat{\gamma}_{t_2}^{-1} \circ \widehat{\gamma}_{t_1} \circ \widehat{\theta}_{b}(u) ) \\
&&  \quad \quad \quad \quad \quad \quad \quad \quad \cdot \left( gh_2 \right) (\widehat{\theta}_b^{-1} \circ \widehat{\gamma}_{t_2}^{-1} \circ \widehat{\gamma}_{t_1} \circ \widehat{\theta}_{b}(u) ) \widehat{\varrho}_2(t_2, \widehat{\gamma}_{t_1} \circ \widehat{\theta}_{b}(u)) \\
&&  \quad \quad \quad \quad \quad \quad \quad \quad  \quad \quad  \quad \quad \quad \quad  \quad \cdot \widehat{\varrho} (t_1, \widehat{\theta}_{b}(u)) \chi(t_1)\chi(t_2) dt_1dt_2\\
&=&  \left(\widehat{\psi_1}h_1 \right)(u) \int  \int  g(\widehat{\theta}_b^{-1} \circ \widehat{\gamma}_{t_2}^{-1} \circ \widehat{\gamma}_{t_1} \circ \widehat{\theta}_{b}(u) )\left(\overline{\widehat{\psi}_1} h_1\right) (\widehat{\theta}_b^{-1} \circ \widehat{\gamma}_{t_2}^{-1} \circ \widehat{\gamma}_{t_1} \circ \widehat{\theta}_{b}(u) ) \\
&&  \quad \quad \quad \quad \quad \quad \quad \quad \cdot \widehat{\varrho}_3(t_1, t_2, u) \chi(t_1) \chi(t_2)dt_1dt_2,
\end{eqnarray*}
where again, in the last equation, we use the fact that $h_1=h_1h_2$, and the function  $\widehat{\varrho}_3(t_1, t_2, u)$ is defined as
$$
\widehat{\varrho}_3(t_1, t_2, u):= \left(h_1^4 |\widehat{\psi}_2|^2 \right) (\widehat{\gamma}_{t_1} \circ \widehat{\theta}_{b}(u)) \widehat{\varrho}_2(t_2, \widehat{\gamma}_{t_1} \circ \widehat{\theta}_{b}(u))  \widehat{\varrho} (t_1, \widehat{\theta}_{b}(u)).
$$
Clearly, for $a''$ sufficiently small, we see that $\widehat{\varrho}_3 \in C^\infty (\overline{B^k(a'')} \times \overline{B^k(a'')} \times U_6)$ and both $\widehat{\psi_1}h_1$ and $\overline{\widehat{\psi_1}}h_1$ belongs to $C^\infty_0$, compactly supported in $U_4$. Finally, we note that to apply Corollary \ref{apdcor01}, the conditions need to be checked are exactly the same as those in the proof of \eqref{modeq03}. Hence, the proof of \eqref{modeq04} is complete. 
\end{proof}

\bigskip

\section{The sparse domination}

\bigskip

In this section, we first make a review of the setting of the sparse domination, which continues the construction in Section 3. Next, we study an $L^p$ improving property for the singular Radon transform, which plays an important role in our main estimate later. The key point for this section is to make all assumptions ``quantitative", so that one can make all the estimations later independent of the ``scales $\del^j$" and ``centers $x_c(Q)$". 

\medskip

\subsection{Setting revisited}

We shall first give a quantitative version of our setting. Recall that we are given a $C^\infty$ mapping $\gamma$, which maps a neighborhood of $(0, 0) \in \R^n \times \R^k$ to $\R^n$, which is assumed to satisfy the following conditions: 
\begin{enumerate}
\item [1.] $\gamma(x, 0) \equiv x$;
\item [2.] $\gamma$ is curved to finite order at $0$, namely, if we consider the vector field
$$
W(t, x)=\frac{\partial}{\partial \varepsilon} \bigg|_{\varepsilon=1} \gamma_{\varepsilon t} \circ \gamma_t^{-1}(x)
$$
with its formal Taylor expansion
$$
W(t) \sim \sum_\alpha t^\alpha X_\alpha
$$
then the collection of $C^\infty$ vector fields $\left\{X_\alpha: 0 \neq \alpha \in \N^k  \right\}$ defined on a neighborhood $U$ of $0 \in \R^n$ satisfies the H\"ormander's condition at $0$ (See Definition \ref{Hormanderdefn01}). As a consequence,  $\gamma$ satisfies the condition $(\calC_J)$ uniformly in some $V \Subset U$, where $V$ is an open, path-connected neighborhood of $0 \in \R^n$, and relatively compact in $U$. 
\end{enumerate}

Recall that under the above assumptions, we are allowed to apply Algorithm \ref{alg001}, together with Theorem \ref{dyadicSHT}, to get 
\begin{enumerate}
\item [$\bullet$] A list of $C^\infty$ vector fields $(X, d)=\{ (X_i, d_i)\}_{1 \le i \le q}$;
\item [$\bullet$] A space of homogeneous type $(V, \rho, |\cdot|)$, where $\rho$ is the Carnot-Carath\'edory metric induced by the list $(X, d)$, together with a dyadic decomposition 
$$
\calD=\bigcup_{k \in \Z} \calD_k
$$
with parameters $\frakC>1$, $0<\del<\frac{1}{100}$ and $0< \epsilon<1$. 
\end{enumerate}

\medskip

The setting is as follows. 

\medskip

\textit{Step I: Assumptions on the ``sets".}

\medskip

We start with taking and fixing a sequence of increasing neighborhoods of $0 \in \R^n$, more precisely, we take
$$
0 \in K''' \Subset  K'' \Subset K' \Subset V''\Subset V' \Subset V \Subset U.
$$
Here, we require
\begin{enumerate}
\item [(a).]  $U$ and $V$ are defined as above;
\item [(b).]  $V' \subsetneq V$ is an open, path-connected neighborhood of $0 \in \R^n$, relatively compact in $V$. Once we fix our choice of $V'$, we can also take some $a'>0$, such that
\begin{equation} \label{supportofK01}
\gamma(x, t): V' \times B^k(a') \rightarrow V,
\end{equation} 
where we may assume that $a'$ is sufficiently small, so that for every $t \in B^k(a')$, $\gamma_t$ is a diffeomorphism onto its image and it makes sense to write $\gamma_t^{-1}$;
\item [(c).] $V'' \Subset V', V'' \subsetneq V'$ is a compact set which satisfies the properties in Definition \ref{controlVC}; 
\item [(d).] $K' \subsetneq V''$ is compact and path-connected. As a conclusion of Algorithm \ref{alg001}, we assume that $\gamma$ is controlled by the list $(X, d)$ for every $x \in  K'$;
\item [(e).] $K''' \subsetneq K'' \subsetneq K'$  where both $K'''$ and $K''$ are compact and path-connected.
\end{enumerate}

\medskip

\textit{Step II: Construction of space of homogeneous type and its dyadic decomposition.}

\medskip

For some technical reasons, we need some adjustments to the space of homogeneous type $(V, \rho, |\cdot|)$, so that we can apply the lemma of modulus of continuity in the proof of sparse bound. We need some observations before we state the required adjustments. 
\medskip

\underline{\textit{Observation $II.1$.}} Since $K''' \subsetneq K''$, we can take some $\widetilde{N_0}$ sufficiently large, such that for $ x_0 \in K'''$, 
$$
B_{\left( \left(\del^{\widetilde{N_0}}\right)^d X, d\right)}\left(x_0, \frac{10\frakC}{\del} \right)=B_{(X, d)} \left( x_0, \frac{10\frakC}{\del} \cdot \del^{\widetilde{N_0}} \right) \subset K'', 
$$
where $\widetilde{N_0}$ only depends on $K'''$, $K''$ and the $C^m$-norm of $X_i, 1 \le i \le q$ (hence on $\gamma$). 

\medskip

\underline{\textit{Observation $II.2$.}} Let $x_0 \in K'''$. Since $\gamma$ is controlled by the list $(X, d)$ (and hence controlled at the unit scale) in $K''' \subset K'$, by Proposition  \ref{controlQ1Q2} (more precisely, by $\calQ_2$), we can take some $\frakC_1>3\frakC$  and $a''>0$ sufficiently small, such that
\begin{equation} \label{supportofK02}
\gamma\left(\overline{B^k(a'')} \times \overline{B_{(X, d)}\left(x_0, \frac{\frakC}{\del} \cdot \del^{\widetilde{N_0}} \right)}\right) \subseteq B_{(X, d)}\left(x_0, \frac{\frakC_1}{\del} \cdot \del^{\widetilde{N_0}} \right) \subset K'',
\end{equation}
where $a''$ and $\frakC_1$ only depend on $\frakC$, $\del$, $\widetilde{N_0}$ and how $\gamma$ is controlled by $(X, d)$ (more precisely, the constants in $\calQ_1$). Note that this step corresponds to our early Obersevation I in Section 4. 

Moreover, we wish to have for any $x_0 \in K'''$, $B_{(X, d)}  \left(x_0, \left(\frac{10\frakC_1}{\del} \right)^{10} \cdot \del^{\widetilde{N_0}} \right) \subset K'$, which, however, is not necessary true. Therefore, we take some $\widetilde{N_1}>1$ sufficiently large, such that
\begin{equation} \label{1109eq01}
B_{\left( \left(\del^{\widetilde{N_1}} \right)^d X, d \right)} \left(x_0, \left(\frac{10\frakC_1}{\del}\right)^{10} \cdot \del^{\widetilde{N_0}} \right) \subset K',
\end{equation}
where the choice of $\widetilde{N_1}$ only depends on the $C^m$-norm of $\gamma$, $\frakC_1$, $\del$, $\widetilde{N_0}$, $K'''$ and $K'$. This is clear since\footnote{The reason for us to take a large constant $\left( \frac{10\frakC_1}{\del} \right)^{10}$ in our setting is to avoid the case that the Carnot-Carath\'edory ball goes outside $V'$, as we need to use the doubling constant $C_{V'}$ (see \eqref{doublingineq}) for our later estimate (see, e.g., \eqref{1125eq10}).}
$$
B_{\left( \left(\del^{\widetilde{N_1}} \right)^d X, d \right)} \left(x_0, \left( \frac{10\frakC_1}{\del} \right)^{10} \cdot \del^{\widetilde{N_0}} \right)=B_{(X, d)} \left(x_0, \left(\frac{10\frakC_1}{\del} \right)^{10} \cdot \del^{\widetilde{N_1}+\widetilde{N_0}} \right).
$$
Denote $\widetilde{N}=\widetilde{N_0}+\widetilde{N_1}$. Therefore and in particular, we have
\begin{equation} \label{1125eq05}
\gamma\left( \overline{B^k(a'')} \times \overline{B_{\left( \left(\del^{\widetilde{N}} \right)^dX, d\right)}\left(x_0, \frac{\frakC}{\del} \right) }\right) \subseteq B_{\left( \left(\del^{\widetilde{N}} \right)^dX, d\right)}\left(x_0, \frac{\frakC_1}{\del} \right)\subset K'.
\end{equation}

\medskip

\underline{\textit{Observation $II.3$.}} Apply Theorem \ref{scalingmap} to $x_0 \in K'''$ and the list
\begin{equation} \label{correctscalingeq01}
\left( \left( \frac{8\frakC_1}{\del} \cdot \del^{\widetilde{N}} \right)^d X, d \right)
\end{equation}
to get a scaling map. Therefore, all the admissible constants derived from this scaling map only depends on $\frakC_1$, $\del$, $\widetilde{N}$ (and hence, the constant $\frakC$ the $C^m$-norm of $\gamma$, $K'''$ and $K'$) and how $\gamma$ is controlled by $(X, d)$ (more precisely, the constants in $\calQ_1$). In particular, we have the following conclusions. 

\medskip

\textit{$(a).$} There exists a $0<\widetilde{\zeta_0}<1$, only depending on $2$-admissible constants, such that for any $j \ge 0$, 
\begin{eqnarray} \label{1127eq01}
\Phi_{x_0, j}: B_{(Y, d)}(0,\widetilde{\zeta_0})%
&\longrightarrow&  B_{\left( \left( \frac{8\frakC_1}{\del} \cdot \del^{\widetilde{N}+j} \right)^d X, d \right)}(x_0, \widetilde{\zeta_0}) \\
&& =B_{\left( \left(\del^{\widetilde{N}+j} \cdot \widetilde{\zeta_0}\right)^d X, d \right)}\left(x_0, \frac{8\frakC_1}{\del} \right) \nonumber \\
&&= B_{\left(\left(\del^{\widetilde{N}+j} \right)^d X, d\right)} \left(x_0, \frac{8\frakC_1}{\del} \cdot \widetilde{\zeta_0}  \right) \subset K' \nonumber
\end{eqnarray}
is a $C^\infty$ diffeomorphism, where $\Phi_{x_0, j}$ is the scaling map defined in Theorem \ref{scalingmap} when we apply it to $x_0 \in K'''$ and the list $\left( \left( \frac{8\frakC_1}{\del} \cdot \del^{\widetilde{N}+j} \right)^d X, d \right)$, and $(Y, d)$ is the pullback of the list $\left( \left( \frac{8\frakC_1}{\del} \cdot \del^{\widetilde{N}+j} \right)^d X, d \right)$ via $\Phi_{x_0. j}$. 

We make a remark that the reason that we can choose $\widetilde{\zeta_0}$ independent of $x_0 \in K'''$ and $j \ge 0$ is that the admissible constants (See \eqref{correctscalingeq01}) are uniform for all $x_0 \in K'''$ and $j \ge 0$ by our assumption. Finally, note that this step corresponds to the Observation II in Section 4; 

\medskip

\textit{$(b).$} The second conclusion is based on the following lemma in \cite{NAO}. 

\begin{lem}[{\cite[Lemma 2.15.38]{NAO}}] \label{lem190201}
There exists a $2$-admissible constant $\widetilde{\eta_1}>0$, such that for every $0<\eta' \le \widetilde{\eta_1}$, there exists $2$-adimissible constants $0<\eta_4=\eta_4(\eta')$ and $\eta_3$ with $\eta_4<\eta_3$, such that for every $j \ge 0$ and every $f  \ge 0$ with $\textrm{supp}(f) \subset K'''$, we have (for $x \in K'$)
\begin{eqnarray*}
\int_{B^q(\eta_4)} f\left(e^{t \cdot \left(\frac{8\frakC_1}{\del} \cdot \del^{\widetilde{N}+j} \right)^d X} x \right)dt %
&\lesssim_2& \int_{B^n(\eta')} f \circ \Phi_{x, j}(u)du \\
&\lesssim_2&  \int_{B^q(\eta_3)} f\left(e^{t \cdot \left(\frac{8\frakC_1}{\del} \cdot \del^{\widetilde{N}+j} \right)^d X} x \right)dt.
\end{eqnarray*}
\end{lem}

Note that it is important that all the implicit constants taken in Lemma \ref{lem190201} are independent of the choice of $j$. For simplicity, we will only consider the case $\eta'=\widetilde{\eta_1}$ in Lemma \ref{lem190201} for our further application.  

Next by Corollary \ref{scalingmap01}, we can find an $2$-admissible constant $\widetilde{\zeta^1}>0$, such that for each $j \ge 0$ and $x_0 \in K'''$, 
\begin{equation} \label{190102eq022}
B_{(Y, d)}\left(0, \widetilde{\zeta^1} \right) \subseteq B^n\left(\widetilde{\eta_1}\right),
\end{equation}
where the list $(Y, d)$ is defined the same as in \eqref{1127eq01}. 

We make a remark that the second step is crucial later when we construct the $C^\infty$ mapping ``$\theta$'' for each ``centers" $x_c(Q) $ and each scale $\del^j, j \ge 0$ in the proof of the sparse bound, so that Lemma \ref{moduluscont} can be applied. 

\medskip

To this end, we take
$$
0<\widetilde{\zeta}<\min\left\{\widetilde{\zeta_0}, \widetilde{\zeta^1}  \right\}.
$$
Now we are ready to state the required adjustments: instead of using the list $(X, d)$ to induced the Carnot-Carath\'edory metric $\rho$, we use the list
\begin{equation} \label{correctSHTeq003}
\left( \left(\del^{\widetilde{N}} \cdot \widetilde{\zeta} \right)^d X,  d \right)
\end{equation}
to induce a Carnot-Carath\'edory metric $\rho_{\left(\widetilde{\zeta} \cdot \del^{\widetilde{N}}\right)^d X}$, which\footnote{Recall that $\rho_{\left(\widetilde{\zeta} \cdot \del^{\widetilde{N}}\right)^d X}(x, y)=\left(\widetilde{\zeta} \cdot \del^{\tilde{N}} \right)^{-1} \rho(x, y)$.}, together with Theorem \ref{uniformdecomp}, gives us a new space of homogeneous type 
\begin{equation} \label{correctSHTeq001}
(V, \rho_{\left(\widetilde{\zeta} \cdot \del^{\widetilde{N}}\right)^d X}, |\cdot|),
\end{equation}
with a dyadic decomposition,
\begin{equation} \label{correctSHTeq002}
\calD_{\widetilde{\zeta} \cdot \del^{\widetilde{N}}}=\bigcup_{k \in \Z} \calD_{\widetilde{\zeta} \cdot \del^{\widetilde{N}}, k}.
\end{equation}
In particular, we have
$$
\del_{\widetilde{\zeta} \cdot \del^{\widetilde{N}}}=\del, \ \epsilon_{\widetilde{\zeta} \cdot \del^{\widetilde{N}}}=\epsilon \quad \textrm{and} \quad \frakC_{\widetilde{\zeta}\cdot \del^{\widetilde{N}}}=\frac{\frakC}{\del}. 
$$
From now on, we shall use the space of homogeneous type \eqref{correctSHTeq001} and its dyadic decomposition \eqref{correctSHTeq002}. For simplicity and henceforth, we write 
\begin{equation} \label{1116eq02}
(W, d)
\end{equation}
to denote the list \eqref{correctSHTeq003}, 
$$
(V, \rho', |\cdot|)
$$
to denote the space of homogeneous type \eqref{correctSHTeq001} and 
\begin{equation} \label{190322eq01}
\calG=\bigcup_{k \in \Z} \calG_k
\end{equation}
for its dyadic decomposition \eqref{correctSHTeq002}.

\begin{rem}
In our most applications later, we take $U=\Omega$, $V=\Omega'$, $V'=\Omega''$, $K'=K_0$, $K''=K_0'$ and $K'''=K_1$, where the sets $0 \in K_1 \Subset K_0' \Subset K_0 \Subset \Omega'' \Subset \Omega' \Subset \Omega$ are defined in Section 4. Moreover, in the above setting, we can think $\frac{\frakC}{\del}$ as $\check{C}$ and $\frac{\frakC_1}{\del}$ as $\check{C_1}$ in the context there. 
\end{rem}

\begin{rem}
Without the loss of generality, we may also assume that the integer $\widetilde{N}$ fixed in \eqref{correctSHTeq003} is so large that $\del^{\widetilde{N}} \cdot \widetilde{\zeta}<\del_0$, where $\del_0$ is defined in \eqref{190404eq01}. This is to make sure that all the balls lie inside $V'$. 
\end{rem}

\medskip

\textit{Step III: Assumptions on $T$.}

\medskip

Recall the singular Radon transform $T$ is defined by
$$
Tf(x)=\psi(x) \int f(\gamma_t(x))K(t)dt,
$$
where $\psi$ is a $C_0^\infty$ cut-off function, supposed near $0 \in \R^n$ and $K$ is a standard Calder\'on-Zygmund kernel on $\R^k$, supported for $t$ near $0$. Here, we shall make these assumptions more quantitative as follows:
\begin{enumerate}
\item [(i).] $\textrm{supp} \psi \subseteq K'''$;
\item [(ii).] $\textrm{supp} K \subseteq B^k(a)$, where $a<\min \{a', a''\}$, $a', a''$ are defined in \eqref{supportofK01} and \eqref{supportofK02}, respectively. Moreover, we also require $a$ is small enough so that Lemma \ref{moduluscont} holds, when we apply it to any $x_0 \in K'''$, the list
$$
\left( \left(\frac{1}{\widetilde{\zeta}} \right)^d W, d \right), 
$$
and all  suitable\footnote{Recall that the although all the implict constants in Lemma \ref{moduluscont} depend on how $\theta$ is controlled at the unit scale, namely, the constants in $\calQ_1$. We will see later that the ``$\theta$"s will be choosen ``uniformally" for all $x_0 \in K'''$ and all scales $\del^j, j \ge 0$,  in the sense that the constants in $\calQ_1$ for all these ``$\theta$"s are uniform for all $x_0 \in V'''$ and all scales $\del^j, j \ge 0$. Therefore, we are able to make our assumptions uniformly for all ``$\theta$"s. 

However,  the exact quantitative description on how small ``$a$" is is quite involved, and we will come back to this point in the proof of our main result. Finally, we make a remark that at this moment, there is no harm to think $``a"$ is a small number which is less than $\min\{a', a''\}$, and is independent of the choice of $x_0 \in K'''$ and the scales $\del^j, j \ge 0$.} ``$\theta$''s . 

\end{enumerate}

The setting is complete. To this end, we should mention the following important result  by Christ, Nagel, Stein and Wainger in \cite{CNSW}. 

\begin{thm}[{\cite[Theorem 11.1]{CNSW}}] \label{1116thmLpbdd}
Under the above assumptions, the operator $T$ extends a bounded operator on $L^p(\R^n)$, $1<p<\infty$.
\end{thm}

As a conseuqence, we have the following result. 

\begin{cor} \label{1116corLpbdd}
For any $1<p<\infty$, 
$$
\sup_{N \ge 0} \left\| \sum_{j \ge N} \calT_j^{(j)} f \right\|_{L^p} \lesssim \|f\|_{L^p}, \quad \forall f \in L^p(\R^n),
$$
where the implicit constant in the above inequality can be choosen to be the same as the one in Theorem \ref{1116thmLpbdd}.
\end{cor}


\subsection{The main result} 

Before we state the main result, we recall and make a remark that in the rest of this section, as well as in next section,  by saying ``\emph{admissible constants}", we mean those constants derived from the scaling map $\Phi$ defined in Theorem \ref{scalingmap}, where $\Phi$ is applied to $x_0 \in K'''$ and the list \eqref{correctscalingeq01}, that is, by \eqref{1116eq02}, the list\footnote{In our application below, we will apply $\Phi$ to lots of centers $x_c(Q) \in K'''$ and lots of scales $\del^j$, namely, the list
$$
\left( \left( \frac{8\frakC_1}{\del \widetilde{\zeta}} \cdot \del' \del^j \right)^d W, d \right),
$$
where $j \ge 0$ and $\del' \in (0, 1]$. By our assumption and \eqref{1116eq16}, it is clear this new list shares the same defining conditions (that is, the constants in Definition \ref{admiconstdefn}) as the list \eqref{1116eq05}, and therefore, so are their admissible constants. 
}
\begin{equation} \label{1116eq05}
\left( \left( \frac{8\frakC_1}{\del \widetilde{\zeta}} \right)^d W, d \right). 
\end{equation}
By our assumption in Section 5.1, these constants can be choosen uniformly for $x_0 \in K'''$. Moreover, we refer the reader the remarks after \eqref{correctscalingeq01} for an explicit description on how these admissible consntants depend on our assumptions in Section 5.1. 

To begin with, we decompose the kernel $K$ as follows. 

\begin{lem} \label{CZdecomp123}
Let $K$ be a Calder\'on-Zygmund kernel, supported in $B^k(a) \subset \R^k$. Then for any $\del \in (0, 1)$, we have
$$
K=\sum_{j \ge 0} \chi_j^{(\del^{-j})}:=\sum_{j \ge 0} \del^{-kj} \chi_j (\del^{-j} \cdot)
$$
in the sense of distribution, where $\{\chi_j\}_{j \ge 0} \subseteq C_0^\infty (B^k(a))$ is a bounded set with
$$
\int_{\R^k} \chi_j(t)dt=0, \forall j \ge 1. 
$$
\end{lem}

\begin{proof}
The proof of the above result is an easy modifitication by replacing the role of $2$ by $\del^{-1}$ in the proof of \cite[Lemma 2.2.3]{NRS}. 
\end{proof}

Denote $\calB:=\{\chi_j\}_{j \ge 0}$ to be a bounded set in $C^\infty_0(B^k(a))$ and for any $\chi \in \calB$, define the \emph{single scale operator} as
$$
\calT_{\chi}f(x):= \psi(x) \int f(\gamma_t(x)) \chi(t)dt. 
$$
 Then an application of \cite[Theorem 8.11]{CNSW} together with Sobolev embedding, yields the following ``baby version" of  the ``$L^p$ improving property"\footnote{Here we require a bit general and quantitative version of the $L^p$ improving property, which was studied carefully in \cite{BS3}, and we will come back to it later.} : 

\medskip

\underline{\textit{``Baby Version": }} Given $r \in [1, \infty)$, there exists $s>r$, such that 
\begin{equation} \label{Lpimproving}
\sup_{ \chi \in \calB} \left\| \calT_j f\right\|_{L^s} \lesssim \| f\|_{L^r}, \quad \forall j \ge 0, 
\end{equation}
where the implicit constant in the above inequality only depends on $r, s$, the uniform $C^m$-norms of the family $\calB \subseteq C^\infty_0(B^k(a))$ and how $\gamma$ satisfies the condition $(\calC_J)$ (more precisely, see Theorem  \ref{apdthm002}). In particular, this constant is independent of the choice of kernel $K$. 

To this end, we denote
$$
\Sigma:=\left\{ \left(\frac{1}{r}, \frac{1}{s} \right): (r, s) \ \textrm{satisfies} \ \eqref{Lpimproving} \right\} 
$$
(see, Figure 2). 

\begin{figure}[ht]
\begin{tikzpicture}[scale=5]
\draw (0,0) [->] -- (0,1.2) node [left] {$\frac1s$};
\draw (0,0) [->] -- (1.2,0) node [below] {$\frac1r$};
\fill [opacity=.1,blue] (0, 0) -- (1, 1) -- (.8, .3) -- (.7, .2) -- cycle;
\draw  (.0, 0) -- (1, 1);
\draw (1, 1) -- (.8, .3);
\draw  (.8, .3)--(.7, .2); 
\draw (0, 0)--(.7, .2); 
\draw[dashed] (1, 1)--(0, 1);
\draw [dashed] (1, 0)--(1,  1); 
\fill (0, 1) circle [radius=.4pt] node [left] {\footnotesize{$1$}};
\fill (1, 0) circle [radius=.4pt];
\fill (1, -0.01) node[below] {\footnotesize{$1$}};
\draw (.6, .4)  node [right] {$\Sigma$};
\end{tikzpicture}
\caption{}
\end{figure}

Now we are ready to state our main result in this paper, that is, the sparse domination of singular Radon transform. 

\begin{thm} \label{MainTheorem}
Suppose the assumptions in Section 5.1 hold, and let $(r, s)$ be a pair such that $\left( \frac{1}{r}, \frac{1}{s} \right)$ belongs to the interior of  $\Sigma$. Then for any compactly supported bounded functions $f_1, f_2$ on $\R^n$, and any $0<\sigma<1$, there exists a $\sigma$-sparse collection $\calS$ of some dyadic system $\widetilde{\calG}$, such that\footnote{The dyadic system $\widetilde{\calG}$ may not be the same as $\calG$, which is defined in \eqref{190322eq01}. This will be clear from Theorem \ref{MainTheorem2} below.}
\begin{equation} \label{sparseq01}
|\langle Tf_1, f_2 \rangle| \lesssim \Lambda_{\calS, r, s'}(f_1, f_2).
\end{equation}
where the implicit constant in the above inequality only depends on $r, s$ and $T$. 
\end{thm}

As an immediate corollary of \cite[Section 6]{BFP}, we obtain weighted inequalities, holding in an appropriate intersection of Muckenhoupt and reverse H\"older weighted class. More precisely, for $1<p<\infty$ and the sparse of homogeneous type $(V, \rho', |\cdot|)$ with dyadic structure $\calG$, the Muckenhoupt type $A_p$ weights consisting of locally integrable positive functions $w$ on $V$ such that
$$
[w]_{A_p}:=\sup_{B \in \calG} \langle w \rangle_B \langle w^{1-p'} \rangle_B^{p-1}<\infty;
$$
while the reverse H\"older class $\textrm{RH}_p$ of weights contains locally integrable positive functions $w$ such that
$$
[w]_{\textrm{RH}_p}:=\sup_{B \in \calG} \langle w \rangle_B^{-1} \langle w \rangle_{B, p}<\infty. 
$$

\begin{cor} \label{weightedineq}
Suppose the assumptions in Theorem \ref{MainTheorem} hold. Then for any $r<p<s$ and weight $w \in A_{\frac{p}{r}} \cap \textrm{RH}_{\left(\frac{s}{p}\right)'}$, there holds
$$
\|T\|_{L^p(w) \rightarrow L^p(w)} \lesssim \left( [w]_{A_{\frac{p}{r}}}[w]_{\textrm{RH}_{\left(\frac{s}{p}\right)'}} \right)^\alpha, \quad \alpha:=\max\left\{ \frac{1}{p-r}, \frac{s-1}{s-p}\right\}. 
$$
\end{cor}

Actually, Theorem \ref{MainTheorem} follows directly form the following, slightly more general theorem.

\begin{thm} \label{MainTheorem1}
Suppose the assumptions in Section 5.1 hold. Then there exists $a>0$ such that for every $\psi_1, \psi_2 \in C_0^\infty(\R^n)$ supported in the interior of $K'''$, every $K$, a Calder\'on-Zygmund kernel supported in $C^\infty(B^k(a))$ and every $C^\infty$ function
$$
\varrho(t, x): B^k(a) \times V' \mapsto \C,
$$
the operator\footnote{From now on, the operator $T$ refers to this general expression.} 
\begin{equation} \label{refinedop}
T(f)(x)=\psi_1(x) \int f(\gamma_t(x)) \psi_2(\gamma_t(x)) \varrho(t, x)K(t)dt 
\end{equation}
satisfies the same conclusion as in Theorem \ref{MainTheorem}. 
\end{thm}

\begin{proof}[Proof of Theorem \ref{MainTheorem} given Theorem \ref{MainTheorem1}]
The proof for this claim is exactly the same as the argument in \cite[Theorem 7.1]{BS}, and hence we omit it here. 
\end{proof}

Therefore and henceforth, by applying  Lemma \ref{CZdecomp123} again, we shall also pay our attention to a more general single scale operator, which we still denote as $\calT_j$, and is defined as for each $j \ge 0$,
\begin{equation} \label{190310eq01}
\calT_j(f)(x)=\psi_1(x) \int f(\gamma_t(x)) \psi_2(\gamma_t(x)) \varrho(t, x) \chi_j(t)dt.
\end{equation}
Moreover, its $i$-th dilation is defined as 
\begin{eqnarray*} 
\calT_j^{(i)}(f)(x):%
&=&\psi_1(x) \int f(\gamma_t(x)) \psi_2(\gamma_t(x)) \varrho(t, x) \chi^{(\del^{-i})}_j(t)dt\\
&=& \psi_1(x) \int f(\gamma_{\del^i t}(x)) \psi_2(\gamma_{\del^i t}(x)) \varrho(\del^i t, x) \chi_j(t)dt.
\end{eqnarray*}
In particular, we have 
\begin{equation} \label{1114eq01}
T=\sum\limits_{j \ge 0} \calT_j^{(j)}
\end{equation}
in the sense of distribution. Also that the results in Theorem \ref{1116thmLpbdd} and Corollary \ref{1116corLpbdd} still hold. Here, the first part of the above claim is exactly \cite[Theorem 7.2]{BS}, and similarily, the second one follows from the $L^p$ boundedness of $T$ easily.

\medskip

\subsection{$L^p$ improving property} 

Finally, we shall state the exact version of the $L^p$ improving property for the above general case, which is an application of Theorem \ref{apdthm003}.

Note that to apply Theorem \ref{apdthm003}, we need to to construct four bounded sets $\calB_1$, $\calB_2$, $\calB_3$ and $\calB_4$. We would like to divide the construction into several steps. 

\medskip

\textit{Step I:} Let $\Phi$ be the scaling map defined in Theorem \ref{scalingmap} applied to $x_0 \in K'''$ and the list \eqref{1116eq05}. Note that by \eqref{1116eq02}, 
$$
\eqref{1116eq05}= \left( \left( \frac{8\frakC_1}{\del \widetilde{\zeta}} \cdot \del^{\widetilde{N}} \cdot \widetilde{\zeta} \right)^d X, d \right)= \left( \left( \frac{8\frakC_1}{\del} \cdot \del^{\widetilde{N}} \right)^d X, d \right)=\eqref{correctscalingeq01}. 
$$
By \eqref{1127eq01}, we have 
\begin{eqnarray*}
\Phi: B_{(Y, d)}(0, \widetilde{\zeta_0}) \longrightarrow%
&&  B_{ \left( \left( \frac{8\frakC_1}{\del} \cdot \del^{\widetilde{N}} \right)^d X, d \right)} (x_0, \widetilde{\zeta_0}) \\
&& =B_{ \left( \left( \frac{8\frakC_1}{\del \widetilde{\zeta}} \right)^d W, d \right)}  (x_0, \widetilde{\zeta_0}) \\
&&= B_{(W, d)} \left(x_0, \frac{8\frakC_1 \widetilde{\zeta_0}}{\del \widetilde{\zeta}} \right) \subset K'
\end{eqnarray*}
is a $C^\infty$ diffeomorphism, where $(Y, d)$ is the pullback of the list $ \left( \left( \frac{8\frakC_1}{\del \widetilde{\zeta}} \right)^d W, d \right)$ via $\Phi$. Now for the choosen $x_0 \in K'''$, consider the Carnot-Carath\'edory ball
$$
B_{(W, d)}\left(x_0, \frac{\frakC}{\del} \right), 
$$
which clearly is contained in $B_{(W, d)} \left(x_0, \frac{8\frakC_1 \widetilde{\zeta_0}}{\del \widetilde{\zeta}} \right)$ and hence also in  $K'$. Next for any $j \ge 0$ and $x \in B_{(W, d)}\left(x_0, \frac{\frakC}{\del} \right)$, we consider the quantity
$$
\calT_j(f)(x),
$$
where $f$ is some measurable function defined on $\R^n$ with\footnote{Here we only care about the value of $\calT_j(f)$ at those $x  \in B_{(W, d)}\left(x_0, \frac{\frakC}{\del} \right)$, therefore,  by \eqref{1125eq05}, we may assume $f$ is compactly supported in $B_{(W, d)}\left(x_0, \frac{\frakC_1}{\del} \right)$.}
$$
\textrm{supp}(f) \subseteq B_{(W, d)}\left(x_0, \frac{\frakC_1}{\del} \right) \subset B_{(W, d)}\left(x_0, \frac{\frakC_1 \widetilde{\zeta_0}}{\del \widetilde{\zeta}} \right).
$$ 

\medskip

\textit{Step II:} Using the scaling map $\Phi$, there exists some $u \in B_{(Y, d)}\left(0, \frac{\frakC\widetilde{\zeta}}{8\frakC_1} \right)$, such that $x=\Phi(u)$, and therefore, we can write
\begin{eqnarray} \label{1127eq06}
&&\calT_j(f)(x)= \calT_j(f)(\Phi(u)) \nonumber \\
&&=\psi_1(\Phi(u))\int \left(f \circ \Phi \right) \circ \left(\Phi^{-1} \circ \gamma_t \circ \Phi\right)(u) \nonumber \\
&& \quad \quad \quad \quad \quad \quad \quad \quad  \left(\psi_2 \circ \Phi \right) \circ \left(\Phi^{-1} \circ \gamma_t \circ \Phi\right)(u) \rho (t, \Phi(u)) \chi_j(t)dt \nonumber \\
&&= \widehat{\psi_1}(u) \int \widehat{f} (\widehat{\gamma}_t(u)) \widehat{\psi_2}(\widehat{\gamma}_t(u)) \widehat{\rho}(t, u)\chi_j(t)dt,
\end{eqnarray}
where 
$$
\widehat{\psi_1}(u):=\psi_1 \circ \Phi(u), \ \widehat{\psi_2}(u):=\psi_2 \circ \Phi(u), 
$$
$$
\widehat{f}(u):=f \circ \Phi(u),
$$
$$
\widehat{\gamma}_t(u):=\Phi^{-1} \circ \gamma_t \circ \Phi(u)
$$
and
$$
\widehat{\rho}(t, u):=\rho(t, \Phi(u))=\rho(t, x)
$$
as usual. Here, as in the proof of Lemma \ref{moduluscont}, we have \footnote{Note that  we can state these conditions in a more quantitative way based on the proof of  Lemma \ref{moduluscont}.  However, we also need this quantitative decription later in Step IV when we define the bounded sets $\calB_1, \calB_2, \calB_3$ and $\calB_4$. Thus, we postpone such a description there.}
\begin{enumerate}
\item [1.] $\widehat{\psi_1}, \widehat{\psi_2} \in C^\infty (B_{(Y, d)}(0, \widetilde{\zeta_0}))$;
\item [2.] $\widehat{f}$ is a measurable function defined on $B_{(Y, d)}(0, \widetilde{\zeta_0})$, compactly supported in $B_{(Y, d)}\left(0, \frac{\widetilde{\zeta_0}}{8} \right)$;
\item [3.] $\widehat{\gamma}_t: B_{(Y, d)}\left(0, \frac{3 \widetilde{\zeta_0}}{4} \right) \longrightarrow B_{(Y, d)}(0, \widetilde{\zeta_0})$ is a $C^\infty$ mapping, when $|t| \le a$, with satisfying the curvature condition $(\calC^u_J)$ with respect to some parameters \footnote{These parameters will be specified in Step IV below.} (see Theorem \ref{apdthm002}). This is due to the fact that $\gamma$ satisfies the curvature condition $(\calC_J)$;
\item [4.] $\widehat{\rho} \in C^\infty (\overline{B^k(a)} \times \overline{B_{(Y, d)}(0, \widetilde{\zeta_0})})$. 
\end{enumerate}

\medskip

\textit{Step III:} Our next goal is to define some operator ``$\calS_j$'' globally via \eqref{1127eq06}. Note that the quantity \eqref{1127eq06} is only defined for those $u \in  B_{(Y, d)}\left(0, \frac{3 \widetilde{\zeta_0}}{4} \right)$. This can be done by taking some cut-off function, as what we did in the proof Lemma \ref{moduluscont}. More precisely, by Corollary \ref{scalingmap01}, we can take two 2-admissible constants $\widetilde{\zeta'_0}>\widetilde{\zeta''_0}>0$, (in particular, these constants can be choosen to be indepedent of $x_0 \in K'''$ and $\calT_j$) such that
$$
B_{(Y, d)} \left(0, \frac{\widetilde{\zeta_0}}{8} \right) \subset B^n(\widetilde{\zeta''_0}) \subset  B^n(\widetilde{\zeta'_0}) \subset B_{(Y, d)} \left(0, \frac{\widetilde{\zeta_0}}{2} \right).  
$$
Let $\widetilde{h_1}$ be a bump function satisfying $0 \le h_1 \le 1$ and $h_1 \equiv 1$ on $B^n(\widetilde{\zeta''_0})$, compactly supported in $B^n(\widetilde{\zeta'_0})$. Here,  the $C^m$-norm of $\widetilde{h_1}$  only depends on $m$, any dimension constants, and the constants $\widetilde{\zeta_0'}, \widetilde{\zeta_0''}$ fixed above (therefore, also depends on $\frakC$, $\del$ and any $2$-admissible constants). 

Using these cut-off functions, for each $j$, we define
\begin{equation} \label{1128eq01}
\calS_j(g)(u):=(\widehat{\psi_1}h_1)(u) \int g(\widehat{\gamma}_t(u)) (\widehat{\psi_2}h_1)(\widehat{\gamma}_t(u)) \widehat{\rho}(t, u) \chi_j(t)dt, 
\end{equation}
where $g$ is any measurable function defined on $\R^n$. Clearly, $\widehat{\calT_j}(g)$ is well-defined for any $u \in \R^n$. 

\medskip

\textit{Step IV:} The next step is to construst some bounded sets of some function spaces, quantitatively, from the collection of operators $\{\calS_j\}_{j \ge 0}$. 

\begin{enumerate}
\item [(a).] The first bounded set can be derived Lemma \ref{CZdecomp123}. More precisely, since $\{\chi_j\}_{j \ge 0} \subseteq C^\infty_0( B^k(a))$, we have for each $m \in \N$, 
$$
\sup_{j \ge 0} \|\chi_j\|_{C^m_0( B^k(a))} \le C^m_{\{\chi_j\}},
$$
where $C_{\{\chi_j\}}^m$ is some absoulte constant, independent of $j$.

Define
\begin{eqnarray*}
\calB_1%
&:=&\bigg\{ \chi \in C^\infty_0(\R^k): \textrm{supp}(\chi) \subseteq \overline{B^k(a)}, \\
&& \quad \quad \quad \quad  \quad \|\chi\|_{C^m_0(\R^k)} \le C^m_{\{\chi_j\}}, m \ge 0 \bigg\} \subset C_0^\infty(\R^n). 
\end{eqnarray*}

\medskip

\item [(b).] The second bounded set is constructed from the functions $\widehat{\psi_1}h_1$ and $\widehat{\psi_2}h_1$. Since both $\widehat{\psi_1}h_1$ and $\widehat{\psi_2}h_1$ are smooth, compactly support in $B^n(\widetilde{\zeta'_0})$, we have for each $m \ge 0$, 
$$
\sup_{i=1, 2} \| \widehat{\psi_i} h_1 \|_{C^m(B^n(\widetilde{\zeta'_0}))} \le C_{\widehat{\psi_1}, \widehat{\psi_2}}^m,
$$
where $C^m_{\widehat{\psi_1}, \widehat{\psi_2}}$ only depends on the $C^m$-norm of $\psi_1$ and $\psi_2$, on the $2$-admissible constants $\widetilde{\zeta_0'}$ and $\widetilde{\zeta_0''}$, on the $C^m$-norm of $h_1$, and on any $m$-admissible constants.

Define
\begin{eqnarray*}
\calB_2%
&:=&\bigg\{ \psi \in  C^\infty_0(\R^n): \textrm{supp}(\psi) \subseteq \overline{B^n(\widetilde{\zeta'_0})}, \\
&&\quad \quad \quad \|\psi\|_{C_0^m(\R^n)} \le C_{\widehat{\psi_1}, \widehat{\psi_2}}^m, m \ge 0 \bigg\} \subset  C^\infty_0(\R^n). 
\end{eqnarray*}

\medskip

\item [(c).] The third bounded set can be constructed from $\widehat{\rho}$. First, by applying Corollary \ref{scalingmap01} again, we are able to take a $2$-admissible constant $\widetilde{\zeta_0'''}>0$,  (in particular, these constants can be choosen to be indepedent of $x_0 \in K'''$ and $\calT_j$) such that\footnote{Note that only those $u \in B^n( \widetilde{\zeta''_0})$ make a non-trivial contribution to $\calS_j$. Therefore, we only need to consider the restriction of $\widehat{\rho}$, as well as $\widehat{\gamma}$ below, on a smaller Euclidean ball $B^n(\widetilde{\zeta_0'''})$. }
$$
B_{(Y, d)} \left(0, \frac{\widetilde{\zeta_0}}{2} \right) \subset B^n(\widetilde{\zeta_0'''}) \subset B_{(Y, d)} \left(0, \frac{3\widetilde{\zeta_0}}{4} \right).
$$

 Now, for each $m \ge 0$, we have
$$
\|\widehat{\rho}\|_{C^m \left(\overline{B^k(a)} \times \overline{B^n(\widetilde{\zeta_0'''}) } \right)} \le C_{\widehat{\rho}}^m,
$$
where the constant $C_{\widehat{\rho}}^m$ only depends on the $C^m$-norm of $\rho$ and any $m$-admissible constants.

Define
\begin{eqnarray*}
\calB_3%
&:=& \bigg\{\widetilde{\rho} \in C^\infty\left(\overline{B^k(a)} \times \overline{B^n(\widetilde{\zeta_0'''}) } \right): \\
&& \quad \quad \quad \quad  \quad \quad  \|\widetilde{\rho}\|_{C^m  \left(\overline{B^k(a)} \times \overline{B^n(\widetilde{\zeta_0'''}) } \right)} \le C^m_{\widehat{\rho}}, m \ge 0 \bigg\}\\
& \subset& C^\infty\left(\overline{B^k(a)} \times \overline{B^n(\widetilde{\zeta_0'''})} \right).
\end{eqnarray*}

\medskip 

\item [(d).] The last bounded set is defined according to the condition $(\calQ_2)$ and the curvature condition $(\calC_J)$ (see Definition \ref{defnCJ}). First, recall that $\widehat{\gamma}_t$ is pullback of $\gamma_t$ via $\Phi$, where $\gamma_t$ is controlled by the list \eqref{1116eq05} at the unit scale, uniformly for $x \in K'$.  Therefore by $(\calQ_2)$, we have for each $m \ge 0$, 
$$
\|\widehat{\gamma} \|_{C^m \left( B^k(a) \times B^n(\widetilde{\zeta_0'''}) \right)} \le \sigma_m^2,
$$
where $\sigma_m^2$ is the constant defined in $(\calQ_2)$. Moreover, since $\gamma$ is assumed to satisfy $(\calC_J)$  in $K' \subset V$, using, for example, Theorem \ref{scalingmap} and Theorem \ref{apdthm002}, we can find some $\check{M} \ge 0$ and $\check{c}>0$, such that $\widehat{\gamma}$ satisfies the curvature condition $(\calC_J)_{\check{M}, \check{c}, a, \widetilde{\zeta_0'''}, \{\sigma_m^2\}_{m \in \N}}$ (see Theorem \ref{apdthm002}). 

Define
\begin{eqnarray*}
\calB_4%
&:=& \bigg\{ \widetilde{\gamma} \in C^\infty  \left( \overline{B^k(a)} \times \overline{B^n(\widetilde{\zeta_0'''})} \right): \widetilde{\gamma}(0, 0) \equiv 0 \ \textrm{and} \\
&&  \quad \quad \quad \quad  \quad \quad \widetilde{\gamma} \ \textrm{satisfies} \ (\calC_J)_{\check{M}, \check{c}, a, \widetilde{\zeta_0'''}, \{\sigma_m^2\}_{m \in \N}} \bigg\}\\
&\subset&  C^\infty  \left( \overline{B^k(a)} \times \overline{B^n(\widetilde{\zeta_0'''})} \right).
\end{eqnarray*}
\end{enumerate}

We now finish the quantitative construction of the four bounded sets. We make a remark that although the scaling map depends on the base point $x_0 \in K'''$, these four bounded sets $\calB_1, \calB_2, \calB_3$ and $\calB_4$ (more precisely, the parameters used to define them) can be constructed indepedent of the choice of $x_0 \in K'''$. 

\medskip

Applying Theorem \ref{apdthm003}, we have the following quantitative $L^p$ improving property. 

\medskip

\underline{\textit{$L^p$ improving property:}} Let $\chi \in \calB_1$, $\breve{\psi_1}, \breve{\psi_2} \in \calB_2$, $\widetilde{\rho} \in \calB_3$ and $\widetilde{\gamma} \in \calB_4$, define
$$
\calS (g)(u):=\breve{\psi_1}(u) \int g (\widetilde{\gamma}_t(u)) \breve{\psi_2} (\widetilde{\gamma}_t(u)) \widetilde{\rho}(t, u) \chi(t)dt.
$$
Then for any $r \in [1, \infty)$, there exists $s>r$, such that
\begin{equation} \label{190321eq01}
\sup_{ \chi \in \calB_1, \breve{\psi_1}, \breve{\psi_2} \in \calB_2, \widetilde{\rho} \in \calB_3, \widetilde{\gamma} \in \calB_4} \|\calS f \|_{L^s} \lesssim \|f\|_{L^r}
\end{equation}
where the implicit constant above only depends on $r, s$ and on all parameters used to define $\calB_1, \calB_2, \calB_3$ and $\calB_4$. 

Finally, we denote
$$
\Sigma:=\left\{ \left(\frac{1}{r}, \frac{1}{s} \right) \in \R^2: (r, s) \ \textrm{satisfies} \ \eqref{190321eq01} \right\}. 
$$
It is clear that by interpolation,  $\Sigma$ is a convex set in $\R^2$.

\begin{rem}
Note that the exact $L^p$ improving property is indeed the ``pullback" version of the early baby version. This allowed us to get some ``local $L^p$ improving estimate" with respect to each dyadic cube. 
\end{rem}

\bigskip

\section{Proof of the main result: Part I}

\bigskip

The next two sections are devoted to the proof of our main result Theroem \ref{MainTheorem1}. Indeed, we will prove a slightly different version of Theorem \ref{MainTheorem1}.

\begin{prop} \label{MainTheorem2}
Suppose the assumptions of Theorem \ref{MainTheorem1} hold. Then there exists $a>0$, such that  for any compactly supported bounded functions $f_1, f_2$ on $\R^n$, and any $0<\sigma<1$, there exists a $\sigma$-sparse collection $\calS$ of $\calG$, such that
$$
|\langle Tf_1, f_2 \rangle| \lesssim \Lambda_{\calS, r, s'}^{\kappa'} (f_1, f_2), 
$$
where 
$$
\Lambda_{\calS, r, s'}^{\kappa'} (f_1, f_2):= \sum_{S \in \calS} |S| \langle f_1 \rangle_{S, r} \langle f_2 \rangle_{\kappa' S, s'} 
$$
and $\kappa'>1$ is an abosulte constant which only depends on $\frakC_1$ and $\del$. 
\end{prop}

\begin{proof} [Proof of Theorem \ref{MainTheorem1} given Theorem \ref{MainTheorem2}]
By Theorem \ref{TLT} (since $\del<\frac{1}{100}<\frac{1}{96}$), for each $S \in \calS$, we can take a dyadic cube $Q_S \in \calG^i$ for some $i \in \{1, 2, \dots, K_0\}$, such that
\begin{equation} \label{190322eq02}
S \subseteq \kappa'S \subset Q_S \quad \textrm{and} \quad \ell(Q_S) \le \widetilde{\frakC} \kappa' \ell(S).
\end{equation} 
For each $i \in \{1, 2, \dots, K_0\}$, denote $\calS_i:=\left\{Q_S \in \calG^i, S \in \calS\right\}$, and we claim that $\calS_i$ is a $\sigma$-sparse collection. Indeed, this follows easily if we apply Theorem \ref{MainTheorem2} to some $\sigma' \in (0, 1)$ with $\sigma' \ll \sigma$. We, therefore, leave such details to the interested reader. Finally, by \eqref{190322eq02}, we have
\begin{eqnarray*} 
 \Lambda_{\calS, r, s'}^{\kappa'} (f_1, f_2)%
& \lesssim& \sum_{i=1}^{K_0}  \Lambda_{\calS^i, r, s'} (f_1, f_2)\\
&\le& K_0 \max_{1 \le i \le K_0}  \Lambda_{\calS^i, r, s'} (f_1, f_2).
\end{eqnarray*}
Clearly, this implies Theorem \ref{MainTheorem1}. 
\end{proof}

\subsection{Some reductions}

To begin with, recall that from the previous subsection, under our assumptions, we can construct a space of homogeneous type $(V, \rho', |\cdot|)$ and an associated dyadic system $\calG$, with parameters $\frac{\frakC}{\del}, \del, \epsilon$ and a corresponding collection of centers $\{x_c(Q)\}_{Q \in \calG}$. 

Without the loss of generality, we may assume that $\textrm{supp}(f_1), \textrm{supp}(f_2) \subseteq \footnote{Such an inclusion makes sense since we require $\textrm{supp}\psi_1, \textrm{supp}\psi_2 \Subset K'''$.} Q_0 \in \calD, Q_0 \Subset K'''$, where $\ell(Q_0)=\del^{j_0}$ with some fixed $j_0<0$. Moreover, we may also assume that in the summation
$$
Tf_1=\sum_{j \ge 0}\calT_j^{(j)} f_1,
$$
there are only finite many nonzero terms as our estimate will be uniform over all finite series. 

For $1 \le p<\infty$, we denote\footnote{Recall $(W, d)$ is the list which induces the Carnot-Carath\'edory metric $\rho'$ on $V$.}
$$
A^t_p(f)(x):=\left( \frac{1}{|B_{(W, d)}(x, t)|} \int_{B_{(W, d)}(x, t)} |f(y)|^pdy \right)^{\frac{1}{p}},
$$
$$
M_p (f)(x):=\sup_{x \in B_{(W, d)}(\tilde{x}, t), t>0} A^t_p(f)(\tilde{x}).
$$
Denote
$$
E_1:=\left\{x \in V: M_r (f_1)(x)>D\langle f_1\rangle_{Q_0, r} \right\},
$$
$$
E_2:=\left\{x \in V: M_{s'}(f_2)(x)>D\langle f_2 \rangle_{\kappa' Q_0, s'} \right\}.
$$
Here $D$ is large enough (depending on the weak $L^r$ and $L^{s'}$  bounds for the operators $M_r$ and $M_{s'}$ respectively) so that 
$$
|E| \le (1-\sigma)|Q_0|,
$$
where $E:=E_1 \cup E_2$, and $E \subseteq 10Q_0=B_{(W, d)}\left(x_c(Q_0), \frac{10\frakC}{\del} \cdot \del^{j_0}\right)$. Note that since $Q_0 \Subset K'''$, in particular, we have $x_c(Q) \in K'''$ and therefore, by \eqref{1109eq01}, this implies that
$$
E \subsetneq K'. 
$$
Moreover, for the choice of $\kappa'$, we will come back to it momentarily. 

By the Whitney decomposition (see Theorem \ref{Whitney}), we can write $E$  as a union of disjoint dyadic cubes\footnote{Recall that $M_E \subset \calD$ denotes the Whiteney decomposition associated to $E$.} 
$$
E=\bigcup_{Q \in M_E} Q, 
$$
each of which satisfies
\begin{equation} \label{Maineq01}
\left( \frac{\frakc'\del}{2\frakC}-1 \right) \textrm{diam}(Q) \le \textrm{dist}(Q, E^c) \le \frac{3\frakc'}{\del} \textrm{diam}(Q),
\end{equation}
for some $\frakc'>\frac{2\frakC}{\del}$ (Recall that for the space of homogeneous type $(V, \rho', |\cdot|)$, $\kappa=1$, $A$ can be taken the value $3$ and the constant ``$\frakC$'' used in Theorem \ref{Whitney} under our current setting takes the value $\frac{\frakC}{\del}$). We make a remark here that the choice of $\frakc'$ is not arbitrary, and we will come back to the choice of $\frakc'$ momentarily. 

\begin{lem} \label{Maineq02}
For any cube $Q'$ with $Q \subseteq Q' \in \calG, Q \in M_E$ and $r>1$, 
\begin{equation} 
\langle f_1 \rangle_{Q', r} \lesssim \langle f_1 \rangle_{Q_0, r},
\end{equation}
where the implicit constant in the above inequality only depends on $\frakC$ and $\del$. 
\end{lem}

\begin{proof}
Indeed, from the construction of Whitney decomposition, for any $Q \in M_E$, there exists some $D_0>1$, which is independent of the choice $Q$, such that $D_0Q \cap E^c \neq \emptyset$, and hence $D_0Q' \cap E^c \neq \emptyset$. Thus,  
\begin{equation} \label{modyeq01}
\langle f_1 \rangle_{Q', r} \lesssim \langle f_1 \rangle_{D_0Q', r} \lesssim \langle f_1 \rangle_{Q_0, r}. 
\end{equation}
\end{proof} 

Perform a Calder\'on-Zygmund decomposition of $f_1$ with respect to the collection of Whitney cubes
\begin{eqnarray} \label{Maineq04}
f_1%
&=&g_1+\sum_{Q \in M_E} \one_Q \left(f_1-\langle f_1 \rangle_Q\right) \nonumber\\
&=& g_1+\sum_{Q \in M_E} b_{1, Q}  \nonumber \\
&=& g_1+\sum_{Q \in M_E, Q \subseteq Q_0} b_{1, Q} \nonumber\\
&=& g_1+\sum_{k \ge j_0+1} b_{1, k},
\end{eqnarray}
where 
\begin{equation} \label{1114eq02}
b_{1, Q}:=  \one_Q \left(f_1-\langle f_1 \rangle_Q\right)  \quad \textrm{and} \quad  b_{1, k}:=\sum_{Q \in M_E, Q \subset Q_0, \ell(Q)=\del^k} b_{1, Q}. 
\end{equation}
Here in the second last equation in the above decomposition, we use the facts that $\textrm{supp}(f_1) \subset Q_0$ and $Q_0 \nsubseteq E$ (since $|E| \le (1-\sigma)|Q_0|$), which implies that if $Q \cap Q_0 \neq \emptyset$, then $Q \subseteq Q_0$; while in the last equation, we use the fact that $Q_0$ is not a Whitney cube and hence the index starts from $j_0+1$. Moreover, without the loss of generality, we may assume again that the second summation in \eqref{Maineq04} is finite with $k \le K_0$ for $K_0$ sufficiently large, again, which is because our estimate will be independent of the choice $K_0$ and we can pass the limit at the final stage in our estimation.

Thus, as a consequence of Lebesgue differentiation theorem (see Proposition \ref{LDT}), the ``good function" is bounded, namely
\begin{equation} \label{1114eq03}
\|g_1\|_{L^\infty} \lesssim \langle f_1 \rangle_{Q_0, r}. 
\end{equation}

For the ``bad function $b_{1, Q}$", we have the following easy observation. 

\begin{lem} \label{modyeq02}
For any cube $Q'$ and $r \ge 1$, 
\begin{equation} 
\left\| \sum_{Q \subseteq Q', Q \in M_E} b_{1, Q} \right\|_{L^r(Q')} \lesssim  \|f_1\|_{L^r(Q')}.
\end{equation}
\end{lem}

\begin{proof}
Indeed, using the fact that Whitney cubes are disjoint, we have
\begin{eqnarray*}
\int_{Q'} \left| \sum_{Q \subset Q', Q \in M_E} b_{1, Q} \right|^rdx%
&=& \sum_{Q \subset Q', Q \in M_E} \int_{Q} |b_{1, Q}|^rdx\\ 
&=&  \sum_{Q \subset Q', Q \in M_E} \int_{Q} \left| f_1-\langle f_1 \rangle_Q \right|^rdx \\ 
&\lesssim& \sum_{Q \subset Q', Q \in M_E} \int_Q |f_1|^rdx+\sum_{Q \subset Q', Q \in M_E} |Q| \langle f_1 \rangle_Q^r \\
&\lesssim& \int_{Q'} |f_1|^r dx. 
\end{eqnarray*}
\end{proof}

By \eqref{Maineq04}, \eqref{1114eq01} and \eqref{1114eq02}, we decompose
\begin{eqnarray} \label{Maineq03}
|\langle Tf_1, f_2 \rangle|%
&\le&|\langle Tg_1, f_2 \rangle|+\left| \sum_{k \ge j_0+1} \langle Tb_{1, k}, f_2 \rangle \right| \nonumber \\
&=& |\langle Tg_1, f_2 \rangle|+\left| \sum_{k \ge j_0+1} \sum_{j \ge 0} \langle \calT_j^{(j)} b_{1, k}, f_2 \rangle \right| \nonumber\\
&=&  |\langle Tg_1, f_2 \rangle|+\left| \sum_{k \ge j_0+1} \sum_{j \ge 0} \sum_{Q \subset Q_0, Q \in M_E, \ell(Q)=\del^k}\langle \calT_j^{(j)} b_{1, Q}, f_2 \rangle \right|  \nonumber\\
&:=& |\langle Tg_1, f_2 \rangle|+I, 
\end{eqnarray}
where 
$$
I:=\left| \sum_{k \ge j_0+1} \sum_{j \ge 0} \sum_{Q \subset Q_0, Q \in M_E, \ell(Q)=\del^k}\langle \calT_j^{(j)} b_{1, Q}, f_2 \rangle \right|.
$$
We estimate $|\langle Tg_1, f_2 \rangle|$ first. Using \eqref{1114eq03} and the $L^s$ boundedness of $T$ (see Theorem \ref{1116thmLpbdd}), we have
\begin{eqnarray} \label{Maineq05}
|\langle Tg_1, f_2 \rangle|%
&\le& \|Tg_1\|_{L^s} \|f_2\|_{L^{s'}} \nonumber \\
&\lesssim& \|g_1\|_{L^s} \|f_2\|_{L^{s'}} \nonumber \\
&\lesssim& |Q_0| \langle |f_1| \rangle_{Q_0, r} \langle |f_2| \rangle_{Q_0, s'}.
\end{eqnarray}

Next we estimate $I$. For any $Q \in M_E$, $Q \subset Q_0$ and a bounded, compactly supported function $f$, we denote

\begin{equation} \label{1114eq03}
T_Q(f):=\begin{cases}
\sum\limits_{\del^j \le \ell(Q), j \ge 0} \calT_j^{(j)} (f \one_Q), & \log_\del(\ell(Q))>0; \\
 \\
T(f \one_Q), & \log_\del(\ell(Q)) \le 0.
\end{cases} 
\end{equation} 

Then using \eqref{1114eq01}, \eqref{1114eq02} and \eqref{1114eq03}, we can write
\begin{eqnarray*}
I%
&=&\Bigg| \sum_{j_0+1 \le k<0} \sum_{j \ge 0} \sum_{Q \subset Q_0, Q \in M_E, \ell(Q)=\del^k}\langle \calT_j^{(j)} b_{1, Q}, f_2 \rangle \\
&& \quad \quad \quad \quad \quad \quad+\sum_{k \ge 0} \sum_{j \ge 0} \sum_{Q \subset Q_0, Q \in M_E, \ell(Q)=\del^k}\langle \calT_j^{(j)} b_{1, Q}, f_2 \rangle \Bigg|\\
&=& \left| \sum_{Q \in M_E, Q \subseteq Q_0, j_0+1 \le \log_\del (\ell(Q))<0}\langle T b_{1, Q}, f_2 \rangle+ \sum_{k \ge 0} \sum_{j \ge 0} \langle \calT_j^{(j)} b_{1, k}, f_2 \rangle \right|\\
&=&  \Bigg| \sum_{Q \in M_E, Q \subseteq Q_0, j_0+1 \le \log_\del (\ell(Q))<0}\langle T_Q(f_1 \one_Q), f_2 \rangle\\
&&\quad \quad \quad \quad \quad \quad \quad   - \sum_{Q \in M_E, Q \subseteq Q_0, j_0+1 \le \log_\del (\ell(Q))<0}\langle f_1 \rangle_Q \langle T_Q(\one_Q), f_2 \rangle\\
&&\quad \quad \quad \quad \quad \quad \quad   + \sum_{k \ge 0}\sum_{0 \le j \le k} \langle \calT_j^{(j)} b_{1, k}, f_2 \rangle+\sum_{k \ge 0}\sum_{j>k} \langle \calT_j^{(j)} b_{1, k}, f_2 \rangle  \Bigg| \\
&=&  \Bigg| \sum_{Q \in M_E, Q \subseteq Q_0, j_0+1 \le \log_\del (\ell(Q))<0}\langle T_Q(f_1 \one_Q), f_2 \rangle\\
&&\quad \quad \quad \quad \quad \quad \quad   - \sum_{Q \in M_E, Q \subseteq Q_0, j_0+1 \le \log_\del (\ell(Q))<0}\langle f_1 \rangle_Q \langle T_Q(\one_Q), f_2 \rangle\\
&& \quad \quad \quad \quad \quad \quad \quad + \sum_{k \ge 0}\sum_{0 \le j \le k} \langle \calT_j^{(j)} b_{1, k}, f_2 \rangle \\
&& \quad \quad \quad \quad \quad \quad \quad  +\sum_{k \ge 0} \sum_{Q \subset Q_0, Q\in M_E, \ell(Q)=\del^k} \sum_{j>k} \langle \calT_j^{(j)} b_{1, Q}, f_2 \rangle  \Bigg| ,
\end{eqnarray*}
where in the second equation, we use the fact that 
$$
T_Q(b_{1, Q})=T(b_{1, Q}), \quad \textrm{for} \ \log_\del(\ell(Q)) \le 0,
$$ 
and in the last equation, we use our previous assumption that the sum $\sum\limits_{j \ge 0} \calT_j^{(j)}$ is a finite sum, also for a fixed $k$, the set $E$ contains finite many Whitney cubes of the sidelength $\del^k$ and the Whitney cubes are pairwisely disjoint, thus, we are able to switch the two summations
$$
\sum_{j>k} \quad \textrm{and} \quad \sum_{Q \subset Q_0, Q \in M_E, \ell(Q)=\del^k}.
$$
Hence, we have
\begin{eqnarray*}
I%
&=&  \Bigg| \sum_{Q \in M_E, Q \subseteq Q_0, j_0+1 \le \log_\del (\ell(Q))<0}\langle T_Q(f_1 \one_Q), f_2 \rangle\\
&&\quad \quad \quad \quad \quad \quad \quad   - \sum_{Q \in M_E, Q \subseteq Q_0, j_0+1 \le \log_\del (\ell(Q))<0}\langle f_1 \rangle_Q \langle T_Q(\one_Q), f_2 \rangle\\
&&\quad \quad \quad \quad \quad \quad \quad   + \sum_{k \ge 0}\sum_{0 \le j \le k} \langle \calT_j^{(j)} b_{1, k}, f_2 \rangle+\sum_{k \ge 0} \sum_{Q \subset Q_0, Q\in M_E, \ell(Q)=\del^k} \langle T_Q b_{1, Q}, f_2 \rangle  \Bigg|\\
&=&  \Bigg| \sum_{Q \in M_E, Q \subseteq Q_0, j_0+1 \le \log_\del (\ell(Q))<0}\langle T_Q(f_1 \one_Q), f_2 \rangle\\
&&\quad \quad \quad \quad \quad \quad \quad   - \sum_{Q \in M_E, Q \subseteq Q_0, j_0+1 \le \log_\del (\ell(Q))<0} \langle f_1 \rangle_Q \langle T_Q(\one_Q), f_2 \rangle\\
&&+ \sum_{k \ge 0}\sum_{0\le j \le k} \langle \calT_j^{(j)} b_{1, k}, f_2 \rangle\\
&& +\sum_{k \ge 0} \sum_{Q \subset Q_0, Q\in M_E, \ell(Q)=\del^k} \langle T_Q(f_1 \one_Q), f_2 \rangle\\
&&\quad \quad \quad \quad \quad \quad \quad -\sum_{k \ge 0} \sum_{Q \subset Q_0, Q\in M_E, \ell(Q)=\del^k} \langle f_1 \rangle_Q \langle T_Q(\one_Q), f_2 \rangle   \Bigg|\\
&&\le  \left |\sum_{Q \in M_E, Q \subset Q_0} \langle T_Q(f_1 \one_Q), f_2 \rangle \right|+  \left|\sum_{k \ge 0}\sum_{0 \le j \le k} \langle \calT_j^{(j)} b_{1, k}, f_2 \rangle\right|\\
&& \quad \quad \quad \quad \quad \quad \quad +\left| \sum_{Q \in M_E, Q \subset Q_0} \langle f_1 \rangle_Q \langle T_Q(\one_Q), f_2 \rangle \right|\\
&&=I_1+I_2+I_3,
\end{eqnarray*}
where
$$
I_1:= \left |\sum_{Q \in M_E, Q \subset Q_0} \langle T_Q(f_1 \one_Q), f_2 \rangle \right|,
$$
$$ 
 I_2:=  \left|\sum_{k \ge 0}\sum_{0 \le j \le k} \langle \calT_j^{(j)} b_{1, k}, f_2 \rangle\right|,
$$
and
$$
I_3:=\left| \sum_{Q \in M_E, Q \subset Q_0} \langle f_1 \rangle_Q \langle T_Q(\one_Q), f_2 \rangle \right|. 
$$

\medskip

For the term $I_1$, we note that it is exactly of the same form as the original expression $\langle T(f_1), f_2 \rangle$, and hence it will go to recursion. More precisely, by replacing $Q_0$ in the proof of the sparse bound of $\langle T(f_1), f_2 \rangle$ by some $Q \in M_E, Q \subset Q_0$, we have there exists a sparse family $\calS_Q$, such that
$$
\left| \langle T_Q(f_1), f_2 \rangle \right| \lesssim \Lambda^{\kappa'}_{\calS_Q, r, s'} (f_1, f_2).
$$
Here $\kappa'$ is an absolute constant such that
\begin{equation} \label{1116eq022}
\textrm{supp}(T_Q(f_1)) \subseteq B_{(W, d)}(x_c(Q), \kappa'  \ell(Q))=\kappa' Q,
\end{equation}
where $\kappa'$ only depends on the constants $a, \frakC, \del, \epsilon$, any $2$-admissible constants and how $\gamma$ is controlled by the list $(W, d)$ (more precisely, $\calQ_2$). In particular, $k'$ is independent of the choice of the center $x_c(Q)$ and the scale $\ell(Q)$. For example, by \eqref{supportofK02}, one can take 
$$
\kappa'=\frac{\frakC_1}{\del}. 
$$

From now on, we shall fix our choice of $\kappa'$. Then the desired sparse collection $\calS$ will be
$$
\calS:=\{Q_0\} \bigcup \left( \bigcup_{Q \in M_E, Q \subset Q_0} \calS_Q \right),
$$
which is due to the fact that $|E| \le (1-\sigma)|Q_0|$. 

\medskip

Next, we estimate $I_3$, which is straighforward. Indeed, by Corollary \ref{1116corLpbdd},  Lemma \ref{Maineq02} and  the disjointness of Whitney cube , we have
\begin{eqnarray} \label{Maineq07}
I_3%
&=& \left| \sum_{Q \in M_E, Q \subset Q_0} \langle f_1 \rangle_Q \langle T_Q(\one_Q), f_2 \rangle \right| \nonumber\\
&\lesssim& \sum_{Q \in M_E, Q \subset Q_0} \langle f_1 \rangle_{Q, r} |Q|^{1/s} \|f_2\|_{L^{s'}(\kappa' Q)} \nonumber \\
&\lesssim& |Q_0| \langle f_1 \rangle_{Q_0, r} \langle f_2 \rangle_{\kappa' Q_0, s'}.
\end{eqnarray}

\medskip

Finally, we estimate $I_2$, which is the main estimation in our proof. Recall that we are aiming to show that
\begin{equation} \label{Maineq06}
I_2= \left|\sum_{k \ge 0}\sum_{0 \le j \le k} \langle \calT_j^{(j)} b_{1, k}, f_2 \rangle\right| \lesssim |Q_0| \langle f_1 \rangle_{Q_0, r} \langle f_2 \rangle_{\kappa' Q_0, s'}. 
\end{equation}

Perform a Calder\'on-Zygmund decomposition of $f_2$ with respect to $M_E$, we have
\begin{eqnarray} \label{Maineq08}
f_2%
&=& g_2+ \sum_{Q \in M_E} \one_Q \left(f_2-\langle f_2 \rangle_Q \right) \nonumber\\
&=& g_2+\sum_{Q \in M_E} b_{2, Q} \nonumber \\
&=&g_2+\sum_{Q \in M_E, Q \subset Q_0} b_{2, Q} \nonumber \\
&=& g_2+\sum_{k' \ge j_0+1} b_{2, k'},
\end{eqnarray}
where 
\begin{equation} \label{1116eq001}
b_{2, Q}:=\one_Q \left(f_2-\langle f_2 \rangle_Q \right) \ \textrm{and} \ b_{2, k'}:= \sum_{Q \in M_E, Q \subset Q_0, \ell(Q)=\del^{k'}} b_{2, Q}.
\end{equation}
Here, again, in the second sum of \eqref{Maineq08}, we may assume that the summation is finite if needed. Clearly, the ``good" function is bounded by 
\begin{equation} \label{Maineq09}
\|g_2\|_{L^\infty} \lesssim  \langle f_2 \rangle_{\kappa' Q_0, s'}.
\end{equation}
Then, we can write $I_2$ as follows.
\begin{eqnarray*}
I_2%
&=&  \left|\sum_{k \ge 0}\sum_{0 \le j \le k} \left\langle \calT_j^{(j)} b_{1, k},  g_2+\sum_{k' \ge j_0+1} b_{2, k'} \right\rangle\right| \\
&\le& \left|\sum_{k \ge 0}\sum_{0 \le j \le k} \left\langle \calT_j^{(j)} b_{1, k},  g_2 \right\rangle \right|+\left| \sum_{k \ge 0}\sum_{0 \le j \le k} \left\langle \calT_j^{(j)} b_{1, k}, \sum_{k' \ge j_0+1} b_{2, k'} \right\rangle\right| \\
&=& I_{2, 1}+I_{2, 2},
\end{eqnarray*}
where 
$$
I_{2, 1}:= \left|\sum_{k \ge 0}\sum_{0 \le j \le k} \left\langle \calT_j^{(j)} b_{1, k},  g_2 \right\rangle \right|
$$
and
$$
I_{2, 2}:=\left| \sum_{k \ge 0}\sum_{0 \le j \le k} \left\langle \calT_j^{(j)} b_{1, k}, \sum_{k' \ge j_0+1} b_{2, k'} \right\rangle\right| .
$$

For an estimation of $I_{2, 1}$,  using the fact that $\textrm{supp} g_2 \subseteq 100Q_0$, the $L^r$ boundedness of $T$ and $T_Q$, \eqref{modyeq01}, \eqref{1114eq02},  \eqref{Maineq09}, Lemma \ref{modyeq02}, \eqref{1116eq022}  and the disjointness of Whitney cubes, we have
\begin{eqnarray*}
I_{2, 1}%
&=& \left| \sum_{k \ge 0} \sum_{0 \le j \le k} \sum_{Q \in M_E, Q \subset Q_0, \ell(Q)=\del^k} \langle \calT_j^{(j)} b_{1, Q}, g_2 \rangle \right| \\
&=& \left| \sum_{Q \in M_E, Q \subset Q_0, \ell(Q) \le 1} \left\langle (T-T_Q)(b_{1, Q}), g_2 \right\rangle \right| \\
&\le& \left|\left\langle T \left(  \sum_{Q \in M_E, Q \subset Q_0, \ell(Q) \le 1} b_{1, Q} \right), g_2 \right\rangle \right|\\
&& \quad \quad \quad \quad \quad \quad \quad \quad \quad \quad \quad \quad \quad \quad + \left| \sum_{Q \in M_E, Q \subset Q_0, \ell(Q) \le 1} \left\langle T_Q(b_{1, Q}), g_2 \right\rangle \right|  \\
&\lesssim& \left\| \sum_{Q \in M_E, Q \subset Q_0} b_{1, Q} \right\|_{L^r} \|g_2\|_{L^{r'}}\\
&& \quad \quad \quad \quad \quad \quad \quad \quad \quad \quad +\sum_{Q \in M_E, Q \subseteq Q_0} \|b_{1, Q}\|_{L^r} \left\|g_2 \one_{\kappa' Q}\right\|_{L^{r'}}\\
&\lesssim& |Q_0| \langle f_1 \rangle_{Q_0, r} \langle f_2 \rangle_{\kappa' Q_0, s'}
\end{eqnarray*}

For an estimation of $I_{2, 2}$, we need to make a further decomposition. 
\begin{eqnarray*}
I_{2, 2}%
&=& \left| \sum_{k \ge 0}\sum_{0 \le j \le k} \left\langle \calT_j^{(j)} b_{1, k}, \sum_{k' \ge j_0+1} b_{2, k'} \right\rangle\right| \\
&\le& \left| \sum_{k \ge 0}\sum_{0 \le j \le k} \left\langle \calT_j^{(j)} b_{1, k}, \sum_{j_0+1 \le k'<j} b_{2, k'} \right\rangle\right|+\left| \sum_{k \ge 0}\sum_{0 \le j \le k} \left\langle \calT_j^{(j)} b_{1, k}, \sum_{k' \ge j} b_{2, k'} \right\rangle\right| \\
&=& \left| \sum_{Q, Q' \in M_E, Q, Q' \subset Q_0} \sum_{j \ge 0, \ell(Q) \le \del^j \le \ell(Q')} \left\langle \calT_j^{(j)}(b_{1, Q}), b_{2, Q'} \right\rangle \right| \\
&& \quad \quad \quad \quad \quad \quad \quad \quad \quad \quad \quad \quad \quad \quad+ \left| \sum_{k \ge 0}\sum_{0 \le j \le k} \sum_{k' \ge j} \left\langle \calT_j^{(j)} b_{1, k},  b_{2, k'} \right\rangle\right| \\
& \le &  \sum_{Q, Q' \in M_E, Q, Q' \subset Q_0} \sum_{j \ge 0, \ell(Q) \le \del^j \le \ell(Q')}  \left|\left\langle \calT_j^{(j)}(b_{1, Q}), b_{2, Q'} \right\rangle \right| \\
&& \quad \quad \quad \quad \quad \quad \quad \quad \quad \quad \quad \quad \quad \quad+ \left| \sum_{j \ge 0}\sum_{k \ge j} \sum_{k' \ge j} \left\langle \calT_j^{(j)} b_{1, k},  b_{2, k'} \right\rangle\right| \\
&=& J_1+J_2,
\end{eqnarray*}
where 
$$
J_1:= \sum_{Q, Q' \in M_E, Q, Q' \subset Q_0} \sum_{j \ge 0, \ell(Q) \le \del^j \le \ell(Q')}  \left|\left\langle \calT_j^{(j)}(b_{1, Q}), b_{2, Q'} \right\rangle \right| 
$$
and 
$$
J_2:= \left| \sum_{j \ge 0}\sum_{k \ge j} \sum_{k' \ge j} \left\langle \calT_j^{(j)} b_{1, k},  b_{2, k'} \right\rangle\right|.
$$
The rest of this section is devoted to estimate $J_1$ and $J_2$. 

\subsection{Estimation of $J_1$}

We divide the estimation of $J_1$ into two steps.

\medskip

\underline{Step I: Fix choice of $\frakc'$ in the Whitney decomposition in \eqref{Maineq01}.}

\medskip

We start with analysing a single term in $J_1$. Namely, for any $Q, Q' \in M_E, Q, Q' \subset Q_0$ and $j \ge 0$ satisfying $\ell(Q)=\del^k \le \del^j \le \del^{k'}=\ell(Q')$, if the expression
\begin{eqnarray} \label{1124eq10}
\langle \calT_j^{(j)} (b_{1, Q}), b_{2, Q'} \rangle%
&=& \int_{\R^n}  \int_{\R^k} \left(\psi_1  b_{2, Q'} \right)(x)\left(b_{1, Q}  \psi_2 \right) \left(\gamma_{\del^jt}(x)\right) \varrho(\del^j t, x) \chi_j(t) dtdx \nonumber\\
&\neq& 0,
\end{eqnarray}
then there exists some $x \in Q'$ such that $x':=\gamma_{\del^j t}(x) \in Q$ for some $t \in B^k(a)$.

Recall from Section 5.1 that we are using the Carnot-Carath\'edory metric $\rho'$, which is induced by the list $(W, d)$. 

\begin{lem}
For any $t \in B^k(a)$ and $x \in K'''$,  
\begin{equation} \label{1124eq01}
\rho'(x, x')=\rho'\left(x, \gamma_{\del^jt} (x)\right) \le \frakC^* \del^j,  
\end{equation}
where $\frakC^*$ is some $2$-admissible constant. Moreover, $\frakC^* \le \frac{2\frakC_1}{\del}$. 
\end{lem}

\begin{proof}
Indeed, the first conclusion in the above claim follows from the first condition in $(\calQ_2)$ directly. Here, we would like to give a proof of the above claim without using that condition. First, note that it suffices to show
$$
\rho'_{(\del^j)^d W}(x, \gamma_{\del^j t}(x)) \le \frakC^*,
$$
where $\rho'_{(\del^j)^d W}$ is the Carnot-Carath\'edory metric induced by the list $\left( (\del^j)^d W, d\right)$ and $\frakC^*$ is some admissible constant. Let $\Phi$ be the scaling map defined in Theorem \ref{scalingmap} applied to the point $x$ and the list\footnote{Note that this list shares the same defining constants in Definition \ref{admiconstdefn} with the list \eqref{1116eq05}, and so are all the admissible constants.} $\left( (\del^j)^d W, d\right)$. Then we have
\begin{eqnarray*}
\rho'_{(\del^j)^d W} (x, \gamma_{\del^j t}(x))%
&=& \rho'_{(\del^j)^d W} \left( \Phi(0), \gamma_{\del^j t}\left( \Phi (0) \right) \right)\\
&=& \rho'_{(\del^j)^d W} \left( \Phi(0), \Phi \left( \hat{\gamma}_{\del^j t}(0) \right) \right) \\
&=& \rho'_{ \Phi^* \left( (\del^j)^d W\right)} \left(0, \hat{\gamma}_{\del^j t}(0) \right) \\
&& \quad  (\textrm{by \eqref{1207eq01}}.) \\
&\le& \frakC^*,
\end{eqnarray*}
where in the second equation, the mapping $\hat{\gamma}_{\del^j t}:=\Phi^{-1} \circ \gamma_{\del^j t} \circ \Phi$, namely, the pullback of the mapping $\gamma_{\del^j t}$ via $\Phi$; in the last equation, $\Phi^*  \left( (\del^j)^d W \right)$ denotes the pullback of the list $ \left( (\del^j)^d W\right)$ via $\Phi$; and in the last estimate, we use the fact that $\hat{\gamma}_{\del^j t}$ is a $C^\infty$ function. Finally, the estimate
$$
\frakC^* \le \frac{2\frakC_1}{\del}
$$
follows directly from \eqref{1125eq05}. 
\end{proof}

By \eqref{Maineq01}, we can take a point $y \in E^c$ and $x'' \in Q$, such that
\begin{equation} \label{Maineq10}
\rho(x'', y) \le \frac{6\frakc'}{\del} \textrm{diam}(Q)
\end{equation}
(see, Figure 3).

\begin{figure}[ht]
\begin{tikzpicture}[scale=5]

\draw (0,0) .. controls (0,.8) and (.8,0) .. (.8,.8);
\fill (.015, .01) node[right]{$E$}; 
\draw (.4, .17) circle (0.18);
\fill (0.3, 0.6) circle [radius=.25pt] node[above]{\footnotesize{$y$}};
\fill (0.34, 0.13) circle [radius=.25pt] node[below]{\footnotesize{$x''$}};
\fill (0.43, 0.08) circle [radius=.25pt] node[right]{\footnotesize{$x'$}};
\draw (.18, .2) node[below]{\footnotesize{$Q$}}; 
\draw (1.04, .3) circle (0.25);
\fill (0.9, 0.33) circle [radius=.25pt] node[right]{\footnotesize{$x$}};
\draw (.99, .64) node[below]{\footnotesize{$Q'$}}; 
\draw (.9, .33) -- (.43, .08); 
\draw (.9, .33)--(.3, .6);
\draw (.3, .6)-- (.34, .13);
\draw (.34, .13)--(.43, .08);
\draw [dashed] (.43, .08) -- (.3, .6);
\end{tikzpicture}
\caption{}
\end{figure}

Then using triangle inequality and \eqref{1124eq01}, we have
\begin{eqnarray} \label{1124eq02}
\rho'(x, y)%
&\le&  \rho(x, x')+\rho'(x', y) = \rho'(x, \gamma_{\del^j t}(x))+\rho'(x', y) \nonumber \\
&\le& \frakC^* \del^j+\rho'(x', y).
\end{eqnarray}
By \eqref{Whitney01} and \eqref{Maineq10}, we have 
\begin{equation} \label{1124eq03}
\rho'(x', y) \le \rho'(x', x'')+\rho'(x'', y) \le  \frac{2\frakC}{\del} \cdot \del^k+\frac{2\frakC}{\del} \cdot \del^k \cdot \frac{6\frakc'}{\del}.
\end{equation}
Also note that by \eqref{Whitney01}, we have
\begin{equation} \label{1124eq04}
\rho'(x, y) \ge \textrm{dist}(Q', E^c) \ge \left(\frac{\frakc'\del}{2\frakC}-1 \right) \cdot \frac{\del^{k'}}{3}.
\end{equation} 
Therefore, combining \eqref{1124eq02}, \eqref{1124eq03} and \eqref{1124eq04}, we have
\begin{equation} \label{Maineq11}
\left(\frac{\frakc'\del}{2\frakC}-1 \right) \cdot \frac{\del^{k'}}{3} \le \frakC^* \del^j+\frac{2\frakC}{\del} \cdot \del^k+\frac{2\frakC}{\del} \cdot \del^k \cdot \frac{6\frakc'}{\del}. 
\end{equation}
Recall that $k' \le j \le k$. We then have the following observations. 

\begin{lem} \label{1207lem01}
In \eqref{Maineq11}, when $k'$ is fixed and $c'$ is large enough, then the choice of $j$ is limited, and the number of such a $j$ depends onl on $\frakC$ and $\del$.
\end{lem}

\begin{proof}
Using the fact that $\del^k \le \del^j$, we have
\begin{eqnarray*}
\left(\frac{\frakc'\del}{2\frakC}-1 \right) \cdot \frac{\del^{k'}}{3}%
& \le& \left( \frakC^*+\frac{2\frakC}{\del}+\frac{2\frakC}{\del} \cdot \frac{6\frakc'}{\del} \right) \del^j \\
&=& \left( \frakC^*+\frac{2\frakC}{\del}+\frac{12\frakC\frakc'}{\del^2} \right) \del^j 
\end{eqnarray*}
and hence
$$
\del^{k'-j} \le \frac{\frakC^*+\frac{2\frakC}{\del}+\frac{12\frakC\frakc'}{\del^2}}{\left( \frac{\frakc'\del}{2\frakC}-1\right) \cdot \frac{1}{3}}.
$$
Note that the right hand side in the above inequality converges to $\frac{72\frakC^2}{\del}$ as $\frakc' \to \infty$. Now we require our $\frakc'$ is large enough, such that
\begin{equation} \label{Maineq12}
\frac{36\frakC^2}{\del} \le \frac{\frakC^*+\frac{2\frakC}{\del}+\frac{12\frakC\frakc'}{\del^2}}{\left( \frac{\frakc'\del}{2\frakC}-1\right) \cdot \frac{1}{3}} \le \frac{144\frakC^2}{\del}.
\end{equation}
This is the first condition we impose on the selction of $\frakc'$. Hence, we have
$$
\del^{k'-j} \le \frac{144\frakC^2}{\del},
$$
which implies the desired claim. 
\end{proof}

\begin{lem} \label{1207lem02}
In \eqref{Maineq11}, when $k'$ is fixed and $c'$ is large enough, then the choice of $j$ is limited, and the number of such a $k$ depends onl on $\frakC$ and $\del$.
\end{lem}

\begin{proof}
Again, we begin with \eqref{Maineq11}. Divide $\del^{k'}$ on both sides and use the relation $k' \le j \le k$ again, we have
\begin{eqnarray*}
\left( \frac{\frakc'}{2\frakC}-1 \right) \cdot \frac{1}{3}%
&\le&  \frakC^*\del^{j-k'}+ \frac{2\frakC}{\del} \cdot \del^{k-k'}+\frac{12\frakC\frakc'}{\del^2} \cdot \del^{k-k'} \\
&\le& \frakC^*+\left(\frac{2\frakC}{\del}+\frac{12\frakC\frakc'}{\del^2} \right) \del^{k-k'}.
\end{eqnarray*}
First, we take $c'$ large enough, such that
\begin{equation} \label{Maineq13}
\left( \frac{\frakc'}{2\frakC}-1 \right) \cdot \frac{1}{3} \ge 100\frakC^*.
\end{equation}
Second, note that we have 
$$
\del^{k-k'} \ge \frac{\left( \frac{\frakc'}{2\frakC}-1 \right) \cdot \frac{1}{3}-\frakC^*}{\frac{2\frakC}{\del}+\frac{12\frakC\frakc'}{\del^2}}.
$$
Again, we let $c'$ converges to $\infty$ and find the limit of the right hand side in the above inequality is $\frac{\del}{72\frakC^2}$. Hence, we can make a choice of a $\frakc'$, such that 
\begin{equation} \label{1124eq09}
\frac{\del}{144\frakC^2} \le  \frac{\left( \frac{\frakc'}{2\frakC}-1 \right) \cdot \frac{1}{3}-\frakC^*}{\frac{2\frakC}{\del}+\frac{12\frakC\frakc'}{\del^2}}\le \frac{\del}{36\frakC^2}.
\end{equation}
Thus,
\begin{equation} \label{1207eq02}
\del^{k-k'} \ge \frac{\del}{144\frakC^2}, 
\end{equation}
for such a choice of $\frakc'$, which implies the desired result.
\end{proof}

Finally, we take a $\frakc'$ such that the conclusion in Lemma \ref{1207lem01} and Lemma \ref{1207lem02} hold, more precisely, we take a $\frakc'$ such that \eqref{Maineq12}, \eqref{Maineq13} and \eqref{1124eq09} hold. 

\medskip

\underline{Step II: Complete the estimation of $J_1$.}

\medskip

From the above argument, it is easy to see that there exists some $D_1>1$, which only depends on $\frakC$, $\del$, $a$ and how $\gamma$ is controled by the list $(W, d)$ (namely, $(\calQ_1)$), such that $Q \subset D_1 Q'$, with $Q, Q'$ satisfying \eqref{1124eq10}. Moreover, we have the trivial upper bound 
\begin{equation} \label{1125eq06}
D_1 \le \frac{6\frakC_1}{\del}. 
\end{equation} 
Indeed, this is an easy consequence of triangle inequality, \eqref{1124eq01}, the fact that $k' \le j \le k$ and the fact that $\frakC>1$ and $0<\del<1$. Furthermore, we have the following observation.

\begin{lem}
Let $Q$ and $Q'$ be two dyadic cubes satisfy the assumption in Step I. If we fix the choice of $Q'$, then the choice of $Q$ is finite, which is bounded by some number only depends on $\frakC$ and $\del$.
\end{lem}

\begin{proof}
 Indeed, by Step I, we know when $Q'$ is fixed, that is, when $k'$ is fixed, the choice of $k$ is limited, say $k'+N$ is the largest choice of $k$, where $N$ is some absolute number only depends on $\frakC$ and $\del$, in particular, by \eqref{1207eq02}, we have
\begin{equation} \label{1125eq01}
\del^{-N} \le \frac{144\frakC^2}{\del}. 
\end{equation}

By \eqref{1125eq06}, \eqref{1125eq01} and the fact that $Q \subset D_1 Q'$, we have
\begin{eqnarray}  \label{1125eq10}
\textrm{Vol} (B_{(W, d)}(x_c(Q), \del^k)%
&\le&  \textrm{Vol} (Q) \le \textrm{Vol}(D_1Q') \nonumber\\
&=& \textrm{Vol}\left(B_{(W, d)}\left(x_c(Q'), \frac{D_1 \frakC}{\del} \cdot \del^{k'} \right) \right)  \nonumber\\
&\le& \textrm{Vol} \left( B_{(W, d)} \left(x_c(Q),  \frac{2D_1 \frakC}{\del} \cdot \del^{k'} \right) \right)  \nonumber \\
&=& \textrm{Vol} \left( B_{(W, d)} \left(x_c(Q),  \frac{2D_1 \frakC}{\del} \cdot \del^{k'-k} \cdot \del^k \right) \right)  \nonumber\\
&\le& \textrm{Vol} \left( B_{(W, d)} \left(x_c(Q),  \frac{2D_1 \frakC}{\del} \cdot \del^{-N} \cdot \del^k \right) \right)  \nonumber\\
&\le& \textrm{Vol} \left( B_{(W, d)} \left(x_c(Q),  \frac{288D_1 \frakC^3}{\del^2} \cdot \del^k \right) \right)  \nonumber\\
&\le&  \textrm{Vol} \left( B_{(W, d)} \left(x_c(Q),  \frac{2000 \frakC_1 \frakC^3}{\del^3} \cdot \del^k \right) \right) \\
&\lesssim& \textrm{Vol} (B_{(W, d)}(x_c(Q), \del^k) \le \textrm{Vol}(Q) \nonumber,
\end{eqnarray} 
where in the last inequality, we use \eqref{1109eq01}, moreover, the implicit constant in the last inequality only depends on $\frakC_1, \frakC, \del$ and the doubling constant $C_{V'}$ for $\rho'$ (see, Figure 4).

\begin{figure}[ht]
\begin{tikzpicture}[scale=5]
\draw (0,0) circle (.65);
\draw (.65, -0.5)  node [left]{\footnotesize{$D_1Q'$}};  
\draw (0,0) circle (.25);
\fill (0, 0) circle [radius=.25pt];
\fill (-0.32, -0.32) circle [radius=.25pt];
\draw (0, 0) node[below] {\footnotesize{$x_c(Q')$}}; 
\draw (.24, -0.23)  node [left]{\footnotesize{$Q'$}};  
\draw (-0.32, -0.32) circle (.13);
\draw (-0.32, -0.32) node[below] {\footnotesize{$x_c(Q)$}}; 
\draw (-0.15, -0.45)  node [left]{\footnotesize{$Q$}};  
\end{tikzpicture}
\caption{}
\end{figure}

As a consequence, we have
\begin{equation} \label{1125eq13}
\textrm{the total choices of}  \ Q \lesssim 1,
\end{equation}
where in the above estimation, we use the fact that Whiteney cubes are disjoint and the implict constant can be taken to be the same as the one in \eqref{1125eq10}. 
\end{proof}

Next, by using the $L^p$ improving property, we have the following result, which can be understood as a lemma of change of variable. 

\begin{lem} \label{1222lem01}
 Let $Q, Q'$ be two Whitney cubes such that
$$
\del^k=\ell(Q) \le \del^j \le \ell(Q')=\del^{k'},
$$
where $k \ge j \ge k' \ge 0$ are some integers as before. Let further
$$
\textrm{supp} (f) \subseteq Q.
$$
Then 
\begin{equation} \label{1130eq03}
\|\calT_j^{(j)} (f)\|_{L^s(Q')} \lesssim \textrm{Vol}(Q')^{\frac{1}{s}-\frac{1}{r}} \|f\|_{L^r(Q)},
\end{equation}
where the implict constant only depends on $\frakC$, $\del$, on the implicit constant in the inequality \eqref{190321eq01} and on any $2$-admissible constants. 
\end{lem}

\begin{proof}
Let $\Phi$ be the scaling map defined in Theorem \ref{scalingmap} with respect to the center $x_c(Q')$ and the list
\begin{equation} \label{1130eq02}
\left( \left( \frac{8\frakC_1}{\del \widetilde{\zeta}} \cdot \del^{k'} \right)^d W, d \right),
\end{equation} 
which by \eqref{1116eq02}, is the list
$$
\left( \left( \frac{8\frakC_1}{\del} \cdot \del^{k'+\widetilde{N}} \right)^d X, d \right).
$$
This, by \eqref{1127eq01}, implies 
\begin{eqnarray*}
\Phi: B_{(Y, d)}(0, \widetilde{\zeta_0}) \longrightarrow %
&& B_{\left( \left( \frac{8\frakC_1}{\del} \cdot \del^{k'+\widetilde{N}} \right)^d X, d \right)} \left(x_c(Q'), \widetilde{\zeta_0} \right) \\
&&= B_{ \left( \left( \frac{8\frakC_1}{\del \widetilde{\zeta}} \cdot \del^{k'} \right)^d W, d \right) } \left(x_c(Q'), \widetilde{\zeta_0} \right) \\
&&=B_{(W, d)} \left(x_c(Q'), \frac{8\frakC_1}{\del} \cdot \del^{k'} \cdot \frac{\widetilde{\zeta_0}}{\widetilde{\zeta}} \right) \subset K'
\end{eqnarray*}
is a $C^\infty$ diffeomorphism. Therefore, for any $x \in Q' \subset B_{(W, d)} \left(x_c(Q'), \frac{\frakC}{\del} \cdot \del^{k'} \right)$, we can find a $u \in B_{(Y, d)} \left(0, \frac{\frakC \widetilde{\zeta}}{8C_1} \right)$, such that $x=\Phi(u)$, and therefore we can write
\begin{eqnarray} \label{1129eq03}
&&\calT_j^{(j)}(f)(x)=  \calT_j^{(j)}(f) (\Phi(u)) \nonumber \\
&&= \psi_1 (\phi (u))\int \left(f \circ \Phi \right) \circ \left( \Phi^{-1} \circ \gamma_{\del^j t} \circ \Phi \right)(u)\left(\psi_2 \circ \Phi \right) \circ \left( \Phi^{-1} \circ \gamma_{\del^j t} \circ \Phi \right)(u) \nonumber \\
&&  \quad \quad \quad \quad \quad  \quad \quad \quad \quad \quad  \rho(\del^j t, \Phi(u)) \chi_j(t)dt \\
&&= \widehat{\psi}_1(u) \int \widehat{f} \left(\widehat{\gamma_{\del^j t}}(u) \right)  \widehat{\psi_2} \left( \widehat{\gamma_{\del^j t}}(u) \right) \widehat{\rho}(t, u) \chi_j(t)dt  \nonumber
\end{eqnarray}
where 
$$
\widehat{\psi_1}:=\psi_1 \circ \Phi, \ \widehat{\psi_2}:=\psi_2 \circ \Phi,
$$
$$
\widehat{f}:=f \circ \Phi,
$$
and 
$$
\widehat{\rho}(t, u):=\rho(\del^j t, \Phi(u)). 
$$
Therefore, we have
\begin{eqnarray} \label{1130eq01}
&&\left\| \calT_j^{(j)}(f) \right\|_{L^s(Q')}^s= \int_{Q'} \left| \calT_j^{(j)}(f)(x) \right|^s dx \nonumber \\
&&= \int_{Q'} \left| \calT_j^{(j)} \left(f \one_{B_{(W, d)} \left( x_c(Q'), \frac{\frakC_1}{\del} \cdot \del^{k'} \right)} \right)(x) \right|^s dx \nonumber \\
&& \le  \int_{B_{(W, d)} \left(x_c(Q'), \frac{\frakC}{\del} \cdot \del^{k'} \right)}  \left| \calT_j^{(j)}\left(f  \one_{B_{(W, d)} \left( x_c(Q'), \frac{\frakC_1}{\del} \cdot \del^{k'} \right)} \right)(x) \right|^s dx \nonumber \\
&&= \int_{B_{(Y, d)} \left(0, \frac{\frakC \widetilde{\zeta}}{8C_1} \right)} \left| \calT_j^{(j)} \left(f  \one_{B_{(W, d)} \left( x_c(Q'), \frac{\frakC_1}{\del} \cdot \del^{k'} \right)}  \right)( \Phi (u)) \right|^s |\det d\Phi(u)| du \nonumber  \\
&& \quad \quad (\textrm{change variables with} \ x=\Phi(u) \ \textrm{with applying} \  \eqref{changevab01}.) \nonumber \\
&& \simeq \textrm{Vol}(Q)  \int_{B_{(Y, d)} \left(0, \frac{\frakC \widetilde{\zeta}}{8C_1} \right)} \left| \calT_j^{(j)}\left(f  \one_{B_{(W, d)} \left( x_c(Q'), \frac{\frakC_1}{\del} \cdot \del^{k'} \right)}  \right)( \Phi (u)) \right|^s  du \nonumber \\
&& \quad \quad (\textrm{by} \ \eqref{20180902} \ \textrm{and} \ \eqref{changevab}.) \nonumber \\ 
&&= \textrm{Vol}(Q) \int_{B_{(Y, d)} \left(0, \frac{\frakC \widetilde{\zeta}}{8\frakC_1} \right)}  \left| \widehat{\psi}_1(u) \right|^s \bigg| \int \widehat{f} \left(\widehat{\gamma_{\del^j t}}(u) \right) \widehat{\psi_2} \left( \widehat{\gamma_{\del^j t}}(u) \right)  \nonumber \\
&&   \quad \quad \quad \quad \quad  \quad \quad \quad  \quad \one_{B_{(Y, d)} \left(0, \frac{\widetilde{\zeta}}{8} \right)}  \left( \widehat{\gamma_{\del^j t}}(u) \right)  \widehat{\rho}(t, u) \chi_j(t)dt \bigg|^s du, 
\end{eqnarray}
where in the second equality above, we are using the fact that only those values of $f$ on $B_{(W, d)} \left(x_c(Q'), \frac{\frakC_1}{\del} \cdot \del^{k'} \right)$ makes a non-trivial contribution to the term
$$
\int_{Q'} \left| \calT_j^{(j)}(f)(x) \right|^s dx.
$$
This is a consequence of \eqref{1125eq05}. Moreover, in the last equality, we are using \eqref{1129eq03}.

Recall the bump function $h_1$ is smooth satisfying $0 \le h_1 \le 1$ and $h_1 \equiv 1$ on $B^n(\widetilde{\zeta_0''})$, compactly supported in $B^n(\widetilde{\zeta_0'})$, where\footnote{Although here we are using the list $\left( \left( \frac{8\frakC_1}{\del \widetilde{\zeta}} \cdot \del^{k'} \right)^d W, d \right)$, instead of $\left( \left( \frac{8\frakC_1}{\del \widetilde{\zeta}} \right)^d W, d \right)$, which is used to defined the quantitave $L^p$ improving property, it is clear that they are sharing the same admissible constants, which implies we can take the same $\widetilde{\zeta_0'}$ and $\widetilde{\zeta_0''}$ as those we picked early.}
\begin{eqnarray*}
B_{(Y, d)} \left(0, \frac{\frakC \widetilde{\zeta}}{8\frakC_1} \right) %
&\subset&  B_{(Y, d)} \left(0, \frac{\widetilde{\zeta}}{8} \right) \subset B_{(Y, d)} \left(0, \frac{ \widetilde{\zeta_0}}{8} \right)  \\
&\subset& B^n(\widetilde{\zeta_0''}) \subset B^n(\widetilde{\zeta_0'}) \subset B_{(Y, d)} \left(0, \frac{\widetilde{\zeta_0}}{2} \right). 
\end{eqnarray*}

Using the bump function $h_1$, we have
\begin{eqnarray} \label{1221eq01}
\eqref{1130eq01}%
&\le& \textrm{Vol}(Q) \int_{\R^n} \left| \left(\widehat{\psi_1} h_1 \right)(u) \right|^s \bigg|  \int \left(\widehat{f}  \one_{B_{(Y, d)} \left(0, \frac{\widetilde{\zeta}}{8} \right)} \right)  \left( \widehat{\gamma_{\del^j t}}(u) \right)  \nonumber \\
&& \quad \quad \quad \quad \quad \quad \quad  \left(\widehat{\psi_2}  h_1 \right)  \left( \widehat{\gamma_{\del^j t}}(u) \right) \widehat{\rho}(t, u) \chi_j(t)dt \bigg|^s du  \nonumber  \\
&\le& \textrm{Vol}(Q) \left( \int_{\R^n} \left|\widehat{f} \one_{B_{(Y, d)} \left(0, \frac{\widetilde{\zeta}}{8} \right)}(u) \right|^r du \right)^{\frac{s}{r}}  \nonumber  \\
&=& \textrm{Vol}(Q) \left( \int_{B_{(Y, d)} \left(0, \frac{\widetilde{\zeta}}{8} \right)} \left|\widehat{f}(u)\right|^r du \right)^{\frac{s}{r}}  \nonumber  \\
&=& \textrm{Vol}(Q)^{1-\frac{s}{r}} \left( \int_{B_{(Y, d)} \left(0, \frac{\widetilde{\zeta}}{8} \right)} \left|f(\Phi(u))\right|^r |\det d\Phi(u)| du \right)^{\frac{s}{r}} \nonumber   \\
&& \quad \quad (\textrm{by} \ \eqref{changevab}.) \nonumber \\
&=& \textrm{Vol}(Q)^{1-\frac{s}{r}} \left( \int_{B_{(W, d)} \left( x_c(Q'), \frac{\frakC_1}{\del} \cdot \del^{k'} \right)} |f(x)|^rdx \right)^{\frac{s}{r}} \\
&\le&  \textrm{Vol}(Q)^{1-\frac{s}{r}} \left( \int_{\R^n} |f(x)|^rdx \right)^{\frac{s}{r}}=\textrm{Vol}(Q)^{1-\frac{s}{r}} \|f\|_{L^r(Q)}^s,  \nonumber 
\end{eqnarray}
where, in the second inequality above, we are using the quantitative $L^p$ improving property. To do this, we need to verify
\begin{enumerate}
\item [(i).] $\chi_j \in \calB_1$;
\item [(ii).] $\widehat{\psi_1}h_1, \widehat{\psi_2}h_2 \in \calB_2$;
\item [(iii).] $\widehat{\rho} \in \calB_3$;
\item [(iv).] $\widehat{\gamma_{\del^j t}} \in \calB_4$. 
\end{enumerate}
The first one is obvious, and the last one follows from the fact that the smooth mapping $\gamma_{\del^j t}$ is controlled at the unit scale by the list \eqref{1130eq02}, with the same paremeters in $(\calQ_1)$ (or equivalently, $(\calQ_2)$) as those for the list 
\eqref{1116eq05} (since $k' \le j$). While for the second and the third one, we can easily see these by chain rule and the fact that 
$$
\Phi(u):=e^{u \cdot \left( \left( \frac{8\frakC_1}{\del \widetilde{\zeta}} \cdot \del^{k'} \right)^d W \right)_{J_0}} \left(x_c(Q')\right), 
$$
where 
$$
\left( \left( \frac{8\frakC_1}{\del \widetilde{\zeta}} \cdot \del^{k'} \right)^d W \right)_{J_0}:=\left\{ \left( \frac{8\frakC_1}{\del \widetilde{\zeta}} \cdot \del^{k'} \right)^{d_{j_1}} W_{j_1}, \dots ,\left( \frac{8\frakC_1}{\del \widetilde{\zeta}} \cdot \del^{k'} \right)^{d_{j_n}} W_{j_n} \right\} 
$$
such that 
$$
\left| \det \left( \left( \frac{8\frakC_1}{\del \widetilde{\zeta}} \cdot \del^{k'} \right)^d W \right)_{J_0}(x_c(Q')) \right|=\left| \det_{n \times n} \left( \left( \frac{8\frakC_1}{\del \widetilde{\zeta}} \cdot \del^{k'} \right)^d W \right)(x_c(Q')) \right|_\infty. 
$$
The proof of the claim is complete. 
\end{proof}

Therefore, for a pair of Whitney cubes $Q$ and $Q'$ with satisfying $\ell(Q) \le \del^j \le \ell(Q')$ (namely, the pair of cubes$(Q, Q')$ appears in a single term of $J_1$, see \eqref{1124eq10}), we have
\begin{eqnarray*}
&&|\langle \calT_j^{(j)}(b_{1, Q}), b_{2, Q'} \rangle| \le \|\calT_j^{(j)}(b_{1, Q}) \|_{L^s(Q')} \|b_{2, Q'}\|_{L^{s'}(Q')}\\
&&\lesssim  \textrm{Vol}(Q')^{\frac{1}{s}-\frac{1}{r}} \|b_{1, Q}\|_{L^r (Q')} \|b_{2, Q}\|_{L^{s'}(Q')}\\
&& \quad (\textrm{by \eqref{1130eq03}.}) \\
&& \lesssim \textrm{Vol}(Q') \langle f_1 \rangle_{D_1Q', r}\langle f_2 \rangle_{Q', s'}.\\
&& \quad (\textrm{by the fact that} \ Q \subseteq D_1Q' \ \textrm{and Lemma \ref{modyeq02}}.)
\end{eqnarray*}

Recall from above that when $Q'$ is fixed, in particular, $k'$ is fixed, we have the choice of both $j$ and $Q$ are finite, which is known bounded by some absolute constant depending on $\frakC$ and $\del$. Combining this fact with the above estimate, we have
\begin{eqnarray*}
J_1%
&=&   \sum_{Q, Q' \in M_E, Q, Q' \subset Q_0} \sum_{j \ge 0, \ell(Q) \le \del^j \le \ell(Q')}  \left|\left\langle \calT_j^{(j)}(b_{1, Q}), b_{2, Q'} \right\rangle \right| \\
&\lesssim& \sum_{Q' \in M_E, Q' \subset Q_0} \textrm{Vol}(Q') \langle f_1 \rangle_{D_1Q', r} \langle f_2 \rangle_{Q', s'} \\
&\lesssim& \left(\sum_{Q' \in M_E, Q' \subset Q_0} \textrm{Vol}(Q') \right) \cdot \langle f_1 \rangle_{Q_0, r}  \langle f_2 \rangle_{\kappa'Q_0, s'} \\
&& \quad (\textrm{by Lemma \ref{Maineq02}}.) \\
&\le& \textrm{Vol}(Q_0) \cdot \langle f_1 \rangle_{Q_0, r}  \langle f_2 \rangle_{\kappa' Q_0, s'}. \\
&& \quad (\textrm{by the disjointness of Whitney cube.}) 
\end{eqnarray*}

The estimate for $J_1$ is complete. 

\medskip 

\section{Proof of the main result: Part II}

In this section, we estimate the term $J_2$. Recall that 
$$
J_2=\left| \sum_{j \ge 0} \sum_{k \ge j} \sum_{k' \ge j} \langle \calT_j^{(j)} b_{1, k}, b_{2, k'}\rangle \right|=\left| \sum_{j \ge 0} \sum_{k_1 \ge 0} \sum_{k_2 \ge 0} \langle \calT_j^{(j)} b_{1, k_1+j}, b_{2, k_2+j}\rangle \right|，
$$
where
$$
b_{1, k_1+j}=\sum_{Q \in M_E, Q \subset Q_0, \ell(Q)=\del^{k_1+j}} b_{1, Q}
$$
and 
$$
b_{2, k_2+j}=\sum_{Q' \in M_E, Q \subset Q_0, \ell(Q)=\del^{k_2+j}} b_{2, Q'}.
$$
We start with writing $J_2$ into two parts as
\begin{eqnarray*}
J_2%
&\le& \left| \sum_{j \ge 0} \sum_{k_2 \ge k_1 \ge 0} \langle \calT_j^{(j)}b_{1, k_1+j}, b_{2, k_2+j} \rangle \right|+ \left| \sum_{j \ge 0} \sum_{k_1 \ge k_2 \ge 0} \langle \calT_j^{(j)}b_{1, k_1+j}, b_{2, k_2+j} \rangle \right| \\
&=& J_{2, 1}+J_{2, 2},
\end{eqnarray*}
where
$$
J_{2, 1}:= \left| \sum_{j \ge 0} \sum_{k_2 \ge k_1 \ge 0} \langle \calT_j^{(j)}b_{1, k_1+j}, b_{2, k_2+j} \rangle \right|
$$
and
$$
J_{2, 2}:=\left| \sum_{j \ge 0} \sum_{k_1 \ge k_2 \ge 0} \langle \calT_j^{(j)}b_{1, k_1+j}, b_{2, k_2+j} \rangle \right| .
$$

It suffices for us to estimate the term $J_{2, 1}$, since the estimation of $J_{2, 2}$ is similar by replacing term $\langle \calT_j^{(j)} b_{1, k_1+j}, b_{2, k_2+j} \rangle$ by its conjugate $\langle b_{1, k_1+j}, \left(\calT_j^{(j)}\right)^*b_{2, k_2+j} \rangle$.

Note that for the term $J_{2, 1}$, it is different from $J_1$, in the sense that it is no longer true that when $k_1$ ($k_2$, respectively) is fixed, then the total choices of $j$ and $k_2$ ($k_1$, respectively) is bounded by a fixed number, which is independent of $j, k_1$ and $k_2$.

\medskip 

\subsection{Local transition of a dyadic cube}

The idea to estimate $J_{2, 1}$ is to make use of the fact that
\begin{equation} \label{0102eq001}
\int_{Q'} b_{2, Q'}(x)dx=0. 
\end{equation}
Recall from the classical Calder\'on-Zygmund theory that we can write \eqref{0102eq001} as 
$$
\frac{1}{|Q'|} \int_{Q'} \int_{Q'} \left( b_{2, Q'}(x)- b_{2, Q'}(x') \right) dxdx'=0. 
$$
To proceed, we define the following concept of ``local transition of a dyadic cube", which is an analog of the Euclidean case. The setting is as follows. We take
\begin{enumerate}
\item [(1).] $j \ge 0, j \in \N$;
\item [(2).] $k \ge K_1, k  \in \N$, where $K_1 \in \N$ is the smallest constant such that $\del^{K_1}<\frac{1}{2}$; 
\item [(3).] $Q$ be a dyadic cube with $\ell(Q)=\del^j$;
\item [(4).] $Q' \subset Q$ be a dyadic cube with $\ell(Q')=\del^{j+k}$;
\item [(5).] $x \in Q'$. 
\end{enumerate}
Let $J_1=J_1(x, j+k) \in \calJ(n, q)$ be the list such that
$$
\left| \det \left( \left( \frac{16\frakC_1}{\del \widetilde{\zeta}} \cdot \del^{j+k} \right)^d W \right)_{J_1} (x) \right|= \left| \det_{n \times n} \left( \left( \frac{16\frakC_1}{\del \widetilde{\zeta}} \cdot \del^{j+k} \right)^d W \right) (x) \right|_\infty.
$$
Note that since $\del^k \le \del^{K_1}<\frac{1}{2}$, the list
\begin{equation} \label{190201eq01}
 \left( \left( \frac{16\frakC_1}{\del \widetilde{\zeta}} \cdot \del^{j+k} \right)^d W, d \right)
\end{equation}
shares the same defining constants with the list \eqref{1116eq05}, and so are all admissible constants. 

By Theorem \ref{scalingmap}, we have $\Phi_{x, j+k}: B^n(\eta_1) \longrightarrow  B_{ \left(\left( \frac{16\frakC_1}{\del \widetilde{\zeta}} \cdot \del^{j+k} \right)^d W, d \right)}(x, 1) $ is a scaling map given by the formula
$$
\Phi_{x, j+k} (u)= e^{u \cdot \left(\left( \frac{16\frakC_1}{\del \widetilde{\zeta}} \cdot \del^{j+k} \right)^d W \right)_{J_1}}(x), \  u \in B^n(\eta_1).
$$
Here $\eta_1>0$ is the some $2$-admissible constant defined in Theorem \ref{scalingmap}. Following the argument in \eqref{1127eq01} with $\widetilde{\zeta_0}$ replaced by $\widetilde{\zeta}$, we see that
\begin{eqnarray}\label{190201eq02}
\Phi_{x, j+k}: B_{(Y, d)}(0, \widetilde{\zeta})%
&\longrightarrow& B_{ \left(\left( \frac{16\frakC_1}{\del \widetilde{\zeta}} \cdot \del^{j+k} \right)^d W, d \right)}(x, \widetilde{\zeta}) \\
&&= B_{ \left(\left( \frac{16\frakC_1 }{\del} \cdot \del^{j+k} \right)^d W, d \right)}(x, 1) \nonumber \\
&&= B_{\left( \left( \frac{8\frakC_1 }{\del} \right)^d W, d \right)} \left(x, 2\del^{j+k} \right) \nonumber
\end{eqnarray}
is a $C^\infty$ diffeomorphism, where the list $(Y, d)$ is the pullback of the list  \eqref{190201eq01} via the scaling map $\Phi_{x, j+k}$. Using \eqref{190102eq022} and \eqref{190201eq01}, we see \eqref{190201eq02} implies 
$$
\Phi_{x, j+k}: B^n(\widetilde{\eta_1}) \longrightarrow B_{\left( \left( \frac{8\frakC_1 }{\del} \right)^d W, d \right)} \left(x, 2\del^{j+k} \right)
$$
is surjective, where $\widetilde{\eta_1}$ is the admissible constant defined in \eqref{190102eq022}. In particular, this implies
\begin{equation} \label{190202eq010}
\Phi_{x, j+k}: B^n(\widetilde{\eta_1}) \longrightarrow B_{(W, d)} \left(x, 2\frakC_1\del^{j+k} \right)
\end{equation}
is surjective. 

\begin{defn}
We call the scaling map $\Phi_{x, j+k}$ defined above the \emph{local transition associated to $Q'$}. 
\end{defn}

Note that by Lemma \ref{lem190201} and \eqref{190201eq01}, we have:
\begin{equation} \label{190201eq03}
\int_{B^n(\widetilde{\eta_1})} f \circ \Phi_{x, j+k}(u)du \lesssim_2 \int_{B^q(\eta_3)} f \left (e^{t \cdot \left( \frac{16\frakC_1}{\del \widetilde{\zeta}} \cdot \del^{j+k} \right)^d W} (x) \right) dt.
\end{equation}
This allows us to pass our estimation from a ``localized" one to a ``global'' one, more precisely, the vector fields involved in the right hand side of \eqref{190201eq03} depends on the choice of $x$, while the left hand side does not. It turns out this observation plays an important role in the sequel. 

Finally, motived by \eqref{190201eq03}, we construct the function ``$\theta$" involved in the statement in Lemma \ref{moduluscont} at scale $\del^j$.

\begin{defn} \label{190202defn01}
For each $j \ge 0$ and $k \ge K_1$, we define $\theta_j: B^q(\vardbtilde{a}) \times K''' \rightarrow K'$ by 
$$
\theta_j(t, x):=e^{\sum\limits_{i=1}^q t_i \cdot \left(\frac{16\frakC_1}{\del \widetilde{\zeta}} \del^j \right)^{d_i} W_i} (x),
$$
where $\vardbtilde{a}>0$ is sufficient small and the upper bound of $\vardbtilde{a}$ only depends on $\frakC_1, \del, \widetilde{\zeta}$ and on any adimissible constants.  
\end{defn}

In our later application, we would like to restrict the behavior of $\theta_j$ on each dyadic cube $Q$ with $\ell(Q)=\del^j$. We summarize these properties below. 

\begin{lem} \label{0102lem02}
Let $\theta_j$ and $Q$ be defined as above. Then 
\begin{enumerate}
\item [(a).] $\theta_j(0, x) \equiv x, \forall x \in B_{(W, d)}\left(x_c(Q), \frac{5\frakC_1}{\del} \cdot \del^j \right)$;
\item [(b).] There exists some $0<\vardbtilde{a}_1 \le \vardbtilde{a}$, independent of $j$, such that 
$$
\theta_j \in C^\infty \left( B^q \left(\vardbtilde{a}_1 \right) \times B_{(W, d)}\left(x_c(Q), \frac{5\frakC_1}{\del} \cdot \del^j \right) \right);
$$
and
$$
\theta_j: B^q(\vardbtilde{a}_1) \times B_{(W, d)}\left(x_c(Q), \frac{5\frakC_1}{\del} \cdot \del^j \right)  \longrightarrow B_{(W, d)}\left(x_c(Q), \frac{6\frakC_1}{\del} \cdot \del^j \right);
$$
\item [(c).] There exists a $0< \vardbtilde{a}_2 \le \vardbtilde{a}_1$, such that for any $b \in B^q(\vardbtilde{a}_2)$, the map $\theta_b(\cdot):=\theta(b, \cdot)$ has an inverse, which maps $\theta_b \left( B_{(W, d)}\left(x_c(Q), \frac{5\frakC_1}{\del} \cdot \del^j \right)  \right)$ back to $B_{(W, d)}\left(x_c(Q), \frac{5\frakC_1}{\del} \cdot \del^j \right) $;
\item [(d).] $\theta_j$ is controlled by the list $ \left(\left( \frac{8\frakC_1}{\del \widetilde{\zeta}} \cdot \del^j \right)^d W, d \right)$. 
\end{enumerate}
\end{lem}

\begin{proof}
The first assertion is clear. The second one follows from the condtion $(\calQ_2)$  and the third one follows from the definition of the exponential map. Finally, the last assertion follows from the condition $(\calQ_1)$. 
\end{proof}

\begin{rem}
Note that all the assertions in Lemma \ref{0102lem02} corresponds exactly those assumptions on the mapping ``$\theta$'' in Lemma \ref{moduluscont}. Moreover, the choice of ``$a$" in the statement of Theorem \ref{MainTheorem2} will be also smaller than $ \vardbtilde{a}$.
\end{rem}

\medskip 

\subsection{Estimation of $J_2$} We begin with the following result, which is an easy consequence of Lemma \ref{1222lem01}. 

\begin{lem} \label{0103lem100}
Let $Q$ be a Whitney dyadic cube with $\ell(Q)=\del^j, j \ge 0$. Then for any measurable function $f$, we have
\begin{enumerate}
\item [(1).]
$$
\|\calT_j^{(j)}(f)\|_{L^s(Q)} \lesssim \textrm{Vol}(Q)^{\frac{1}{s}-\frac{1}{r}} \|f\|_{L^r\left(\frac{\frakC_1Q}{\del} \right)},
$$
\item [(2).] there exists a $2$-admissible constant $\widetilde{\eta_3}>0$, such that for $b \in \R^q$ with $|b| \le \widetilde{\eta_3}$, 
$$
\|\calT_j^{(j)}(f(\theta_j (b, \cdot))\|_{L^s(Q)} \lesssim \textrm{Vol}(Q)^{\frac{1}{s}-\frac{1}{r}} \|f\|_{L^r\left(\frac{\frakC_1Q}{\del} \right)}.
$$
\end{enumerate} 
Finally, there exists $\eta, \widetilde{\eta_4}>0$, such that when $|b|<\widetilde{\eta_4}$, then 
\begin{equation} \label{0102eq100}
\|(\calT_j^{(j)} f)(\cdot)-(\calT_j^{(j)}f)(\theta_j(b, \cdot))\|_{L^s(Q)} \lesssim |b|^{\eta}\textrm{Vol}(Q)^{\frac{1}{s}-\frac{1}{r}} \|f\|_{L^r\left(\frac{\frakC_1Q}{\del} \right)}.
\end{equation}
Here the implict constants in the above inequalities only depends on $\frakC$, $\del$, on the implicit constant in the inequality \eqref{190321eq01} and on any $2$-admissible constants.
\end{lem}

\begin{proof}
The first inequality follows from the proof of Lemma \ref{1222lem01} with $k'=j$ and \eqref{1221eq01}. While for the second one, we can pick $\widetilde{\eta_3}>0$ small enough, such that 
$$
|\det d\theta_j(b, x)| \simeq 1, \ \forall x \in B_{(W, d)}\left(x_c(Q), \frac{5\frakC_1}{\del} \cdot \del^j \right),
$$ 
where we note that $\widetilde{\eta_3}$ can be choosen independently of $j$ and $Q$. The second inequality then follows from the first one and changing of variable. 

Finally, the inequality \eqref{0102eq100} follows from the previous two inequalities, Lemma \ref{moduluscont} and interpolation. 
\end{proof}

Now we turn back to the estimate of $J_2$ (more precisely, $J_{2, 1}$). Recall that
$$
b_{2, k_2}=\sum_{Q' \in M_E, Q' \subset Q_0, \ell(Q')=\del^{j+k_2}} b_{2, Q'}. 
$$
First, we take $K_V \in \N$, sufficient large, such that
\begin{enumerate}
\item [(1).] $K_V \ge K_1$, where $K_1$ is the integer defined in Definition \ref{190202defn01}; 
\item [(2).] $\del^{K_V}<\frac{\widetilde{\eta_4}}{100q}$, where $\widetilde{\eta_4}$ is the admissible constant defined in Lemma \ref{0103lem100}.
\end{enumerate}

Next, we group all the Whitney cubes $Q'$ with $\ell(Q')=\del^{j+k_2}$ with respect to their dyadic parents, namely, for each Whitney cube $Q'$, there exists a unique dyadic parent $Q$ containing $Q'$ with $\ell(Q)=\del^j$. Therefore, we can write
$$
b_{2, k_2}=\sum_{Q \in \calG_j, Q \cap \textrm{supp}(b_{2, j+k_2}) \neq \emptyset} \sum_{Q' \in M_E, Q' \subset Q, \ell(Q')=\del^{j+k_2}} b_{2, Q'}.
$$

We first deal with the case when $k_2> K_V$.  For each $Q'$, we have
\begin{eqnarray*}
&&\left| \langle \calT_j^{(j)}b_{1, j+k_1}, b_{2, Q'} \rangle \right|=\left |\int_{Q'} \calT_j^{(j)} \left(b_{1, j+k_1} \right) (x) b_{2, Q'}(x) dx \right |  \\
&& \le \frac{1}{|Q'|}  \int_{Q'} \int_{Q'} \left| \calT_j^{(j)}  \left(b_{1, j+k_1}  \right)(x)-\calT_j^{(j)}  \left(b_{1, j+k_1} \right)(x') b_{2, Q'}(x) \right| dxdx' \\
&& \lesssim_2   \int_{Q'} \int_{B^n(\widetilde{\eta_1})} \bigg|  \calT_j^{(j)}  \left(b_{1, j+k_1}  \right)(x)-\calT_j^{(j)}(b_{1, j+k_1})(\Phi_{x, j+k_2}(u)) \bigg| |b_{2, Q'}(x)| dudx \\
&&  \quad \quad  (\textrm{by} \ \eqref{190202eq010} \  \textrm{and change of variables with} \ x'=\Phi_{x, j+k_2}(u) .) \\
&& \lesssim_2  \int_{Q'} \int_{B^q(\eta_3)} \bigg|  \calT_j^{(j)}  \left(b_{1, j+k_1}  \right)(x) \\
&& \quad \quad \quad \quad \quad \quad -\calT_j^{(j)}(b_{1, j+k_1})\left( e^{t \cdot \left( \frac{16\frakC_1}{\del \widetilde{\zeta}} \cdot \del^{j+k_2} \right)^d W}(x) \right) \bigg| |b_{2, Q'}(x)| dtdx. \\
&& \quad \quad  (\textrm{by} \ \eqref{190201eq03}.) 
\end{eqnarray*}

Hence, for $k_2> K_V$ and each dyadic cube\footnote{Note that this $Q$ need not to be a Whitney cube.} $Q$ with $\ell(Q)=\del^j$, we see that 
\begin{eqnarray*}
&&\left| \sum_{Q' \in M_E, Q' \subseteq Q, \ell(Q')=\del^{k_2+j}} \langle \calT_j^{(j)}b_{1, j+k_1}, b_{2, Q'} \rangle \right| \nonumber\\
&& \lesssim_2  \sum_{Q' \in M_E, Q' \subseteq Q, \ell(Q')=\del^{k_2+j}}\int_{Q'} \int_{B^q(\eta_3)} \bigg|  \calT_j^{(j)}  \left(b_{1, j+k_1}  \right)(x) \\
&& \quad \quad \quad \quad \quad \quad -\calT_j^{(j)}(b_{1, j+k_1})\left( e^{t \cdot \left( \frac{16\frakC_1}{\del \widetilde{\zeta}} \cdot \del^{j+k_2} \right)^d W}(x) \right) \bigg| |b_{2, Q'}(x)| dtdx \\
&& \le \int_{B^q(\eta_3)} \int_Q  \left|  \calT_j^{(j)}  \left(b_{1, j+k_1}  \right)(x) -\calT_j^{(j)}(b_{1, j+k_1})\left( e^{t \cdot \left( \frac{16\frakC_1}{\del \widetilde{\zeta}} \cdot \del^{j+k_2} \right)^d W}(x) \right) \right|  \\
&&  \quad \quad \quad \quad \quad \quad  \cdot \left|b_{2, k_2}(x)\right| dxdt \\
&& \le  \int_{B^q(\eta_3)} \left\| \calT_j^{(j)}  \left(b_{1, j+k_1}  \right)( \cdot) -\calT_j^{(j)}(b_{1, j+k_1})\left( e^{t \cdot \left( \frac{16\frakC_1}{\del \widetilde{\zeta}} \cdot \del^{j+k_2} \right)^d W}( \cdot) \right) \right\|_{L^s(Q)} \\
&&  \quad \quad \quad \quad \quad \quad  \cdot  \|b_{2, k_2}\|_{L^{s'}(Q)} dt \\ 
&& \lesssim  \sup_{t \in B^q(\eta_3)} \left\| \calT_j^{(j)}  \left(b_{1, j+k_1}  \right)( \cdot) -\calT_j^{(j)}(b_{1, j+k_1})\left( e^{t \cdot \left( \frac{16\frakC_1}{\del \widetilde{\zeta}} \cdot \del^{j+k_2} \right)^d W}( \cdot) \right) \right\|_{L^s(Q)}\\
&&  \quad \quad \quad \quad \quad \quad  \cdot  \|b_{2, k_2}\|_{L^{s'}(Q)}. 
\end{eqnarray*}
Now for each $t=(t_1, \dots, t_q) \in B^q(\eta_3)$, we wish to estimate the term 
$$
 \mathbb A:=\left\| \calT_j^{(j)}  \left(b_{1, j+k_1}  \right)( \cdot) -\calT_j^{(j)}(b_{1, j+k_1})\left( e^{t \cdot \left( \frac{16C_1}{\del \widetilde{\zeta}} \cdot \del^{j+k_2} \right)^d W}( \cdot) \right) \right\|_{L^s(Q)}.
$$
Write
$$
t \cdot \left( \frac{16C_1}{\del \widetilde{\zeta}} \cdot \del^{j+k_2} \right)^d W=\widetilde{t} \cdot \left( \frac{16C_1}{\del \widetilde{\zeta}} \cdot \del^j \right)^d W, 
$$
where $\widetilde{t}= \left( t_1\del^{k_2}, \dots,   t_q \del^{k_2}\right) \in \R^q$.  It is easy to see that $\left|\widetilde{t}\right|<\del^{k_2}<\widetilde{\eta_4}$.  Therefore, by \eqref{0102eq100}, we have
\begin{eqnarray*}
\mathbb A%
&=& \left\| \calT_j^{(j)}  \left(b_{1, j+k_1}  \right)( \cdot) -\calT_j^{(j)}(b_{1, j+k_1})\left( e^{\widetilde{t} \cdot \left( \frac{16\frakC_1}{\del \widetilde{\zeta}} \cdot \del^j \right)^d W}( \cdot) \right) \right\|_{L^s(Q)} \\
&=& \left\| \calT_j^{(j)} (b_{1, j+k_1}(\cdot)-\calT_j^{(j)} (b_{1, j+k_1}) \left(\theta_j( \widetilde{t}, \cdot)\right) \right\|_{L^s(Q)} \\
& \lesssim_2& \del^{k_2 \eta} \textrm{Vol}(Q)^{\frac{1}{s}-\frac{1}{r}} \|b_{1, j+k_1}\|_{L^r \left(\frac{\frakC_1 Q}{\del} \right)}, 
\end{eqnarray*}
which implies
\begin{eqnarray*}
\left| \sum_{Q' \in M_E, Q' \subseteq Q, \ell(Q')=\del^{k_2+j}} \langle \calT_j^{(j)}b_{1, j+k_1}, b_{2, Q'} \rangle \right|%
& \lesssim_2&  \del^{k_2 \eta} \textrm{Vol}(Q)^{\frac{1}{s}-\frac{1}{r}} \\
&& \quad \quad \quad  \cdot \|b_{1, j+k_1}\|_{L^r \left(\frac{\frakC_1 Q}{\del} \right)} \|b_{2, k_2}\|_{L^{s'}(Q)}.
\end{eqnarray*}

While for the case when $0 \le k_2 \le K_V$, we can bound it trivially, namely
\begin{eqnarray*}
\left| \sum_{Q' \in M_E, Q' \subseteq Q, \ell(Q')=\del^{k_2+j}} \langle \calT_j^{(j)}b_{1, j+k_1}, b_{2, Q'} \rangle \right|%
&=& \left| \langle \calT_j^{(j)} b_{1, j+k_1} b_{2, j+k_2} \one_Q \rangle \right| \\
&\le& \|\calT_j^{(j)}(b_{1, j+k_1}) \|_{L^s(Q)} \|b_{2, k_2}\|_{L^{s'}(Q)} \\
&\lesssim&  \textrm{Vol}(Q)^{\frac{1}{s}-\frac{1}{r}} \|b_{1, j+k_1}\|_{L^r \left(\frac{\frakC_1 Q}{\del} \right)} \|b_{2, k_2}\|_{L^{s'}(Q)},
\end{eqnarray*}
where in the last inequality, we use Lemma \ref{0103lem100}. 

Combining the above two estimates, we have
\begin{eqnarray*}
J_{2, 1}%
&=& \left| \sum_{j \ge 0} \sum_{k_2 \ge k_1 \ge 0} \langle \calT_j^{(j)} b_{1, k_1+j}, b_{2, k_2+j} \rangle \right| \\
& \le &  \left| \sum_{j \ge 0} \sum_{k_2 \ge k_1 \ge 0, 0 \le k_2 \le K_V} \langle \calT_j^{(j)} b_{1, k_1+j}, b_{2, k_2+j} \rangle \right|\\
&& \quad \quad \quad  \quad \quad \quad \quad \quad +\left| \sum_{j \ge 0} \sum_{k_2 \ge k_1 \ge 0,  k_2> K_V} \langle \calT_j^{(j)} b_{1, k_1+j}, b_{2, k_2+j} \rangle \right| \\
& \le &  \sum_{j \ge 0} \sum_{k_2 \ge k_1 \ge 0, 0 \le k_2 \le K_V} \sum_{Q \in \calG_j, Q \cap \textrm{supp}(b_{2, j+k_2}) \neq \emptyset}  \\
&& \quad \quad \quad \quad \quad  \quad \quad \quad \quad \quad  \left| \sum_{Q' \in M_E, Q' \subset Q, \ell(Q')=\del^{j+k_2}} \langle \calT_j^{(j)} b_{1, k_1+j}, b_{2, Q'} \rangle \right |\\
&& +\sum_{j \ge 0} \sum_{k_2 \ge k_1 \ge 0,  k_2> K_V} \sum_{Q \in \calG_j, Q \cap \textrm{supp}(b_{2, j+k_2}) \neq \emptyset}  \\
&& \quad \quad \quad \quad \quad  \quad \quad \quad \quad \quad  \left| \sum_{Q' \in M_E, Q' \subset Q, \ell(Q')=\del^{j+k_2}} \langle \calT_j^{(j)} b_{1, k_1+j}, b_{2, Q'} \rangle \right |\\
& \lesssim & \sum_{j \ge 0} \sum_{k_2 \ge k_1 \ge 0, 0 \le k_2 \le K_V} \sum_{Q \in \calG_j, Q \cap \textrm{supp}(b_{2, j+k_2}) \neq \emptyset}  \\
&& \quad \quad \quad \quad \quad  \quad \quad \quad \quad \quad   \textrm{Vol}(Q)^{\frac{1}{s}-\frac{1}{r}} \|b_{1, j+k_1}\|_{L^r \left(\frac{\frakC_1 Q}{\del} \right)} \|b_{2, k_2}\|_{L^{s'}(Q)} \\
&&  +\sum_{j \ge 0} \sum_{k_2 \ge k_1 \ge 0,  k_2> K_V} \sum_{Q \in \calG_j, Q \cap \textrm{supp}(b_{2, j+k_2}) \neq \emptyset}  \\
&& \quad \quad \quad \quad \quad  \quad \quad \quad \quad \quad  \del^{k_2 \eta} \textrm{Vol}(Q)^{\frac{1}{s}-\frac{1}{r}} \|b_{1, j+k_1}\|_{L^r \left(\frac{\frakC_1 Q}{\del} \right)} \|b_{2, k_2}\|_{L^{s'}(Q)}. 
\end{eqnarray*}

The proof will be complete if we can show for each $k_1, k_2 \ge 0$, 
\begin{eqnarray} \label{neweq04}
&& \sum_{j \ge 0}  \sum_{Q \in \calG_j, Q \cap \textrm{supp}(b_{2, j+k_2}) \neq \emptyset}   \textrm{Vol}(Q)^{\frac{1}{s}-\frac{1}{r} } \left\|b_{1, j+k_1} \right\|_{L^r \left(\frac{\frakC_1 Q}{\del} \right)}  \|b_{2, k_2+j}\|_{L^{s'}(Q)} \nonumber \\
&&  \quad \quad \quad \quad \quad  \quad \quad \quad \quad  \quad \quad \quad \lesssim \textrm{Vol}(Q_0) \langle f_1 \rangle_{Q_0, r}\langle f_2 \rangle_{\kappa' Q_0, s'} . 
\end{eqnarray}
We assume that at this moment \eqref{neweq04} holds. Then we have
\begin{eqnarray*}
J_{2, 1}%
&\lesssim&  \textrm{Vol}(Q_0) \langle f_1 \rangle_{Q_0, r}\langle f_2 \rangle_{\kappa' Q_0, s'} \left( \sum_{k_2 \ge k_1 \ge 0, k_2>K_V} \del^{k_2 \eta}+  \sum_{k_2 \ge k_1 \ge 0, 0 \le k_2 \le K_V} 1 \right) \\
&=&   \textrm{Vol}(Q_0)  \langle f_1 \rangle_{Q_0, r}\langle f_2 \rangle_{\kappa' Q_0, s'} \left(\sum_{k_2 \ge K_V} k_2 \del^{k_2 \eta}+\sum_{0 \le k_2 \le K_V} k_2 \right)  \\
&\lesssim &   \textrm{Vol}(Q_0)  \langle f_1 \rangle_{Q_0, r}\langle f_2 \rangle_{\kappa' Q_0, s'}.  
\end{eqnarray*}
We are left to proof \eqref{neweq04}, which follows from an interpolation argument. More precisely, we will show it holds true when $r=s'=1$ and $\frac{1}{r}+\frac{1}{s'}=1$, then apply the complex interpolation. In the sequel, we may assume $\langle f_1 \rangle_{Q_0, r}=\langle f_2 \rangle_{\kappa'Q_0, s'}=1$ by homogeneity. 

\medskip

\textit{Case I: $r=s'=1$.}

\medskip

In this case, we need to show that
$$
\sum_{j \ge 0}  \sum_{Q \in \calG_j, Q \cap \textrm{supp}(b_{2, j+k_2}) \neq \emptyset} \frac{1}{ \textrm{Vol}(Q) } \|b_{1, j+k_1} \|_{L^1\left(\frac{\frakC_1Q}{\del} \right)} \|b_{2, k_2+j} \|_{L^1(Q)} \lesssim |Q_0|.
$$
First we note that, for fixed $Q \in \calG_j, Q \cap \textrm{supp}(b_{2, j+k_2}) \neq \emptyset$ for some $j \ge 0$, we have
$$
\frac{1}{ \textrm{Vol}(Q) } \|b_{1, k_1+j} \|_{L^1\left(\frac{\frakC_1Q}{\del}\right)} \lesssim 1. 
$$
Indeed, we have
\begin{eqnarray*}
\frac{1}{\textrm{Vol}(Q) } \|b_{1, k_1+j} \|_{L^1\left(\frac{\frakC_1Q}{\del}\right) }%
&=& \frac{1}{\textrm{Vol}(Q) } \int_{\frac{\frakC_1Q}{\del}} \sum_{Q' \in M_E, \ell(Q')=\del^{k_1+j}, Q' \cap \frac{\frakC_1Q}{\del} \neq \emptyset} |b_{1, Q'}(x)| dx \\
&\lesssim& \frac{1}{\textrm{Vol} \left( \frac{2\frakC_1Q}{\del}  \right) }  \int_{\frac{2\frakC_1Q}{\del}} \sum_{Q' \in M_E, \ell(Q')=\del^{k_1+j}, Q' \cap \frac{\frakC_1Q}{\del} \neq \emptyset} |b_{1, Q'}(x)| dx \\
&\lesssim& 1,
\end{eqnarray*}
where in the last inequality, we use the fact that
$$
\bigcup_{Q' \in M_E, \ell(Q')=\del^{k_1+j}, Q' \cap \frac{\frakC_1Q}{\del} \neq \emptyset}  Q' \subseteq \frac{2\frakC_1Q}{\del}, 
$$
Lemma \ref{Maineq02}, Lemma \ref{modyeq02} and the assumption $\langle f_2 \rangle_{\kappa'Q_0, 1}=1$. Thus,
\begin{eqnarray*}
&&\sum_{j \ge 0}  \sum_{Q \in \calG_j, Q \cap \textrm{supp}(b_{2, j+k_2}) \neq \emptyset} \frac{1}{\textrm{Vol}(Q) } \|b_{1, j+k_1} \|_{L^1\left(\frac{\frakC_1Q}{\del} \right)} \|b_{2, k_2+j} \|_{L^1(Q)}\\ 
&&\lesssim\sum_{j \ge 0}  \sum_{Q \in \calG_j, Q \cap \textrm{supp}(b_{2, j+k_2}) \neq \emptyset}  \|b_{2, j+k_2}\|_{L^1(Q)}\\
&& \lesssim \textrm{Vol}(Q_0),  
\end{eqnarray*}
where in the last step, we use the disjointness of $\{b_{2, j+k_2} \one_Q\}$ for fixed $k_2$. 

\medskip

\textit{Case II: $\frac{1}{r}+\frac{1}{s'}=1$, that is $r=s$.}

\medskip

In this case, we need to show that
$$
\sum_{j \ge 0}  \sum_{Q \in \calG_j, Q \cap \textrm{supp}(b_{2, j+k_2}) \neq \emptyset} \left\|b_{1, j+k_1} \right\|_{L^r\left(\frac{\frakC_1Q}{\del}\right)}  \|b_{2, k_2+j}\|_{L^{r'}(Q)} \lesssim \textrm{Vol}(Q_0).  
$$
To see this, we apply H\"older's inequality to the left hand side of the above inequality to obtain 
\begin{eqnarray} \label{neweq100}
&&\sum_{j \ge 0}  \sum_{Q \in \calG_j, Q \cap \textrm{supp}(b_{2, j+k_2}) \neq \emptyset} \left\|b_{1, j+k_1} \right\|_{L^r\left(\frac{\frakC_1Q}{\del}\right)}   \|b_{2, k_2+j}\|_{L^{r'}(Q)} \nonumber \\
&& \le \left( \int_{\R^n} \sum_{j \ge 0}  \sum_{Q \in \calG_j, Q \cap \textrm{supp}(b_{2, j+k_2}) \neq \emptyset} |b_{1, j+k_1}(x)|^r  \one_{\frac{\frakC_1Q}{\del}} (x) dx\right)^{\frac{1}{r}} \nonumber \\
&& \quad \quad \quad \quad\quad \quad  \quad \quad \quad \quad   \cdot \left( \int_{\R^n} \sum_{j \ge 0} \sum_{Q \in \calG_j, Q \cap \textrm{supp}(b_{2, j+k_2}) \neq \emptyset} |b_{2, k_2+j}(x)|^{r'} \one_{Q}(x) dx\right)^{\frac{1}{r'}},
\end{eqnarray}
where for the second term in above estimation, using the fact that $\{|b_{2, k_2+j}|^{r'} \one_Q\}$ has disjoint supports, we can bound it as
$$
\left( \int_{\R^n} \sum_{j \ge 0} |b_{2, k_2+j}(x)|^{r'} dx \right)^{\frac{1}{r'}} \lesssim \textrm{Vol}(Q_0)^{\frac{1}{r'}},
$$ 
where in the above estimate, we use again the disjointness of Whitney cubes and Lemma \ref{modyeq02}. 

While for the first term, we note that for fixed $k_1$, the set of functions $\{|b_{1, j+k_1}|^r\}_{j \ge 0} $ has disjoint supports. However, for a fixed $j \ge 0$, the set 
$$
\left\{ |b_{1, j+k_1}|^r \one_{\frac{\frakC_1Q}{\del}} \right\}_{Q \in \calG_j, Q \cap \textrm{supp}(b_{2, j+k_2}) \neq \emptyset}
$$ 
may not, since $\frac{\frakC_1}{\del}>1$, the cubes in the set 
\begin{equation} \label{0104eq11}
\left\{\frac{\frakC_1Q}{\del}\right\}_{Q \in \calG_j, Q \cap \textrm{supp}(b_{2, j+k_2}) \neq \emptyset}
\end{equation}
may have some overlaps. 

\begin{lem} \label{0104lem11}
 Let $j \ge 0$. Then for any $x \in  \textrm{supp}(b_{1, j+k_1})$, there are only finitely many cubes in the set \eqref{0104eq11} containing $x$. Moreover, the number of such cubes are bounded by an absolute constant $\frakC_4>0$, which only depends 
on $\frakC_1, \del$ and any $2$-admissible constants. 
\end{lem}

\begin{proof}
By Theorem \ref{dyadicSHT}, there exists a unique dyadic cube $Q_x \in \calG_j$, such that $x \in Q_x$. Suppose $x \in \frac{\frakC_1Q}{\del}$ for some other $Q \in \calG_j$, then clearly this implies
\begin{equation} \label{neweq10}
\frac{\frakC_1Q}{\del} \bigcap \frac{\frakC_1Q_x}{\del} \neq \emptyset. 
\end{equation}

Thus, the number of choices of the cubes in the set \eqref{0104eq11} that contains $x$ can be bounded by the total number of the choices of $Q$ satisfying \eqref{neweq10}. More precisely, due to \eqref{neweq10} and the fact that $Q$ and $Q_x$ has the same sidelength $\del^j$, we can find an absolute constant $\frakC'_4>0$, which  depends 
on $\frakC_1, \del$ and any $2$-admissible constants,  such that
$$
Q \subseteq \frakC'_4 Q_x.
$$
The desired claim then follows from the fact that $Q$'s are disjoint and both $Q$ and $\frakC'_4 Q_x$ have the same sidelength $\del^j$. 
\end{proof}

Using Lemma \ref{0104lem11}, the first expression in  \eqref{neweq100} can be bounded by 
$$
\frakC_4\left( \int_{\R^n} \sum_{j \ge 0} |b_{1, j+k_1}(x)|^rdx \right)^{\frac{1}{r}}
$$
which clearly is bounded by $\textrm{Vol}(Q_0)^{\frac{1}{r}}$, up to some absolute constant. Combing the two estimations in \eqref{neweq100} above, we prove the claim in the second case. 

The proof is complete.


\begin{appendix}


\section{$(\calC_J)$ revisited}

In this appendix, we recall a ``uniform" version of the curvature condition $(\calC_J)$ from \cite{BS}. This will help us to keep track the dependence of the constants in the Theorem \ref{apdthm01} below, and therefore the constants in Lemma \ref{moduluscont}. 

Fix $\rho, \eta>0$. Let $\gamma: B^k(\rho) \times B^n(\eta) \rightarrow \R^n$ be a $C^\infty$ mapping satisfying $\gamma_0(x) \equiv x$. Recall from Section 3 that, we say $\gamma$ satisfies the condition $(\calC_J)$ at $0$ if there exists some multi-index $\beta$, such that
$$
\left| \left( \frac{\partial}{\partial \tau} \right)^\beta \det_{n \times n} \frac{\partial \Gamma}{\partial \tau} (0, \tau) \bigg |_{\tau=0} \right| \neq 0, 
$$
where $\tau=(t^1, \dots, t^n) \in \R^{kn}$ and 
$$
\Gamma(x, \tau)=\gamma_{t^1} \circ \gamma_{t^2} \circ \dots \circ \gamma_{t^n}(x).
$$
Now we wish to define the ``uniform" version of the condition $(\calC_J)$. Here, ``uniform" refers to the condition $(\calC_J)$ holds uniformly for a collection of $\gamma$'s. More precisely, we let $\calS$ be a set of $C^\infty$ functions $\gamma:  B^k(\rho) \times B^n(\eta) \rightarrow \R^n$ satisfying $\gamma_0(x) \equiv x$. We define the condition $(\calC_J^u)$ at $0$ as follows:

\begin{enumerate}

\item[$\bullet$] $(\calC_J^u)$: For $\gamma \in \calS$, there exists an $M \in \N$ and a $c>0$, both independent of $\gamma \in \calS$ such that for every $\gamma \in \calS$, there exists $\beta$, with $|\beta| \le M$, and
$$
\left | \left(\frac{\partial}{\partial \tau} \right)^\beta \det_{n \times n} \frac{\partial \Gamma}{\partial \tau} (0, \tau) \bigg|_{\tau=0} \right| \ge c. 
$$
\end{enumerate}

We need another ``uniform" version of curvature condition. Let $\gamma \in \calS$. Recall that the $C^\infty$ vector fields $\{X_\alpha\}$ is defined via the $C^\infty$ vector fields
$$
W(t, x)=\frac{\partial}{\partial \epsilon} \bigg|_{\epsilon=1} \gamma_{\epsilon t} \circ \gamma_{t}^{-1}(x),
$$
that is, if we write $W$ as a Taylor series in the $t$ variable, then 
$$
W(t) \sim \sum_\alpha t^\alpha X_\alpha. 
$$
The curvature condition $(\calC_Z^u)$ at $0$ is defined as

\begin{enumerate}
\item [$\bullet$] $(\calC_Z^u)$: For $\gamma \in \calS$, there exists $M \in \N$, indepedent of $\gamma \in \calS$, such that $\{X_\alpha: |\alpha| \le M\}$ satisfies the H\"ormander's condition at $0$, uniformly for $\gamma \in \calS$; more precisely, that there exists $M' \in \N, c>0$, independent of $\gamma \in \calS$ such that if we let $V_1, \dots, V_L$ denote the list of vector fields containing $\{X_\alpha: |\alpha| \le M\}$, along with all commutators of the vector fields in $\{X_\alpha: |\alpha| \le M\}$ up to order $M'$, then we have $|\det\limits_{n \times n} V(0) | \ge c$, where we have written $V$ to denote the matrix whose columns are $V_1, \dots, V_L$. 
\end{enumerate}

We have the following result.

\begin{thm}{\cite[Theorem 9.4]{BS}} \label{apdthm000}
Let $\calS$ be as above, and suppose $\calS \subset C^\infty(B^k(\rho) \times B^n(\eta); \R^n)$ is a bounded set. Then, $(\calC_J^u) \Leftrightarrow (\calC_Z^u)$.
\end{thm}

Indeed, we can restate the above result in a more quantitative way, which is easy to see by checking the proof carefully (or by a compactness argument). More precisely, we may write $(\calC_J^u)$ as
$$
(\calC_J)_{M_1, c_1, \rho_1, \eta_1, \{\sigma_1^m\}_{m \in \N}},
$$
and $(\calC_Z^u)$ as
$$
(\calC_Z)_{M_2, M'_2, c_2, \rho_2, \eta_2, \{\sigma_2^m\}_{m \in \N}}.
$$
Here $M_1 \in \N$ and $c_1>0$ are the parameters in the definition of $(\calC_J^u)$. Moreover, for those $\gamma$ satisfying $(\calC_J^u)_{M_1, c_1, \rho_1, \eta_1, \{\sigma_1^m\}_{m \in \N}}$, we have
$$
\gamma \in C^\infty( B^k(\rho_1) \times B^n(\eta_1); \R^n)
$$
and for each $m \ge 0$, 
$$
\|\gamma\|_{C^m(B^k(\rho_1) \times B^n(\eta_1); \R^n)} \le \sigma_1^m.
$$
The second condition is defined similarily. 

Then one can state Theorem \ref{apdthm000} as follows.

\begin{thm} \label{apdthm002}

$(\calC_J^u) \Leftrightarrow (\calC_Z^u)$ in the following sense: 

\begin{enumerate}
\item [$\bullet$] $(\calC_J)_{M_1, c_1, \rho_1, \eta_, \{\sigma_1^m\}_{m \in \N}} \Rightarrow$ there exists some postive integers
$$
M_2=M_2\left(M_1, c_1, \rho_1, \eta_, \{\sigma_1^m\}_{m \in \N}\right),
$$ 
$$
M_2'=M_2'\left(M_1, c_1, \rho_1, \eta_, \{\sigma_1^m\}_{m \in \N}\right)
$$
and a positive number 
$$
c_2=c_2\left(M_1, c_1, \rho_1, \eta_, \{\sigma_1^m\}_{m \in \N}\right), 
$$
such that $(\calC_Z)_{M_2, M_2', c_2, \rho_1, \eta_1,  \{\sigma_1^m\}_{m \in \N}}$ holds;

\item [$\bullet$] $(\calC_Z)_{M_2, M_2', c_2, \rho_2, \eta_2, \{\sigma_2^m\}_{m \in \N}} \Rightarrow$ there exists some positive integer
$$
M_1=M_1\left(M_2, M_2', c_2, \rho_2, \eta_2, \{\sigma_2^m\}_{m \in \N}\right)
$$
and a positive number
$$
c_1=c_1\left(M_2, M_2', c_2, \rho_2, \eta_2, \{\sigma_2^m\}_{m \in \N} \right),
$$
such that $(\calC_J)_{M_1, c_1, \rho_2, \eta_2, \{\sigma_2^m\}_{m \in \N}}$ holds. 
\end{enumerate}
\end{thm}

\bigskip

\section{A general $L^2$ theorem}

In this appendix, we first state a special case of a general $L^2$ theorem in \cite{BS},  then prove a slightly different version of it, to complete the proof of Lemma \ref{moduluscont}. Here, we say ``a special case", which refers to a ``single scale" and ``full rank" case, while in \cite{BS}, Street stated this result for an ``all scales" and ``small rank" case. We make a remark that the setting we present here is exactly the pullback version of Street's setting via the scaling map $\Phi$ (see, Theorem \ref{scalingmap}). 

The setting is as follows. Let $ K_0 \Subset \Omega'' \Subset \Omega' \Subset \Omega \subseteq \R^n$, with $\Omega$ open, $K_0$ compact, $\Omega''$, $\Omega'$ open and relatively compact in $\Omega$ and $a>0$ sufficiently small. Let
$$
Z_1, \dots, Z_r
$$
be a collection of $C^\infty$ vector fields defined on $\Omega$, with formal degree $\widetilde{d}_1, \dots \widetilde{d}_r$, and satisfying H\"ormander's condition of type $M$ on $\Omega$. Therefore, we can take a $M$-generated set (see Definition \ref{Mgenerate}),
$$
(Z, \tilde{d}):=\{(Z_1, \tilde{d}_1), \dots (Z_q, \tilde{d}_q) \},
$$
such that $n=\dim \ \textrm{span} (Z_1(x_0), \dots, Z_q(x_0))$ for any $x_0 \in \Omega$. 

Next, we turn to define the operator $S_j, 1 \le j \le L$ for some $L \in \N$. We assume, for each $j$, we are given a $C^\infty$ function $\widehat{\gamma}_j: B^{k_j}(\rho) \times \Omega'' \rightarrow \Omega'$ satisfying $\widehat{\gamma}_j(0, x) \equiv x$. We assume that for each $j$, there exists a $0<\rho_j \le \rho$, such that for $m \ge 0, m \in \Z$, there exists a $C_j(m)>0$, such that
$$
\left\| \widehat{\gamma}_j \right\|_{C^m\left(B^{k_j}(\rho_j) \times K_0 \right)}<C_j(m). 
$$
As usual, we restrict our attention to $\rho>0$ small, so that $\widehat{\gamma}_{j, t}^{-1}$ makes sense whenever we use it. We suppose we are given
$\varrho_j \in C^\infty(\overline{B^{k_j}(a)} \times \overline{\Omega''})$ and $\psi_{j, 1}, \psi_{j, 2} \in C^\infty_0(\R^n)$ supported on the interior of $K_0$. Finally, we suppose we are given $\chi_j \in C_0^\infty(B^{k_j}(a))$ for some $a>0$, $a$ sufficiently small. Now, we define
$$
S_jf(x):= \psi_{j, 1}(x) \int f(\widehat{\gamma}_{j, t}(x)) \psi_{j, 2}(\widehat{\gamma}_{j, t}(x)) \varrho_j(t, x) \chi_j(t)dt.
$$
Note that, under the above assumptions
$$
\|S_j\|_{L^\infty \mapsto L^\infty}, \|S_j\|_{L^1 \mapsto L^1} \lesssim 1,
$$
where the implict constants in the above operators norms depend only on the $L^\infty$-norms of $\psi_{j, 1}$ and $\psi_{j, 2}$, $L^\infty$-norm of $\varrho_j$, $L^1$-norm of $\chi_j$, $C^1$-norm of $\widehat{\gamma}$ (that is, the $C^1$-norm of each component of $\widehat{\gamma}$). Furthermore $S_j^*$ is of the same form as $S_j$ with $\widehat{\gamma}_{j, t}$ replaced by $\widehat{\gamma}_{j, t}^{-1}$.

We assume, further, that for each $l$, $1 \le l \le r$, there is a $j$ $(1 \le j \le L)$, and a multi-index $\alpha$ (with $|\alpha| \le B$, where $B \in \N$ is some fixed constant which our results are allowed to depend on), such that
\begin{equation} \label{generalnothold01}
Z_l(x)=\frac{1}{\alpha !} \frac{\partial}{\partial t}^\alpha \bigg|_{t=0} \frac{d}{d\epsilon} \bigg |_{\epsilon=1} \widehat{\gamma}_{j, \epsilon t} \circ \widehat{\gamma}_{j, t}^{-1}(x).
\end{equation}
This concludes our assumptions on $S_1, \dots, S_L$. 

Next, we turn to the operators $R_1$ and $R_2$. We assume we are given a $C^\infty$ function $\tilde{\gamma}_{t, s}$ (with $\tilde{\gamma}_{0, 0}(x) \equiv x$), with satisfying $\widetilde{\gamma}_{0, 0}(x) \equiv x$ and
$$
\|\widetilde{\gamma}\|_{C^m\left( B^{\tilde{k}}(\rho') \times [-1, 1] \times \Omega''\right)} \le C(m)
$$
for some $0<\rho' \le \rho$ and some $C(m)>0$, $m \ge 0$, where the choice of $\rho'$ is independent of $m$. 

We suppose we are given $\widetilde{\varrho}(t, s, x) \in C^\infty \left( \overline{B^{\tilde{k}}(a)} \times [-1, 1] \times \Omega'' \right), \widetilde{\chi}(t) \in L^1(B^{\tilde k}(a))$, and $\widetilde{\psi}_1, \widetilde{\psi}_2 \in C^\infty_0$ supported on the interior of $K_0$. We define, for $\zeta \in [-1, 1]$, 
$$
R^\zeta f(x):=\widetilde{\psi}_1(x) \int f \left( \widetilde{\gamma}_{t, \zeta}(x) \right) \widetilde{\psi}_2 \left(\widetilde{\gamma}_{t, \zeta}(x) \right) \widetilde{\varrho}(t, \zeta, x) \widetilde{\chi}(t)dt.
$$
Note that we have
$$
\|R^\zeta\|_{L^1 \mapsto L^1}, \|R^\zeta\|_{L^\infty \mapsto L^\infty} \lesssim 1,
$$
where the implict constants in the above operators norms depend only on the $L^\infty$-norms of $\widetilde{\psi}_1$ and $\widetilde{\psi}_2$, $L^\infty$-norm of $\widetilde{\varrho}$, $L^1$-norm of $\widetilde{\chi}$, $C^1$-norm of $\widetilde{\gamma}$. We set $R_1=R^\zeta$ and $R_2=R^0$. 

\begin{thm} \label{apdthm01}
In the above setup, if $a>0$ is chosen sufficiently small, we have $\|S_1 \dots S_L (R_1-R_2)\|_{L^2 \mapsto L^2} \le C\zeta^\eta$, for some $\eta>0$.
\end{thm}

\begin{rem}
As in Section 4, $a, C$ and $\eta$ can be chosen to depend on certain parameters, which are independent of other relevalent parameters; more precisely, they only depend on
\begin{enumerate}
 \item [(1).] the norm of the various functions used to define $S_j, R_1$ and $R_2$, that is, the $C^m$-norms of $\psi_{j, 1}, \psi_{j, 2}, \widetilde{\psi}_1, \widetilde{\psi}_2, \varrho_j, \widetilde{\varrho}, \widehat{\gamma}, \widetilde{\gamma}, \chi_j$ and $\widetilde{\chi}$;
\item [(2).] the parameter $B$;
\item [(3).] the parameter $M$;
\item [(4).] the parameter $L$;
\item [(5).] the various dimensions;
\item [(6).] the parameters $\rho_j$ and $\rho'$;
\item [(7).] all the parameters in the definition of $(\calC_Z)$ (see Theorem \ref{apdthm002}) \footnote{When we apply Theorem \ref{apdthm01} to prove Lemma \ref{moduluscont}, and hence to prove the main result in Section 5, there are indeed lots of ``$\Gamma$"s since we apply Lemma \ref{moduluscont} to lots of centers $x_c(Q)$ and all scales $\del^j, j \ge 0$. We make a remark that all these ``$\Gamma$"s satisfy the curvature condition
$$
(\calC_Z)_{\check{M}, \check{M'}, \check{c}, \check{\rho}, \check{\eta}, \{{\check{\sigma}}^m\}_{m \in \N}}
$$
uniformly (See, Remark \ref{apdthm002}), where $\check{M}, \check{M'} \in \N$ and $\{\check{\sigma}^m\}_{m \in \N}$ are some absolute constants only depending on the setting (See Section 5.1), and $\check{\rho}, \check{\eta}>0$ are some $2$-admissible constants. Therefore, when we apply Theorem \ref{apdthm01} to all these ``$\Gamma$", the implicit constant there can be choosen uniformly, indpedent of a particular ``$\Gamma$". }, when we apply it to the $C^\infty$ mapping $\Gamma$, uniformly for all $x \in \Omega$  (see \eqref{defnGAMMA01} for the definition of $\Gamma$, and hence $\Gamma$ satisfies $(\calC_J)$ with respect to these parameters). In particular, it depends on the lower bound of the term
$$
\left| \det_{n \times n} \left(Z_1(x), \dots, Z_q(x) \right) \right|,
$$
for all $x \in \Omega$.
\end{enumerate}
\end{rem}

The rest of this appendix is devoted to the proof of Theorem \ref{apdthm01}. We need some preparations first. 

\begin{defn}
For $0<\eta \le 1$, $L_\eta^1(\R^n)$ is the Banach space consisting of all funcitons $h \in L^1(\R^n)$ that satisfy
$$
\int_{\R^n} |h(y-z)-h(y)|dy \le A|z|^\eta, \quad \textrm{for all} \  z \in \R^n.
$$
The norm on $L_\eta^1$ is defined to be $\|h\|_{L^1}$ plus the smallest constant $A$ for which the above inequality holds. 
\end{defn}

\begin{prop}[{\cite[Proposition 7.2]{CNSW}}] \label{prop2.7}
Let $\Psi: \bar{B} \subset \R^d \mapsto \R^n$, where $d \ge n$ and $\Psi$ is $C^\infty$, $w(\varsigma)d\varsigma$ be a measure in $\R^d$ and $d\mu=\Psi_*(w dt)$ be the transported measure in $\R^n$; that is, $\mu$ is deifned by the integration formula
$$
\int_{\R^n} f(y)d\mu(y)=\int_{\bar{B}} f(\Psi(t)) w(t)dt. 
$$

Let further, $J$ be the determinant of some $n \times n$ sub-matrix of the Jacobian matrix $\partial \Psi/ \partial t$ of $\Psi$. Assume that for some $\alpha$, a multi-index of $d$ many entries, we have
$$
\partial_t^\alpha J(t) \neq 0 \quad \textrm{for every} \quad t \in \bar{B}.
$$
Then the transported measure $d\mu=\Psi_*(w dt)$ satisfies the following:
\begin{enumerate}
\item [1.] It is absolutely continuous with respect to the Lebesgue measure on $\R^n$; 
\item [2.] Its Radon-Nikodym derivative $h$ belongs to $L_\del^1$ for all $\del<(2|\alpha|)^{-1}$; 
\item [3.] The $L_\del^1$ norm of $h$ can be controlled in terms of the $C^{k+2}(\bar{B})$ norm of $\Psi$, a lower bounded for $\partial_t^\alpha J(t)$ in $\bar{B}$, the $C^1$ norm of $w$, and the numbers $\del$ and $|\alpha|$.
\end{enumerate}
\end{prop}

Fix constants $\tilde{C_1}, \tilde{C_2}<\infty$ and let $\tilde{\zeta} \in (0, 1]$. Consider a non-negative measure $\Xi$ on $\R^{2n}$ with the following properties:
\begin{enumerate}
\item [$\bullet$] $\textrm{supp} \Xi \subseteq \left\{ (y, z): |y|, |z| \le \tilde{C_1}, |y-z| \le \tilde{C_1} \tilde{\zeta} \right\}$;
\item [$\bullet$] There exists bounded, nonnegative, measurable functions $m_1, m_2$ such that for every $f \in C^0(\R^n)$,
$$
\int \int f(y)d\Xi(y, z)=\int f(y) m_1(y)dy \ \textrm{and} \ \int \int f(z) d\Xi(y, z)= \int f(z)m_2(z)dz,
$$
with $m_1(y), m_2(z) \le C_2$. 
\end{enumerate}

\begin{prop} [{\cite[Proposition 13.3]{BS}}] \label{prop2.8}
Suppose $h \in L^1_\eta(\R^n)$ and $\Xi$ is a measure as desired above. Then, there exists $\eta', A \in (0, \infty)$ such that 
$$
\int |h(y)-h(z)| d\Xi (y, z) \le A \tilde{\zeta}^{\eta'} \|h\|_{L^1_\eta},
$$
where $\eta'$ depends only on $\eta$ and $n$, and $A$ depends only on $\eta, n$, and upper bounds for the constant $\tilde{C_1}$, $\tilde{C_2}$. 
\end{prop}

\begin{proof}[Proof of Theorme \ref{apdthm01}]
Denote $S=S_1 \dots S_L$ and $R=R_1-R_2$. In what follows, $\eta>0$ will be a positive number that may change from line to line. We wish to show $\|SR\|_{L^2 \mapsto L^2} \lesssim \zeta^\eta$ for some $\eta>0$, and it suffices to show that
$$
\|R^*S^*SR\|_{L^2 \mapsto L^2} \lesssim \zeta^\eta.
$$
Since $\|R^*\|_{L^2 \mapsto L^2} \lesssim 1$, it suffices to show 
$$
\|S^*SR\|_{L^2 \mapsto L^2} \lesssim \zeta^\eta.
$$
Continuing in this manner, it suffices to show that
\begin{equation} \label{selfconeq01}
\|(S^*S)^{2^l} R\|_{L^2 \mapsto L^2} \lesssim \zeta^\eta, 
\end{equation}
for some $l>0$. Since $\|S\|_{L^2 \mapsto L^2}, \|S^*\|_{L^2 \mapsto L^2} \lesssim 1$ trivially, it suffices to show 
\begin{equation} \label{selfconeq02}
\|(S^*S)^n R\|_{L^2 \mapsto L^2} \lesssim \zeta^\eta,
\end{equation}
where we have just taken $l$ so large $2^l \ge n$ and applied \eqref{selfconeq01}. 

It is also easy to see that $\|R^*S^*SR\|_{L^1 \mapsto L^1} \lesssim 1$, and so interpolation shows that to prove \eqref{selfconeq02}, we need only show 
\begin{equation} \label{selfconeq03}
\|(S^*S)^nR\|_{L^\infty \mapsto L^\infty} \lesssim \zeta^\eta.
\end{equation}
Let $f$ be a bounded measurable function. Rephrasing \eqref{selfconeq03}, we wish to show
\begin{equation} \label{selfconeq04}
|(S^*S)^n Rf(x_0)| \lesssim \zeta^\eta \|f\|_{L^\infty}
\end{equation}
for every $x_0 \in K_0$. We now fix $x_0$ and prove \eqref{selfconeq04}. All implicit constants in what follows can be chosen to be independent of $x_0 \in K_0$. 

Define 
$$
\widehat{\gamma}_t(x)=\widehat{\gamma}_{(t_1, \dots, t_L, s_1, \dots, s_L)}(x):=\widehat{\gamma}_{L, t_L} \circ \widehat{\gamma}_{L-1, t_{L-1}} \circ \dots \circ \widehat{\gamma}_{1, t_1} \circ \widehat{\gamma}_{1, s_1}^{-1} \circ \widehat{\gamma}_{2, s_2}^{-1} \circ \dots \circ \widehat{\gamma}_{L, s_L}^{-1}(x). 
$$
Thus $\widehat{\gamma}$ is smooth, and $S^*S$ is given by 
$$
S^*Sf(x)=\psi_1(x) \int f \left (\widehat{\gamma}_t(x)\right) \psi_2\left (\widehat{\gamma}_t(x)\right) \varrho(t, x) \varsigma(t)dt,
$$
where $\varsigma \in C_0^1(B^N(a')), a'>0$ is a small number depending on $a$. (From here on out, $a'>0$ will be a small number (depending $a>0$) that may change from line to line), $N=\sum\limits_{j=1}^L 2k_j, \varrho \in C^\infty$, and $\psi_1, \psi_2 \in C^\infty_0$ are supported in the interior of $K_0$. Define, for $\tilde{\tau}=(t_1, \dots, t_n) \  (t_j \in B^N(a'))$, 
\begin{equation} \label{defnGAMMA01}
\Gamma_{\tilde{\tau}}(x)=\widehat{\gamma}_{t_1} \circ \widehat{\gamma}_{t_2} \circ \dots \widehat{\gamma}_{t_n}(x)
\end{equation}
so that 
$$
(S^*S)^n f(x)=\psi_1(x) \int f(\Gamma_{\tilde{t}}(x)) \psi_2(\Gamma_{\tilde{\tau}}(x)) \varrho(\tilde{\tau}, x) \varsigma(\tilde{\tau}) d\tilde{\tau},
$$
where the various funcitons have changed but are of the same basic form as before. Thus
\begin{eqnarray*}
(S^*S)^n Rf(x)%
&=& \psi_1(x) \int f\left(\tilde{\gamma}_{t, \zeta} \circ \Gamma_{\tilde{\tau}}(x) \right) \widetilde{\psi_2} \left(\tilde{\gamma}_{t, \zeta} \circ \Gamma_{\tilde{\tau}}(x) \right) \varrho(t, \zeta, \tilde{\tau}, x) \varsigma(\tilde{\tau}) \widetilde{\varsigma}(t) d\tilde{\tau}dt\\
&& -\psi_1(x) \int f\left(\tilde{\gamma}_{t, 0} \circ \Gamma_{\tilde{\tau}}(x) \right) \widetilde{\psi_2} \left(\gamma_{t, 0} \circ \Gamma_{\tilde{\tau}}(x) \right) \varrho(t, 0, \tilde{\tau}, x) \varsigma(\tilde{\tau}) \widetilde{\varsigma}(t) d\tilde{\tau}dt.
\end{eqnarray*}
Here, $\varrho$ is $C^\infty$ and $\psi_1$, $\widetilde{\psi_2}$ belongs to $C^\infty$. Now think of $x_0 \in K_0$ as fixed. We wish to establish \eqref{selfconeq04}. The dependence of $\rho$ on $x_0$ is unimportant, so we suppress it. Given a bounded measurable function $f$, we wish to study the integral
\begin{eqnarray} \label{extra01}
&&\calT(f):=\int f\left(\tilde{\gamma}_{t, \zeta} \circ \Gamma_{\tilde{\tau}}(x_0) \right)  \varrho(t, \zeta, \tilde{\tau}) \varsigma(\tilde{\tau}) \widetilde{\varsigma}(t) d\tilde{\tau}dt \nonumber \\
&& \quad \quad \quad \quad \quad \quad -\int f\left(\gamma_{t, 0} \circ \Gamma_{\tilde{\tau}}(x_0) \right)\varrho(t, 0, \tilde{\tau}) \varsigma(\tilde{\tau}) \widetilde{\varsigma}(t) d\tilde{\tau}dt.
\end{eqnarray}
Note that $\psi_1(x_0)\calT(g \widetilde{\psi_2})=(S^*S)^nRf(x_0)$. 

\bigskip

\textit{Claim: For $a>0$ sufficiently small, there exists $\eta>0$ and $M$ such that 
$$
|\calT(f)| \le M \zeta^\eta \sup_{x \in B_{(Z, \tilde{d})}(x_0, 1)} |f(x)|,
$$
where $a, \eta$ and $C$ may only depend on the parameters the constants of the same names were allowed to depend on in our early assumption.
}

Indeed, it suffices to consider only $\rho_0$ of the form $\rho(t, \zeta, \tilde{\tau})=\rho_1(\tilde{\tau}) \rho_{2}(t, \tilde{\zeta})$, since every $\rho$ may be written as a rapidly converging sum of such terms. 

Next, we note that $\tilde{\gamma}_{t, s} \circ \Gamma_{\tilde{\tau}}$ is $C^\infty$ in a small neighborhood of $x_0$, if $a>0$ is sufficiently small (and therefore $t, s$ and $\tilde{\tau}$ are sufficiently small), then $\tilde{\gamma}_{t, s} \circ \Gamma_{\tilde{\tau}} \in B_{(Z, \tilde{d})}(x_0, \zeta_1)$, for some $\zeta_1>0$ small.  Moreover, without the loss of generality, we may assume $x_0=0$ and $B_{(Z, \tilde{d})}(x_0, \zeta_1) \subseteq B^n(\eta_1)$, where $\eta_1>0$ only depends on $1$-admissible constants. Hence, $\calT$ only depends on the values of $f$ on $B^n(\eta_1)$.

Now we turn back to the proof of the claim. Note that, since $\rho_2$ is $C^\infty$, we have $\rho_2(t, \zeta)=\rho_2(t, 0)+O(\zeta)$ uniformly in $t$. Combining this with \eqref{extra01}, it is easy to see
\begin{eqnarray} \label{extra02}
&& \calT(f)= \int \left[ g\left(\tilde{\gamma}_{t, \zeta} \circ \Gamma_{\tilde{\tau}}(0) \right)- g\left(\tilde{\gamma}_{t, 0} \circ \Gamma_{\tilde{\tau}}(0) \right)\right] \rho_1(\tilde{\tau}) \rho_{2}(t, 0) \varsigma(\tilde{\tau}) \widetilde{\varsigma}(t) d\tilde{\tau}dt \nonumber \\
&& \quad \quad  \quad \quad \quad +O\left(|b| \sup_{x \in B^n(\eta_1)} |f(x)|\right):=\widehat{\calT}(f)+ O\left(|b| \sup_{x \in B^n(\eta_1)}  |f(x)| \right). 
\end{eqnarray}
Note that the error term in \eqref{extra02} is of the desired form. Thus, it suffces to bound $|\widehat{\calT}(f)|$. By Theorem \ref{apdthm000}, $\Gamma$ satisfies $(\calC_J)$ uniformly in any relevant parameters, i.e., there exists a multi-index $\beta$ (with $|\beta| \lesssim 1$) such that
\begin{equation} \label{apdmod001}
\left | \left( \frac{\partial}{\partial \tilde{\tau}} \right)^\beta \det_{n \times n} \frac{\partial \Gamma}{\partial \tilde{\tau}} (\tilde{\tau}, 0) \bigg |_{\tilde{\tau}=0} \right | \gtrsim 1,
\end{equation}
where the implicit constant in the above inequality depends only on $r, M, L, \rho_j, K_0$ and $C_j(m)$. Since $\|\Gamma\|_{C^{|\beta|+1}} \lesssim 1$, which follows from the fact that $\Gamma$ is $C^\infty$ (see, e.g., \cite[Proposition 12.3]{BS}), we see that if we take $a>0$ sufficiently small, 
$$
\left | \left( \frac{\partial}{\partial \tilde{\tau}} \right)^\beta \det_{n \times n} \frac{\partial \Gamma}{\partial \tilde{\tau}} (\tilde{\tau}, 0) \right | \gtrsim 1.
$$
for all $\tilde{\tau}$ in the support of $\varsigma$. 

Applying Proposition \ref{prop2.7}, with $\Psi(\tilde{\tau})=\Gamma_{\tilde{\tau}}(0)$ and $\psi(\tilde{\tau})=\rho_1(\tilde{\tau}) \varsigma(\tilde{\tau})$, we see that there exists $\eta \gtrsim 1$ and $h \in L^1_\eta(\R^n)$ (with $\|h\|_{L^1_\eta} \lesssim 1$) such that
$$
\widehat{\calT}(f)=\int \left[  f\left(\tilde{\gamma}_{t, \zeta}(u)\right)- f\left(\tilde{\gamma}_{t, 0}(u) \right)\right] h(u) \rho_2(t, 0) \widetilde{\varsigma}(t)dudt.
$$
Applying two changes of variables, we have
\begin{eqnarray*}
&&\widehat{\calT}(f)=\int f(v) \left( \det \frac{\partial \tilde{\gamma}_{t, \zeta}^{-1}}{\partial v}(v) \right) h\left( \tilde{\gamma}_{t, \zeta}^{-1}(v) \right) \rho_2(t, 0) \widetilde{\varsigma}(t)dvdt\\
&& \quad \quad  \quad  \quad  \quad -\int g(v) \left( \det \frac{\partial \tilde{\gamma}_{t, 0}^{-1}}{\partial v}(v) \right) h\left( \tilde{\gamma}_{t, 0}^{-1}(v) \right) \rho_2(t, 0) \widetilde{\varsigma}(t)dvdt.
\end{eqnarray*}
Using that $\tilde{\gamma}(t, s, u)$ is $C^\infty$, we have
$$
 \left( \det \frac{\partial \tilde{\gamma}_{t, \zeta}^{-1}}{\partial v}(v) \right)= \left( \det \frac{\partial \tilde{\gamma}_{t, 0}^{-1}}{\partial v}(v) \right)+O(\zeta),
$$
and therefore
\begin{eqnarray*}
&& \widehat{\calT}(f)=\int f(v) \left( \det \frac{\partial \tilde{\gamma}_{t, 0}^{-1}}{\partial v}(v) \right) \left( h\left( \tilde{\gamma}_{t, \zeta}^{-1}(v) \right) - h\left( \tilde{\gamma}_{t, 0}^{-1}(v) \right) \right) \rho_2(t, 0) \widetilde{\varsigma}(t)dvdt\\
&& \quad \quad \quad  \quad+ O(\zeta) \sup_{v \in B^n(\eta')} |f(v)|,
\end{eqnarray*}
where $\eta'>0$ is some number depending only on $C(m), C_j(m), \rho_j, \rho'$ and any $2$-admissible constants. Using $\int |\rho_2(t, 0) \widetilde{\varsigma}(t)dt| \lesssim 1$. we have
$$
|\widehat{\calT}(f)| \lesssim \left\{ \sup_{|t| \le a} \left[ \int_{B^n(\eta')} \left |h\left( \tilde{\gamma}_{t, \zeta}^{-1}(v) \right) - h\left( \tilde{\gamma}_{t, 0}^{-1}(v) \right) \right| dv\right]+ \zeta \right\}\sup_{v \in B^n(\eta')} |f(v)|.
$$
The proof will now be completed by showing, for every $|t| \le a$, 
\begin{equation} \label{extra03}
\int \left |h\left( \tilde{\gamma}_{t, \zeta}^{-1}(v) \right) - h\left( \tilde{\gamma}_{t, 0}^{-1}(v) \right) \right| dv \lesssim \zeta^\eta
\end{equation} 
for some $\eta>0$. Define a measure $\Xi$ by
$$
\int k(y, z) d\Xi(y, z)= \int_{B^n(\eta')} k\left(\tilde{\gamma}_{t, \zeta}^{-1}(v), \tilde{\gamma}_{t, 0}^{-1}(v)\right)dv,
$$
so that the left hand side of \eqref{extra03} becomes $\int |h(y)-h(z)| d\Xi(y, z)$. Since $\tilde{\gamma}_{t, s}^{-1}$ depends smoothly on $s$, $\left|\tilde{\gamma}_{t, \zeta}^{-1}(v)-\tilde{\gamma}_{t, 0}^{-1}(v) \right| \lesssim \zeta$. Hence $\Xi$ is supported on those $(y, z)$ such that $|y-z| \lesssim \zeta$. Applying Proposition \ref{prop2.8}, we see $\int |h(y)-h(z)|d\Xi(y, z) \lesssim \zeta^\eta$, for some $\eta \gtrsim 1$. This establishes \eqref{extra03} and completes the proof.
\end{proof}

Next, we state a modified version of Theorem \ref{apdthm01}, in which, the assumption \eqref{generalnothold01} does not hold. 

To state this result, we need some further assumptions based on the setting of Theorem \ref{apdthm01}. More precisely, we let $\Omega'''$ to be an open set satisfying 
$$
K_0 \Subset \Omega''' \Subset \Omega'', 
$$
with $\Omega'''$ being relatively compact in $\Omega''$. Morever, we assume, we are given a $C^\infty$ function $\theta: B^{k_0}(1) \times \Omega''' \mapsto \Omega''$ satisfying $\theta(0, x) \equiv x$. We also assume that there exists a $\rho''>0$, such that for any $m \ge 0, m \in \Z$, there exists some constant $D(m)>0$, such that
$$
\|\theta\|_{C^m\left( B^{k_0}(\rho'') \times K_0 \right)} \le D(m).
$$
Finally, we assume that for any $b \in B^{k_0}(1)$, the map $\theta_b(\cdot):=\theta(b, \cdot)$ has an inverse, which maps $\theta_b(\Omega''')$ back to $\Omega'''$. 

Take some $b \in B^{k_0}(1)$.  Now for each $j$, we define the operator
$$
\widetilde{S}_jf(x):=\psi_{j, 1}(x) \int f(\theta_b^{-1} \circ \widehat{\gamma}_{j, t} \circ \theta_b(x)) \psi_{j, 2} (\theta_b^{-1} \circ \widehat{\gamma}_{j, t} \circ \theta_b(x)) \varrho_j(t, x)\chi_j(t)dt.
$$
Note that, under the above assumptions, again, it holds that
$$
\|\widetilde{S}_j\|_{L^\infty \mapsto L^\infty}, \|\widetilde{S}_j\|_{L^1 \mapsto L^1} \lesssim 1,
$$
and furthermore $\widetilde{S}_j^*$ is of the same form as $\widetilde{S}_j$, with $\theta_b^{-1} \circ \widehat{\gamma}_{j, t} \circ \theta_b$ replaced by $\theta_b^{-1} \circ \widehat{\gamma}^{-1}_{j, t} \circ \theta_b$. 

\begin{cor} \label{apdcor01}
In the above setup, if $a, |b|>0$ are choosen sufficiently small, we have $\|\widetilde{S}_j \dots \widetilde{S}_L (R_1-R_2) \|_{L^2 \mapsto L^2} \le \widetilde{C} \zeta^\eta$, for some $\eta>0$. 
\end{cor}

\begin{proof}
The proof of this corollary follows closely from the one of Theorem \ref{apdthm01}, and here we will only mention those necessary modifications. 

\medskip

1. Recall that curve $\widehat{\gamma}_t=\widehat{\gamma}{(t_1, \dots, t_L, s_1, \dots, s_L)}, t \in \R^N$ defined in the proof of Theorem \ref{apdthm01}. 

\medskip

\textit{Claim: For $a$ and $|b|$ sufficiently small, there exists a $\tilde{\rho}>0$, such that for each $m \ge 0, m \in \Z$, there exists a $\widetilde{D}(m)>0$, such that
$$
\| \theta_b^{-1} \circ \widehat{\gamma}_t \circ \theta_b \|_{C^m\left(B^N(a) \times K_0 \right)}<\widetilde{D}(m).
$$
}

\medskip

This is clear, since all $\widehat{\gamma}$, $\theta$ and $\theta^{-1}$ are $C^\infty$. 

\medskip

2. First, we note that the assumption \eqref{generalnothold01} gives: for each $l$, $1 \le l \le r$, there is a $j$ $(1 \le j \le L)$, and a multi-index $\alpha$, such that
$$
\widetilde{Z_l}(x)=\frac{1}{\alpha !} \frac{\partial}{\partial t}^\alpha \bigg|_{t=0} \frac{d}{d\epsilon} \bigg|_{\epsilon=1} \theta_b^{-1}  \circ \widehat{\gamma}_{j, \epsilon t} \circ \widehat{\gamma}^{-1}_{j, t}  \circ \theta_b(x),
$$
where $\widetilde{Z_l}$ is the pullback of the vector field $Z_l$ via the mapping $\theta_b$ for $1 \le l \le r$.  Therefore, we cannot apply Theorem \ref{apdthm01} directly and a modification is expected. 

\medskip

\textit{Claim: Under the setting of Corollary \ref{apdcor01}, the proof of Theorem \ref{apdthm01} still applies.}

\medskip

Indeed, by checking the proof of Theorem \ref{apdthm01} carefully, one can see that the whole point to make the assumption \eqref{generalnothold01} is to guarantee the mapping 
$$
\Gamma_{\widetilde{\tau}}=\widehat{\gamma}_{t_1} \circ \widehat{\gamma}_{t_2} \circ \dots \circ \widehat{\gamma}_{t_n}
$$ 
satisfies \eqref{apdmod001}, that is, the condition $(\calC_J)$. Therefore, our goal is to show that the estimation \eqref{apdmod001} still holds for the mapping 
\begin{equation} \label{pushbackcurve}
\theta_b^{-1} \circ \Gamma_{\widetilde{\tau}} \circ \theta_b. 
\end{equation}

Note that by \eqref{apdmod001}, there exists a multi-index $\beta$, such that 
$$
\left| \left(\frac{\partial}{\partial \widetilde{\tau}} \right)^\beta \det_{n \times n} \frac{\partial \Gamma}{\partial \widetilde{\tau}}(\widetilde{\tau}, 0) \Bigg|_{\widetilde{\tau}=0} \right| \gtrsim 1.
$$
Fix such a multi-index $\beta$. We wish to show that for $|b|$ sufficiently small, one has
$$
\left| \left(\frac{\partial}{\partial \widetilde{\tau}} \right)^\beta \det_{n \times n} \frac{\partial \left( \theta_b^{-1} \circ \Gamma \circ \theta_b \right)}{\partial \widetilde{\tau}}(\widetilde{\tau}, 0) \Bigg|_{\widetilde{\tau}=0} \right| \gtrsim 1.
$$
Indeed, by chain rule, we have
$$
 \frac{\partial \left( \theta_b^{-1} \circ \Gamma \circ \theta_b \right)}{\partial \widetilde{\tau}}(\widetilde{\tau}, 0)= \frac{\partial{\theta_b^{-1}}}{\partial w}  \left(\Gamma_{\widetilde{\tau}} \circ \theta_b(0) \right) \cdot \frac{\partial \Gamma}{\partial \widetilde{\tau}} \left(\widetilde{\tau}, \theta_b(0)\right),
$$
which is an $n \times Nn$ matrix and whose $(i, j)$-th entry is of the form 
$$
\sum_{l=1}^{\tilde{n}} \frac{\partial \left(\theta_b^{-1} \right)^i }{\partial w_l} \left(\Gamma_{\widetilde{\tau}} \circ \theta_b(0) \right) \cdot \frac{\partial \Gamma^l}{\partial \widetilde{\tau}_j} \left(\widetilde{\tau}, \theta_b(0)\right),
$$
for $1 \le i \le n$ and $1 \le j \le Nn$. Here, we write $w=(w_1, \dots, w_n) \in \R^n$, $\theta^{-1}_b=\left( \left(\theta_b^{-1} \right)^1, \dots, \left(\theta_b^{-1} \right)^n \right)$ and $\Gamma=\left( \Gamma^1, \dots, \Gamma^n \right)$. Note that when $b=0$, we get the $(i, j)$-th entry of the matrix $\frac{\partial \Gamma}{\partial \widetilde{\tau}}(\widetilde{\tau}, 0)$. Therefore, when the operator 
$$
\left( \frac{\partial}{\partial \widetilde{\tau}} \right)^\beta
$$
acts on each of these $(i, j)$-th entries, we see that the derivative is a linear combination of the following two expressions:
\begin{equation} \label{Q2eq002}
\frac{\partial \left(\theta_b^{-1} \right)^i }{\partial w_l} \left(\Gamma_{\widetilde{\tau}} \circ \theta_b(0) \right) \cdot \left( \frac{\partial}{\partial \widetilde{\tau}} \right)^\beta \left(\frac{\partial \Gamma^l}{\partial \widetilde{\tau}_j} \left(\widetilde{\tau}, \theta_b(0)\right) \right)
\end{equation}
and 
\begin{equation} \label{Q2eq003}
\left( \frac{\partial}{\partial \widetilde{\tau}} \right)^{\beta_1}\left(\frac{\partial \left(\theta_b^{-1} \right)^i }{\partial w_l} \left(\Gamma_{\widetilde{\tau}} \circ \theta_b(0) \right) \right) \cdot \left( \frac{\partial}{\partial \widetilde{\tau}} \right)^{\beta_2} \left(\frac{\partial \Gamma^l}{\partial \widetilde{\tau}_j} \left(\widetilde{\tau}, \theta_b(0)\right) \right),
\end{equation}
where $\beta_1+\beta_2=\beta$ and $|\beta_1|>0$. Since $\theta_t^{-1}$ is smooth, $\theta^{-1}_0(w)=w$ and $\Gamma(0, w)=w$, we see that for each $1 \le i, l \le n$,  

$$
\frac{\partial \left(\theta_b^{-1} \right)^i }{\partial w_l} \left(\Gamma_{\widetilde{\tau}} \circ \theta_b(0) \right) \longrightarrow
\begin{cases}
1, \quad   \hfill i=l; \\
0, \quad   \hfill i \neq l, 
\end{cases}
\quad \textrm{as} \quad a, |b| \to 0,
$$
and for any multi-index $\alpha' \in \N^n$ with $|\alpha'| \ge 2$, 
$$
\left(\left(\frac{\partial}{\partial w}\right)^{\alpha'}\left(\theta_b^{-1} \right)^i  \right)  \left(\Gamma_{\widetilde{\tau}} \circ \theta_b(0) \right) \longrightarrow 0, \quad \textrm{as} \quad a, |b| \to 0.
$$
 These, together with another application of chain rule,  and the fact that both $\Gamma_{\widetilde{\tau}}$ and $\theta_t$ are $C^\infty$ in any relevant parameters, implies 

$$
\eqref{Q2eq002} \to
\begin{cases}
\left(\left( \frac{\partial}{\partial \widetilde{\tau}} \right)^{\beta} \frac{\partial \Gamma^l}{\partial \widetilde{\tau}_j} \right) (\widetilde{\tau}, 0) \bigg|_{\widetilde{\tau}=0}, \quad   \hfill i=l; \\\\
0, \quad   \hfill i \neq l, 
\end{cases}
\quad \textrm{as} \quad a, |b| \to 0
$$

and

$$
\eqref{Q2eq003} \to 0, \quad \textrm{as} \quad a, |b| \to 0，
$$
Therefore, by \eqref{apdmod001}, we see that for $a$ and $|b|$ sufficiently small, 
$$
\left| \left(\frac{\partial}{\partial \widetilde{\tau}} \right)^\beta \det_{n \times n} \frac{\partial \left( \theta_b^{-1} \circ \Gamma \circ \theta_b \right)}{\partial \widetilde{\tau}}(\widetilde{\tau}, 0) \Bigg|_{\widetilde{\tau}=0} \right| \gtrsim \left| \left(\frac{\partial}{\partial \widetilde{\tau}} \right)^\beta \det_{n \times n} \frac{\partial \Gamma}{\partial \widetilde{\tau}}(\widetilde{\tau}, 0) \bigg |_{\widetilde{\tau}=0}  \right| \gtrsim 1.
$$

\end{proof}

\section{A quantitative $L^p$ improving theorem}

In this appendix, we recall a quantitative $L^p$ improving theorem for the singular Radon transform, which was a special case\footnote{See Appendix B for a similar comment.} of the one proved in \cite{BS3}. Again, ``a special case" refers to a ``single scale" and ``full rank" case, while in \cite{BS3}, Street studied this result for an ``all scales" and "small rank" case. Similarily, this ``special case" we are interested here comes from a pullback of Street's setting via the scaling map $\Phi$ (see, Theorem \ref{scalingmap}). 

The setting is as follows. Let $a>0$ be a sufficiently small number, $n, k \in \N, n, k \ge 1$, and 
$$
0 \in \Omega_1 \subsetneq B^n(\eta) \subset \R^n,
$$
where $\eta>0$ and $\Omega_1$ is some open set in $\R^n$. We start by defining four bounded sets of smoothing mappings. 

\begin{enumerate}
\item [(a).] $\calB_1 \subseteq C_0^\infty(\R^k)$ is a bounded set, with all $\chi \in \calB_1$ being supported in $\overline{B^k(a)}$. More precisely, there exists a sequence of positive numbers $\{C_{1, m}\}_{m \in \N}$, such that
$$
\calB_1:=\left\{ \chi \in C_0^\infty (\R^k)): \textrm{supp}(\chi) \subseteq \overline{B^k(a)},  \| \chi\|_{C^m_0(\R^k)} \le C_{1, m}, m \ge 0\right\}. 
$$

\bigskip

\item [(b).] $\calB_2 \subseteq C_0^\infty (\R^n)$ is a bounded set, with all $\psi \in \calB_2$ being supported in $\overline{\Omega_1}$. More precisely, there exists a sequence of positive numbers $\{C_{2, m}\}_{m \in \N}$, such that
$$
\calB_2:=\left\{ \psi \in C^\infty_0(\R^n): \textrm{supp}(\psi) \subseteq \overline{\Omega_1}, \|\psi\|_{C^m_0(\R^n)} \le C_{2, m}, m \ge 0 \right\}. 
$$

\bigskip

\item [(c).]  $\calB_3$ is a bounded set of $C^\infty\left( \overline{B^k(a)} \times \overline{B^n(\eta)} \right)$. More precisely, there exists a sequence of positive numbers $\{ C_{3, m} \}_{m \in \N}$, such that 
\begin{eqnarray*}
\calB_3%
&:=& \bigg\{ \rho \in C^\infty\left( \overline{B^k(a)} \times \overline{\B^k(\eta)} \right): \\
&&  \quad \quad \quad \quad \quad \quad \quad \|\rho\|_{C^m\left( \overline{B^k(a)} \times \overline{B^n(\eta)} \right)} \le C_{3, m}, m \ge 0 \bigg\}.
\end{eqnarray*}

\bigskip

\item [(d).] Fix $a' \ge a, \eta' \ge \eta$ and $C_{4, m}$ is a sequense of positive numbers. Let
\begin{eqnarray*}
\calB_4%
&=& \bigg \{ \gamma \in C^\infty(\overline{B^k(a)} \times \overline{B^n(\eta)}; \R^n): \gamma(0, x) \equiv x \ \textrm{and} \\
&& \quad \quad \quad \quad \quad \quad \quad \gamma \ \textrm{satisfies} \ (\calC_J)_{M, c, a', \eta', \{C_{4, m}\}_{m \in \N}} \bigg \}.
\end{eqnarray*}
Note that this implies $\calB_4 \subseteq C^\infty(\overline{B^k(a)} \times \overline{B^n(\eta)}; \R^n)$ is a bounded set. 
\end{enumerate}

Define the Radon transform 
$$
\calT f(x):= \psi_1(x) \int f(\gamma_t(x)) \psi_2(\gamma_t(x)) \rho(t, x) \chi(t)dt, 
$$
where $\chi \in \calB_1$, $\psi_1, \psi_2 \in \calB_2$, $\rho \in \calB_3$ and $\gamma \in \calB_4$. 

\begin{thm} \label{apdthm003}
Given $r \in [1, \infty)$, there exists $s>r$, such that 
$$
\sup_{\chi \in \calB_1, \psi_1, \psi_2 \in \calB_2, \rho \in \calB_3, \gamma \in \calB_4}\| \calT f \|_{L^s} \le C_{r, s} \|f\|_{L^r},
$$
where the constant $C_{r, s}$ only depends on $r$, $s$ and the sets $\calB_1, \dots, \calB_4$. 
\end{thm}

\end{appendix}

\bibliographystyle{plain}
\bibliography{bibfile}

\end{document}